\documentclass[sn-mathphys,Numbered]{sn-jnl}% Math and Physical Sciences Reference Style
\usepackage{graphicx}%
\usepackage{multirow}%
\usepackage{amsmath,amssymb,amsfonts}%
\usepackage{amsthm}%
\usepackage{mathrsfs}%
\usepackage[title]{appendix}%
\usepackage{xcolor}%
\usepackage{textcomp}%
\usepackage{manyfoot}%
\usepackage{booktabs}%
\usepackage{algorithm}%
\usepackage{algorithmicx}%
\usepackage{algpseudocode}%
\usepackage{listings}%
\usepackage{hyperref} 
\usepackage{latexsym}
\usepackage{esint}
\usepackage{mathtools} 
\usepackage{wrapfig}
\usepackage{bbm}
\usepackage{color}
\usepackage{bm} 
\usepackage{todonotes}
\usepackage{stackengine}
\allowdisplaybreaks

\newcommand\ocirc[1]{\ensurestackMath{\stackon[1pt]{#1}{\mkern2mu\circ}}}

\newcommand{\weakto}{\rightharpoonup}

\newcommand{\test}{{\mathscr F^\eta}}
\newcommand{\testE}{{\mathscr F^{(\eta^\epsilon)_\epsilon} }}

\theoremstyle{thmstyleone}%
\newtheorem{theorem}{Theorem} 
\newtheorem{lemma}[theorem]{Lemma}
\newtheorem{proposition}[theorem]{Proposition}
\newtheorem{corollary}[theorem]{Corollary}

\numberwithin{equation}{section}

\theoremstyle{thmstyletwo}%
\newtheorem{remark}{Remark}%

\theoremstyle{thmstylethree}%
\newtheorem{definition}{Definition}%

\newcommand{\bx}{\mathbf{x}}
\newcommand{\by}{\mathbf{y}}

\newcommand{\bu}{\mathbf{u}}
\newcommand{\bfu}{\mathbf{u}}
\newcommand{\bfv}{\mathbf{v}}

\newcommand{\bv}{\mathbf{v}}
\newcommand{\bn}{\mathbf{n}}
\newcommand{\be}{\mathbf{e}}
\newcommand{\bsigma}{\boldsymbol{\varkappa}}

\newcommand{\bfPsi}{\boldsymbol{\Psi}}
\newcommand{\bfvarphi}{\boldsymbol{\varphi}}
\newcommand{\bfphi}{\boldsymbol{\phi}}

\newcommand{\dy}{\, \mathrm{d}\mathbf{y}}
\newcommand{\dd}{\,\mathrm{d}}

\newcommand{\dx}{\, \mathrm{d} \mathbf{x}}

\newcommand{\dt}{\, \mathrm{d}t}
\newcommand{\ds}{\, \mathrm{d}s}

\newcommand{\D}{\, \mathrm{d}}

\newcommand{\dxt}{\,\mathrm{d}\mathbf{x}\, \mathrm{d}t}

\newcommand{\Div}{\mathrm{div}_{\mathbf{x}}}
\newcommand{\divx}{\mathrm{div}_{\mathbf{x}}}
\newcommand{\divy}{\mathrm{div}_{\mathbf{y}}}

\newcommand{\nabx}{\nabla_{\mathbf{x}}}
\newcommand{\naby}{\nabla_{\mathbf{y}}}

\newcommand{\Delx}{\Delta_{\mathbf{x}}}
\newcommand{\Dely}{\Delta_{\mathbf{y}}}

\newcommand{\R}{\mathbb{R}}

\newcommand{\mt}{\Gamma}

\newcommand{\Oeta}{\mathcal{O}_{\eta}}
\begin{document}

\title[Martingale solutions in stochastic fluid-structure interaction]{Martingale solutions in stochastic fluid-structure interaction}

%%=============================================================%%
%% Prefix	-> \pfx{Dr}
%% GivenName	-> \fnm{Joergen W.}
%% Particle	-> \spfx{van der} -> surname prefix
%% FamilyName	-> \sur{Ploeg}
%% Suffix	-> \sfx{IV}
%% NatureName	-> \tanm{Poet Laureate} -> Title after name
%% Degrees	-> \dgr{MSc, PhD}
%% \author*[1,2]{\pfx{Dr} \fnm{Joergen W.} \spfx{van der} \sur{Ploeg} \sfx{IV} \tanm{Poet Laureate} 
%%                 \dgr{MSc, PhD}}\email{iauthor@gmail.com}
%%=============================================================%%

\author*[1]{\fnm{Dominic} \sur{Breit}}\email{dominic.breit@tu-clausthal.de}
\equalcont{These authors contributed equally to this work.}

\author[1]{\fnm{Prince Romeo} \sur{Mensah}}\email{prince.romeo.mensah@tu-clausthal.de}
\equalcont{These authors contributed equally to this work.}

\author[1]{\fnm{Thamsanqa Castern} \sur{Moyo}}\email{thamsanqa.castern.moyo@tu-clausthal.de}
\equalcont{These authors contributed equally to this work.}

\affil*[1]{\orgdiv{Institute of Mathematics}, \orgname{TU Clausthal}, \orgaddress{\street{Erzstra\ss e 1}, \city{Clausthal-Zellerfeld}, \postcode{38678}, \country{Germany}}}

%\affil[2]{\orgdiv{Department}, \orgname{Organization}, \orgaddress{\street{Street}, \city{City}, \postcode{10587}, \state{State}, \country{Country}}}
%
%\affil[3]{\orgdiv{Department}, \orgname{Organization}, \orgaddress{\street{Street}, \city{City}, \postcode{610101}, \state{State}, \country{Country}}}

%%==================================%%
%% sample for unstructured abstract %%
%%==================================%%

\abstract{We consider a viscous incompressible fluid interacting with a linearly elastic  shell of Koiter type which is located at some part of the boundary. Recently models with stochastic perturbation in the shell equation have been proposed in the literature but only analysed in simplified cases. We investigate the full model with transport noise, where (a part of) the boundary of the fluid domain is randomly moving in time. We prove the existence of a weak martingale solution to the underlying system.}


%%\pacs[JEL Classification]{D8, H51}


\maketitle

\section{Introduction}
\label{sec:solveSolventStructure}
The mathematical analysis of systems of partial differential equations arising from fluid-structure interaction has seen a vast progress in the last two decades. This is motivated by a variety of applications, for instance in biomechanics \cite{fluid2014bodnar}, hydro-dynamics \cite{chakrabarti2002theory}, aero-elasticity \cite{dowell2015modern}
 and hemo-mechanics \cite{MR1873203}. 

\subsection{Deterministic models}
We are interested in the case where a viscous incompressible fluid
interacts with an elastic structure located at a part of the boundary of the fluid's domain $\mathcal O\subset\mathbb R^3$ denoted by $\Gamma$. 
 The structure reacts to the forces imposed by the fluid at the boundary. Assuming that this deformation only acts in the normal direction and denoting by $\bn$ the outer unit normal at the reference domain, $\mathcal O$ is deformed to the domain
$\mathcal O_{\eta(t)}$ defined through its boundary
\begin{align}\label{eq:eta}
\partial\mathcal O_{\eta(t)}:=\{\by+\eta(t,\by)\bn(\by):\,\by\in \Gamma\}.
\end{align} 
Here $\eta:(t,\by):  I\times\Gamma\mapsto \eta(t,\by)\in \mathbb{R}$ describes the deformation of the structure and $I:=(0,T)$, for some $T>0$ denotes a time interval. For technical simplification we will suppose that $\Gamma$ is the whole boundary and identify it with the two-dimensional torus (the precise geometric set-up is presented in Section \ref{sec:geometry}).

As a prototype let us consider the following problem where the equation for the shell can be seen as a linearised version of Koiter's model (neglecting lower order terms for simplicity and setting all positive physical constants to 1). In the unknowns
\begin{align*}
&\bu:(t,\bx): I\times\Oeta\mapsto \bu(t,\bx)\in \mathbb{R}^3,
\\
&\pi:(t,\bx): I\times\Oeta\mapsto \pi(t,\bx)\in \mathbb{R},
\end{align*}
accounting for the fluid's velocity field and pressure, respectively (defined on a moving space-time cylinder % defined above in \eqref{eq:eta}
), it reads as (for simplicity we neglect volume forces in the fluid equations)
\begin{align}
\divx \bu=0, 
\label{contEqA}
\\
 \partial_t \bu  + (\mathbf{u}\cdot \nabx)\mathbf{u}  
= 
\Delx \bu -\nabx\pi,
\label{momEqA}
\\
\partial_t^2\eta +\Dely^2\eta =-\bn^\top(\mathbb{T}(\bu, \pi )\bn_\eta)\circ\bm{\varphi}_\eta  \vert\det(\naby\bm{\varphi}_\eta)\vert.
\label{shellEqA}
\end{align}
The system is complemented by the kinetic boundary condition
\begin{align}
\label{noSlipA}
&\bu\circ\bm{\varphi}_\eta   =\bn\partial_t\eta
&\quad \text{on } I \times \Gamma  
\end{align}
at the fluid-structure interface  as well as initial conditions for \eqref{momEqA}-\eqref{shellEqA} and periodic boundary conditions for \eqref{shellEqA}.
Here $\mathbb{T}(\bu, \pi )=(\nabx\bu+\nabx\bu^\top)-\pi\mathbb I_{3\times 3}$
is the stress tensor of the fluid. The vectors $\bn$ and $\bn_\eta$ denote the normal vectors on $\mathcal O$ and $\mathcal O_\eta$, respectively. The function $\bm{\varphi}_\eta $ gives the coordinate transform from $\Gamma \rightarrow\partial\mathcal O_\eta$.
The existence of a weak solution to \eqref{contEqA}--\eqref{noSlipA} has been shown in \cite{lengeler2014weak} (see also \cite{muha2013existence} for the case of a cylindrical shell model). It satisfies the energy balance
%the energy inequality holds in the sense that
\begin{align}\label{eq:enA}
\begin{aligned}
\tfrac{1}{2}\int_{\mathcal O_\eta}|\bu(t)|^2\dx
&+
\int_0^t\int_{\mathcal O_\eta}|\nabx\bu|^2\dx\ds+\tfrac{1}{2}\int_\Gamma|\partial_t\eta(t)|^2\dy
+
\tfrac{1}{2}\int_\Gamma|\Dely\eta(t)|^2\dy\\
&\leq 
\tfrac{1}{2}\int_{\mathcal O_{\eta(0)}}|\bu(0)|^2\dx+\tfrac{1}{2}\int_\Gamma|\partial_t\eta(0)|^2\dy
+
\tfrac{1}{2}\int_\Gamma|\Dely\eta(0)|^2\dy
\end{aligned}
\end{align}
for a.a. $t\in I$, from which one can easily deduce the function spaces in which the weak solution lives.
 The easier case of an elastic plate (where the reference geometry is flat) has been studied before in \cite{grandmont2008existence}. The main advancement in \cite{lengeler2014weak} is a new compactness method which eventually allows to establish compactness of the velocity field. On account of the deformed space-time cylinder on which the problem is posed it is impossible  to apply the standard Aubin-Lions compactness lemma in order to pass to the limit in the convective term of approximate solutions.
Interestingly, this issue is ultimately linked to the divergence-free constraint \eqref{contEqA}. Without it, the compactness can simply be localised thus completely removing the difficulties posed by the moving boundary, see \cite{breitSchw2018compressible} where the compressible Navier--Stokes equations are studied.
 In \cite{MuSc} (where even the fully nonlinear Koiter model is considered) the compactness argument from \cite{lengeler2014weak} has been replaced by an abstract compactness criterion which is more in the spirit of the classical Aubin-Lions result and thus allows for wider applications. Let us finally remark that all the results just mentioned hold under the assumption that there is no self-intersection of the structure (which can always be avoided if $\|\eta\|_{L^\infty_\by}$ is not too large).

\subsection{Stochastic models}
It was recently suggested in
\cite{KuCa1} to consider a stochastic perturbation
in \eqref{shellEqA} to account for random effects in real-life problems and uncertainty in the data. A first step towards a well-posedness theory for such stochastic fluid-structure interaction models is done in \cite{KuCa2},
where the 2D time-dependent Stokes equations are linearly coupled to a structure described by a stochastic 1D wave equation. Although this is only a simplified model (and the boundary is not moving in time) the analysis is already quite advanced. 
As already indicated above, the geometry breaks down
if $\eta$ causes a self-intersection of the domain. In the simplified case, where the reference domain is a box and the deformation only occurs in the vertical direction, this happens exactly when the value of $-\eta$ coincides with the height of the box. If $\eta$ has a Gaussian distribution as in \cite{KuCa2}, this can always happen (though maybe only with a low probability) no matter how short the time horizon is.
%\footnote{We are aware that 
This issue may be circumvented by studying the local-in-time well-posedness of the problem which is done in \cite{KuCa3}. The authors of \cite{KuCa3} study the interaction of an elastic plate (the reference geometry is flat) with the 2D Navier--Stokes equation. The existence time is a random variable about which the only available information is $\mathbb P$-a.s. positivity.

In this paper we aim for the natural next step by considering the full model \eqref{contEqA}--\eqref{noSlipA} globally in time, where \eqref{shellEqA} is subject to some Gaussian noise.\footnote{One can very well add a suitable stochastic term in \eqref{momEqA}. This is covered by our analysis as long as it is energy conservative.}
We take a different perspective to \cite{KuCa1, KuCa2, KuCa3} and do not consider stochasticity entering as an external force but as an intrinsic property of the system. Thus we consider transport noise in the shell equation (see \eqref{shellEq} below). It has the very appealing feature of being energy conservative. If the initial data is deterministic (or simply bounded in probability) we have a pathwise control over the energy and the restrictions on the time interval are the same as in previous deterministic papers such as \cite{grandmont2008existence,lengeler2014weak,MuSc}. Transport noise has a clear physical meaning in fluid mechanical transport processes, see \cite{Holm2,Holm3,Holm1} as well as  \cite{CCR,Fl1,Fl2}. Depending on the particular structure, it can be conservative with respect to several important quantities such as energy, enstrophy and circulation. Note that this is excluded in the case of an It\^{o} noise.
 Also, it has been observed that transport noise has regularising effects on certain ill-posed PDEs, see \cite{Fl0,Fl3}.
Nevertheless, the role of transport noise for elastic materials must be further explored. Understanding the role of noise in the shell equation \eqref{shellEq} is motivated by \cite{KuCa1, KuCa2, KuCa3} and its understanding is only at the beginning stages.

Our goal  is to construct on  random space-time cylinders $\Omega\times I\times\Oeta$ and $\Omega\times I\times\Gamma$, a global weak solution triple (note that the pressure does not enter the weak formulation)
\begin{align*}
&\bu:(\omega,t,\bx): \Omega\times I\times\Oeta\mapsto \bu(\omega,t,\bx)\in \mathbb{R}^3,
\\
&\pi:(\omega,t,\bx): \Omega\times I\times\Oeta\mapsto \pi(\omega,t,\bx)\in \mathbb{R},
\\
&\eta:(\omega,t,\by): \Omega\times I\times\Gamma\mapsto \eta(\omega,t,\by)\in \mathbb{R},
\end{align*}
representing the fluid's velocity, the fluid's pressure and the structure displacement 
of the coupled fluid-structure system given by 
\begin{align}
\divx \bu=0, 
\label{contEq}
\\
 \partial_t \bu  + (\mathbf{u}\cdot \nabx)\mathbf{u}  
= 
\Delx \bu -\nabx\pi,
\label{momEq}
\\
\dd \partial_t\eta +(\Dely^2\eta +g_\eta)\dt 
+
((\bsigma\cdot\naby)\partial_t \eta )\circ
\dd B_t=0,
\label{shellEq}
\end{align}
with $g_\eta=-\bn^\top(\mathbb{T}(\bu, \pi )\bn_\eta)\circ\bm{\varphi}_\eta  \vert\det(\naby\bm{\varphi}_\eta)\vert$.
Here, $\Omega$ is a sample space of a filtered probability space $(\Omega,\mathfrak{F},(\mathfrak{F}_t)_{t\geq0},\mathbb{P})$ with associated expectation $\mathbb{E}(\cdot)$.
%, $I:=(0,T)$ is a time interval with $T>0$ and $\Oeta$ is the deformation of a spatial reference domain $\mathcal{O}$ with respect to a coordinate transform $\bm{\varphi}_{\eta}$ (see Section \ref{sec:geometry}). 
%The boundary of $\mathcal{O}$  is identified with $\Gamma$ and for simplicity, we impose periodic boundary conditions on $\Gamma$. 
%\eqref{momEq}, $\pi:(\omega,t,\bx): \Omega\times I\times\Oeta\mapsto \pi(\omega,t,\bx)\in \mathbb{R}$ represents the pressure in the fluid and its role is to enforce the incompressibility condition \eqref{contEq}. 
Equation  \eqref{shellEq} contains a Stratonovich differential of a real-valued Brownian motion $(B_t)$
and $\bsigma$ is a given solenoidal (incompressible) vector field (i.e. $\divy \bsigma=0$)\footnote{One can allow a (possibly infinite) sum of stochastic transport terms in \eqref{shellEq} without affecting the analysis.}  in $\mathbb{R}^2$.
% such that for any $f$
%\begin{align*} 
%((\Sigma\cdot\nabx)f )\circ
%\dd \mathbf{B}_t
%=
%\sum_{k\geq1}
%((\bsigma_k\cdot\nabx)f )\circ
%\dd B_t^k
%\end{align*}
The initial conditions for \eqref{contEq}--\eqref{shellEq} are
\begin{align}
&\eta(\cdot,0, \cdot) = \eta_0(\cdot, \cdot), \quad \partial_t\eta(\cdot,0, \cdot) = \eta_1(\cdot, \cdot)
&\quad \text{in }  \Omega\times\Gamma ,
\label{initialStructure}
\\
&\mathbf{u}(\cdot,0, \cdot) = \mathbf{u}_0(\cdot, \cdot)
&\quad \text{in }  \Omega\times\mathcal{O}_{\eta_0}.
\label{initialDensityVelo}
\end{align}
With regards to boundary conditions, we supplement  the shell equation \eqref{shellEq} with periodic boundary conditions and impose  
\begin{align}
\label{noSlip}
&\bu\circ\bm{\varphi}_\eta   =\bn\partial_t\eta
&\quad \text{on }\Omega\times I \times \Gamma  
\end{align}
at the fluid-structure interface. Note that \eqref{contEq}--\eqref{noSlip} is  a free-boundary problem where the boundary of the fluid domain is moving randomly in time.

 In the 3D case regularity and uniqueness of solutions to \eqref{contEq}--\eqref{noSlip} is certainly out of reach (at least globally in time) so that one can only hope for the existence of weak martingale solutions. Here weak refers to the analytical concept of distributional derivatives, whereas martingale solutions refers to solutions which are weak in the probabilistic sense (they do not exist on a given stochastic basis; the latter becomes an integral part of the solution).
Such a concept is very common in stochastic evolutionary problems (even on the level of ordinary stochastic differential equations), whenever uniqueness of the underlying system is unavailable.
%\begin{remark}
%We emphasis that the  boundary $\Gamma$ is a two-dimensional surface on which the shell equation is posed. For this reason, we need to impose a further boundary condition on $\Gamma$. We choose periodic boundary conditions although the result does not change nor become more difficult with the imposition of  Dirichlet boundary conditions.
%\end{remark} 

\subsection{The weak formulation}
A first rather philosophical question is to come up with an analytically weak formulation for the problem. In fluid-structure interaction problems the space of test-functions typically depends on the structure displacement $\eta$ (the test-function for the fluid sub-problem and the structure sub-problem must match at the interface as in \eqref{noSlipA} and \eqref{noSlip}). On the other hand, in stochastic PDEs it is common to work with spatial test-functions. This is also our preferred point of view as an $\eta$-dependence of the test-functions in our case means that they depend on time and are also random. The idea now is to start with a pair of test-functions $(\phi,\bfphi)$ on the reference domain (that is $\phi:\Gamma\rightarrow\R$ and $\bfphi:\mathcal O\rightarrow\R^3$) with the correct boundary condition and transform $\bfphi$ to the moving domain. An obvious choice, therefore, is the Hanzawa transform $\bfPsi_\eta:\mathcal O\rightarrow\mathcal O_\eta$ which we formally introduce in Section  \ref{sec:geometry}. Unfortunately, it has the disadvantage of destroying the divergence-free constraint on the test-functions. At the level of weak solutions this cannot be remedied through the recovery of the pressure function as the latter only exists as a distribution on the solenoidal test-functions. Thus, we use instead the Piola transform
\begin{align}
\label{piolaTransformA}
\mathcal{J}_{\eta}\bv
=
\big(\nabx  \bm{\Psi}_{\eta}(\mathrm{det}\nabx  \bm{\Psi}_{\eta})^{-1}
\bv
\big)\circ \bm{\Psi}_{\eta}^{-1},
%, \quad t\in \overline{I}
\end{align}
which preserves the solenoidability of a functions.
Using now $(\iota_\eta\phi,\mathcal J_{\eta(t)} \bm{\phi})$ with $\iota_\eta=(\mathrm{det}\nabx  \bm{\Psi}_{\eta})^{-1}$ we obtain the following weak formulation
\begin{align}
\nonumber
\dd  \bigg(\int_{\mathcal{O}_{\eta}}&\bu  \cdot \mathcal J_{\eta(t)} \bm{\phi}\dx
+
 \int_{\Gamma}   \partial_t\eta \, \iota_\eta\phi\dy
 \bigg)
=
\int_{\Gamma}   \big(\partial_t\eta \, \partial_t(\iota_\eta\phi)
-
 \Dely (\iota_\eta\phi)\, \Dely\eta
 \big)\dy\dt
\\\nonumber&
+ 
\int_{\mathcal{O}_{\eta}} \bigg(\bu  \cdot  \partial_t(\mathcal J_{\eta(t)} \bm{\phi})
+
 ((\bu\cdot\nabx )\mathcal J_{\eta(t)} \bm{\phi}) \cdot \bu
-
 \nabx\bu:\nabx \mathcal J_{\eta(t)} \bm{\phi}
\bigg) \dx\dt 
 \\& 
-
\int_{\Gamma}  
 \partial_t \eta(( \bsigma 
\cdot\naby)\iota_\eta\phi) \dy\circ\dd B_t\label{weaklimit'A}
\end{align}
for all test-functions $(\phi,\bfphi)$ (which clearly depend  only on space), see Section \ref{sec:weak} for the  derivation.  One easily shows that a dense set of pairs of test-functions $(\phi,\bfphi)$ with the correct boundary condition leads to a dense set of pairs of test-functions $(\iota_\eta\phi,\mathcal J_{\eta(t)} \bm{\phi})$ on the moving domain with the right boundary condition, see \cite[page 237]{lengeler2014weak}. Thus our weak formulation is consistent with the strong formulation. However, the Piola transform behaves like $\nabx\bfPsi_\eta$ and  inherits the regularity of $\naby\eta$ so that we require more regularity on $\eta$.
Since the embedding $W^{2,2}(\Gamma)\hookrightarrow W^{1,\infty}(\Gamma)$ fails in two dimensions, $\nabx\bfPsi_\eta$ is not bounded. Thus the information from \eqref{eq:enA} is not sufficient to give a meaning to all the terms in \eqref{weaklimit'A}. Hence we need additional regularity. A crucial point in our approach is, therefore, to establish an estimate for the $L^2(I;W^{2+s,2}(\Gamma))$-norm of $\eta$ for all $s\in(0,1/2)$. One can easily check by using H\"older's inequality and Sobolev's embedding that this information, together with \eqref{eq:enA}, is sufficient to define all integrals appearing in \eqref{weaklimit'A}. Different to \eqref{eq:enA} this estimate is not independent of the transport noise and hence only holds in expectation. As it turns out, the regularity of the terms arising from the transport noise have just enough regularity to close the estimate. Details can be found in Section \ref{sec:highreg}.
Concluding this discussion, the additional fractional differentiability of the shell displacement must be included in the definition of  a solution, see Definition \ref{def:strongSolutionAlone}.  
%The precise formulation of a martingale solution to \eqref{contEq}--\eqref{noSlip} can be found in Definition \ref{def:strongSolutionAlone} followed by our main result giving its existence in Theorem \ref{thm:main}.

\subsection{Plan}

This work straddles different fields of mathematics including fluid mechanics, partial differential equations, differential geometry and stochastic analysis. In order to make this work as self-contained as possible, we collect in Section \ref{sec:frame} useful results in the different fields of mathematics that are essential in establishing our result. We begin by giving a rigorous interpretation of the Stratonovich integral in \eqref{shellEq} after which we introduce the geometric setup for the fluid-structure system \eqref{contEq}--\eqref{shellEq}. We also present the functional analytic framework
(function spaces on moving domains) and present some key results necessary for our analysis (extension operators). Finally, we make precise, the notion of a solution that we are interested in (Definition \ref{def:strongSolutionAlone}) and state our main result (Theorem \ref{thm:main}).

The proof of our main result can be summarized into three main steps. In Section \ref{sec:linear}, we consider an extension of the fluid-structure system that incorporate `artificial' regularising terms in the shell equation \eqref{shellEq}. We then construct a solution to the linearised version of this extended system using a Galerkin approximation and stochastic compactness tools. We then move to Section \ref{sec:regul} where we use a fixed-point argument to remove the linearisation performed in the previous section and obtain a solution to the fully nonlinear system (with the extra regularising terms in the shell equation). To complete the proof of the main result, we pass to the limit in the regularisation parameter in Section \ref{sec:limit} to finally obtain a solution for \eqref{contEq}--\eqref{noSlip}. In each stage we apply a refined stochastic compactness method which is based on Jakubowski's extension of the Skorokhod representation theorem  \cite{jakubow}.
 In our case it is crucial to re-interpret the compactness lemma from \cite{MuSc} in the context of tightness of probability measures, see Sections \ref{ssec:comp} and \ref{ssec:compv}. 

\begin{remark}[The 2D case]
One might wonder whether it is possible to show the global-in-time existence of strong pathwise solutions to \eqref{contEq}--\eqref{noSlip} (solutions existing on a given stochastic basis). In the 2D case the existence of global-in-time strong solutions to \eqref{contEqA}--\eqref{noSlipA} has been shown in \cite{breit2022regularity} with additional dissipation (see also for \cite{MR3466847} for previous results for elastic plates) and in \cite{ScSu} for the case of plates (but without dissipation). In all cases the approach is heavily based on taking temporal derivatives of $\bu$ and $\partial_t\eta$, which is not possible for \eqref{contEq}--\eqref{noSlip}.
Hence its is unclear if one should even expect such a result here.
\end{remark}

\section{Mathematical framework and the main result}
\label{sec:frame}
\subsection{Stratonovich integrals}
\label{sec:strat}

Let $(\Omega,\mathfrak F,(\mathfrak F_t)_{t\geq0},\mathbb P)$ be a stochastic basis with a complete, right-continuous filtration and let $(B_t)$ represent a real-valued Brownian motion relative to $(\mathfrak F_t)$. 
We consider a smooth solenoidal vector field $\bsigma:\mt\rightarrow\R^2$.
If $\xi\in L^2(\Omega;C(\overline I;W^{1,2}(\mt)))$ is $(\mathfrak F_t)$-adapted%\footnote{Adaptedness can be interpreted in the sense of random distributions, cf. \cite[Chapter 2.8]{BFHbook}.}
, the stochastic integral
\begin{align*}
\int_0^t\bsigma\cdot\naby\xi \,\D B_s
\end{align*}
is well-defined in the sense of It\^{o} with values in $L^{2}(\mt)$. 
If we only have
$\xi\in L^2(\Omega;C(\overline I;L^{2}(\mt)))$ one can use the identity $(\bsigma\cdot\naby)\xi=\divy(\bsigma\xi)$ and 
define the stochastic integral
\begin{align*}
\int_0^t\divy(\xi\bsigma) \,\D B_s
\end{align*}
with values in $W^{-1,2}(\mt)$. 
We define the Stratonovich integrals in \eqref{shellEq} by means of the It\^{o}-Stratonovich correction, that is
\begin{equation}
\begin{aligned}
\label{itoStroCross}
\int_0^t\int_{\mt}\xi\bsigma\cdot\naby\phi\dy \,\circ\D B_s
&=
\int_0^t\int_{\mt}\xi \bsigma\cdot\naby\phi\dy \,\D B_s
\\&+
\frac{1}{2}\Big\langle\Big\langle\int_{\mt}\xi \bsigma\cdot\naby\phi\dy,B_t\Big\rangle\Big\rangle_t
\end{aligned}
\end{equation}
for $\phi\in W^{1,2}(\mt)$.
Here $\langle\langle\cdot,\cdot\rangle\rangle_t$ denotes the cross variation. We compute now the cross variations by means of \eqref{shellEq}. If $\xi=\partial_t\eta$, where $\eta$ solves \eqref{shellEq}, we have for all $\phi\in W^{2,2}(\mt)$ and $t\in I$,
\begin{align*}
\int_{\mt}\partial_t\eta \bsigma \cdot\naby\phi\dy
=&
\int_{\mt}\partial_t\eta(0) \bsigma \cdot\naby\phi\dy
-
\int_0^t\int_{\mt}\Dely\eta\Dely((\bsigma \cdot\naby) \phi)\dy\ds
\\&-
\int_0^t\int_{\mt}g_\eta((\bsigma \cdot\naby) \phi)\dy\ds
%\\&-
%\frac{1}{2}
%\int_0^t\int_{\mt}((\bsigma \cdot\naby)\xi)((\bsigma \cdot\naby)(\bsigma\cdot\naby)\phi)\dy\ds
%+
%\int_0^t\int_{\mt}\xi\,((\bsigma \cdot\naby)(\bsigma\cdot\naby)\phi)\dy \,\D B_s,
%\\&-
%\frac{1}{2}
%\int_{\mt}\lim _{P\rightarrow 0}\sum_{i=0}^{k-1}((\bsigma \cdot\naby)(\bsigma\cdot\naby)\phi)(\xi(s_{i+1})-\xi(s_{i}))(B_{s_{i+1}}-B_{s_{i}})\dy 
\\&+
\frac{1}{2}
\Big\langle\Big\langle
\int_0^t\int_{\mt}\partial_t\eta\,((\bsigma \cdot\naby)(\bsigma\cdot\naby)\phi)\dy ,B_t\Big\rangle\Big\rangle_t\\
&+
\int_0^t\int_{\mt}\partial_t\eta\,((\bsigma \cdot\naby)(\bsigma\cdot\naby)\phi)\dy \D B_s.
\end{align*}
%with partition $P:=\{0=s_0<s_1<\ldots< s_k=t\}$. 
Only the last term on the right contributes to the cross variation when inserted in \eqref{itoStroCross}.
Plugging the previous considerations together we obtain
%\begin{align*}
%&\int_0^t\divy (\xi \bsigma ) \,\circ\D B_s=\int_0^t\divy (\xi \bsigma ) \,\D B_s- \frac{1}{2}\int_0^t\divy(\bsigma\otimes\bsigma\naby\xi)\ds,
%\end{align*}
%to be understood in $W^{-1,2}(\mt)$ or $W^{-2,2}(\mt)$, depending on the regularity of $\xi$. By using $\divy\xi=0$ and integration by parts, we can rewrite the above as
\begin{align*}
&\int_0^t(\bsigma\cdot\naby)\partial_t\eta \,\circ\D B_s=\int_0^t(\bsigma\cdot\naby)\partial_t\eta \,\D B_s-\frac{1}{2}\int_0^t(\bsigma\cdot\naby)(\bsigma\cdot\naby)\partial_t\eta \ds.
\end{align*}
to be understood in $W^{-1,2}(\mt)$ or $W^{-2,2}(\mt)$, respectively, depending on the regularity of $\partial_t\eta$.

\subsection{Geometric setup}
\label{sec:geometry}
 The spatial domain $\mathcal O$ is assumed to be an open bounded subset of $\mathbb{R}^3$ with smooth boundary $\partial\mathcal O$ and an outer unit normal ${\bn}$. We assume that
 $\partial\mathcal O$ can be parametrised by an injective mapping ${\bfvarphi}\in C^k(\Gamma;\R^3)$ for some sufficiently large $k\in\mathbb N$, where $\Gamma$ is the two-dimensional torus. We suppose for all points $\by=(y_1,y_2)\in \Gamma$ that the pair of vectors  
%$\mathbf{a}_i(\by):= 
$\partial_i {\bfvarphi}(\by)$, $i=1,2,$ is linearly independent.
%If $n=2$ the corresponding assumption simply asks for $\partial_\by\bfvarphi=\partial_y\bfvarphi$ not to vanish.
 For a point $\bx$ in the neighbourhood
of $\partial\mathcal O$, we define the functions $\by$ and $s$ by   
\begin{align*}
 \by(\bx)=\arg\min_{\by\in\Gamma}|\bx-\bfvarphi(\by)|,\quad s(\bx)=(\bx-\mathbf p(\bx))\cdot\bn(\by(\bx)),
 \end{align*}
where we used the projection $\mathbf p(\bx)=\bfvarphi(\by(\bx))$. We define $L>0$ to be the largest number such that $s,\by$ and $\mathbf p$ are well-defined on $S_L$, where
\begin{align}\label{eq:boundary1}
S_L=\{\bx\in\R^n:\,\mathrm{dist}(\bx,\partial\mathcal O)<L\}.
\end{align}
Due to the smoothness of $\partial\mathcal O$ for $L$ small enough we have $|s(\bx)|=\min_{\by\in\Gamma}|\bx-\bfvarphi(\by)|$ for all $\bx\in S_L$. This implies that $S_L=\{s\bn(\by)+\by:(s,\by)\in (-L,L)\times \omega\}$.
For a given function $\eta : I \times \Gamma \rightarrow\R$ we parametrise the deformed boundary by
\begin{align}\label{eq:bfvarphi}
{\bfvarphi}_\eta(t,\by)={\bfvarphi}(\by) + \eta(t,\by){\bn}(\by), \quad \,\by \in \Gamma,\,t\in I.
\end{align}
%By possibly decreasing $L$, one easily deduces from this formula that $\Omega_{\eta}$ does not degenerate, that is
%\begin{equation}\label{eq:1705}
%\begin{aligned}
%\partial_1\bfvarphi_\eta\times\partial_2\bfvarphi_\eta(t,\by)\neq0,\quad
% \bfn(\by)\cdot\bfn_{\eta(t)}(\by)&>0,\quad \,\by \in \Gamma,\,t\in I,
% \end{aligned}
%\end{equation}
%provided that $\sup_t\|\eta\|_{W^{1,\infty}_{\by}}<L$. Here $\bfn_{\eta(t)}$
%is the normal of the domain $\mathcal O_{\eta(t)}$
% defined through
%\begin{equation}\label{eq:2612}
%\partial\mathcal O_{\eta(t)}=\set{{\bfvarphi}(\by) + \eta(t,\by){\bfn}(\by):\by\in \Gamma}.
%\end{equation}
With an abuse of notation we define the deformed space-time cylinder as $I\times\mathcal O_\eta=\bigcup_{t\in I}\{t\}\times\mathcal O_{\eta(t)}\subset\R^{4}$.
The corresponding function spaces for variable domains are defined as follows.
\begin{definition}{(Function spaces)}
For $I=(0,T)$, $T>0$, and $\eta\in C(\overline{I}\times\omega)$ with $\|\eta\|_{L^\infty(I\times\Gamma)}< L$ we define for $1\leq p,r\leq\infty$
\begin{align*}
L^p(I;L^r(\mathcal O_\eta))&:=\big\{v\in L^1(I\times\mathcal O_\eta):\,\,v(t,\cdot)\in L^r(\mathcal O_{\eta(t)})\,\,\text{for a.e. }t,
\\&\qquad\qquad\qquad
\|v(t,\cdot)\|_{L^r(\mathcal O_{\eta(t)})}\in L^p(I)\big\},\\
L^p(I;W^{1,r}(\mathcal O_\eta))&:=\big\{v\in L^p(I;L^r(\mathcal O_\eta)):\,\,\nabx v\in L^p(I;L^r(\mathcal O_\eta))\big\}.
\end{align*}
\end{definition}
\noindent 
To establish a relationship between the 
fixed domain and the time-dependent domain, we introduce the Hanzawa transform $\bm{\Psi}_\eta : \mathcal O \rightarrow\mathcal O_\eta$ defined by
\begin{equation}
\label{map}
\bfPsi_\eta(\bx)
=
 \left\{
  \begin{array}{lr}
    \mathbf{p}(\bx)+\big(s(\bx)+\eta(\by(\bx))\phi(s(\bx))\big)\bn(\by(\bx)) &\text{if dist}(\bx,\partial\mathcal O)<L,\\
    \bx &\text{elsewhere}.
  \end{array}
\right.
\end{equation}
for any $\eta:\omega\rightarrow (-L,L)$. Here $\phi\in C^\infty(\mathbb R)$ is such that 
$\phi\equiv 0$ in a neighborhood of $-L$ and $\phi\equiv 1$ in a neighborhood of $0$. The other variables  $\mathbf{p}$, $s$ and $\bn$ are as defined earlier in this section.
A straightforward verification shows that the inverse of $\bm{\Psi}_{\eta(t)}$ is $\bm{\Psi}_{-\eta(t)}$.

In order to obtain a weak formulation for the fluid-structure  system, we also introduce the Piola transform
\begin{align}
\label{piolaTransform}
\mathcal{J}_{\zeta}\bv
=
\big(\nabx  \bm{\Psi}_{\zeta}(\mathrm{det}\nabx  \bm{\Psi}_{\zeta})^{-1}
\bv
\big)\circ \bm{\Psi}_{\zeta}^{-1}
%, \quad t\in \overline{I}
\end{align}
of a vector field $\mathbf{v}:\mathcal O\rightarrow\R^3$  with respect to a mapping $\zeta:\Gamma \rightarrow \mathbb{R}$. The Piola transform is invertible with inverse
\begin{align}
\label{piolaTransformInverse}
\mathcal{J}_{\zeta}^{-1}\bv
=
\big((\nabx  \bm{\Psi}_{\zeta})^{-1}(\mathrm{det}\nabx  \bm{\Psi}_{\zeta}) 
\bv
\big)\circ \bm{\Psi}_{\zeta} 
%, \quad t\in \overline{I}
.
\end{align}
It preserves vanishing boundary values as well as the divergence-free property of a function.
In order to compensate for the additional factor $(\mathrm{det}\nabx  \bm{\Psi}_{\zeta})^{-1}
$ in the trace of $\mathcal{J}_{\zeta}\bv$ we define the mapping
\begin{align*}
\iota_\zeta \phi:=(\mathrm{det}\nabx  \bm{\Psi}_{\zeta}
%\circ\bfvarphi
)^{-1}
\phi
\end{align*}
for a function $\phi:\Gamma\rightarrow\R$. If $\bfphi\circ \bfvarphi=\phi$ on $\Gamma$ it follows that $(\mathcal J_\eta\bfphi)\circ \bfvarphi_\eta=\iota_\eta\phi$ on $\Gamma$.
Thus a pair of test-functions $(\phi,\bfphi)$ with the correct boundary condition leads to a pair of test-functions $(\iota_\eta\phi,\mathcal J_{\eta(t)} \bm{\phi})$ on the moving domain with the right boundary condition. Also, a dense set of test-functions on the reference domain  leads to a dense set of test-functions on the moving domain, see \cite[page 237]{lengeler2014weak}.

We finish this section by recalling  the following Aubin-Lions type lemma which is shown in~\cite[Theorem 5.1. \& Remark 5.2.]{MuSc} and slightly reformulated for our purposes.
\begin{theorem}
\label{thm:auba}
Let $X,Z$ be two Banach spaces, such that $X'\subset Z'$.
Assume that $f_n:I\to X$ and $g_n: I\to X'$, such that $g_n\in L^\infty(I;Z')$ uniformly. Moreover assume the following: 
\begin{enumerate}
\item[(a)] The {\em boundedness}: for some $s\in [1,\infty]$ we have that $(f_n)$ is bounded in $L^s(X)$ and $(g_n)$ in $L^{s'}(X')$.
%\item The {\em uniform bound} on one sequence
%\[
%\sup_n\norm{f_n}_{L^p(0,T;Y)}\leq C.
%\]
%moreover we suppose that $f_n\toweakstar f$ in $L^p(0,T;X)$.
\item[(b)] The {\em approximability-condition} is satisfied: For every $\kappa\in (0,1]$ there exists a $f_{n,\kappa}\in L^s(I;X)\cap L^1(I;Z)$,
%
%and every $t\in [0,T]$ there exists a mollifying operator $\bor{X\ni}\phi\mapsto (\phi)_{\kappa,t}\ni Z$.
such that for every $\epsilon\in (0,1)$ there exists some $\kappa_\epsilon\in(0,1)$ (depending only on $\epsilon$) such that %for a.e. $t\in [0,T]$
\[
%\sup_{k,n}
\|f_n-f_{n,\kappa}\|_{L^s(I;X)}\leq \epsilon\text{ for all } \kappa\in (0,\kappa_\epsilon]%\text{ and }\norm{f_{n,\kappa}}_{Z}\leq  C(\kappa)(1+\norm{f_n}_Y^p).
\]
and for every $\kappa\in (0,1]$ there is some $C(\kappa)$ such that
\[
\|f_{n,\kappa}\|_{L^1(I;Z)}\dt\leq C(\kappa).
\]
%Moreover, we assume that for every $\kappa$ there is a function $f_\kappa$, and a subsequence such that $f_{n,\kappa}\rightharpoonup f_\kappa$ in $L^s(0,T;X)$.

\item[(c)] The {\em equi-continuity} of $g_n$: We require that there exists some $\alpha\in (0,1]$,  functions $A_n$ with $A_n\in L^1(I)$ uniformly, such that for every $\kappa>0$ that there exist some $C(\kappa)>0$ and some $n_\kappa\in \mathbb N$ such that for $\tau>0$, $n\geq n_\kappa$ and a.e.\ $t\in [0,T-\tau]$
\[
 \Big|\tau^{-1}\int_{0}^\tau\langle g_n(t)-g_n(t+s),f_{n,\kappa}(t)\rangle_{X',X}\,\dd s\Big| \leq C(\kappa)\tau^\alpha(A_n(t)+1).
\]

\item[(d)] The {\em compactness assumption} is satisfied: $X'\hookrightarrow \hookrightarrow Z'$. More precisely, every uniformly bounded sequence in $X'$ has a strongly converging subsequence in $Z'$. 
\end{enumerate}
Then there is a subsequence, such that
\[
\int_0^T\langle f_n,g_n\rangle_{X,X'}\dt\rightarrow \int_0^T\langle f,g\rangle_{X,X'}\dt.
\]
\end{theorem}

\subsection{Solenoidal extension}
In this section, we present a linear solenoidal extension operator that maps boundary elements of a spatial domain into the interior. For this end, we first consider the \textit{corrector map}
\begin{align*}
\mathscr K_\eta:L^1(\Gamma )\rightarrow\mathbb{R},\qquad \mathscr K_\eta(\xi)=\frac{\int_{\mathcal{A}_\kappa}{\xi}(\by(\bx))\lambda_\eta(t,\bx)\dx}{\int_{\mathcal{A}_\kappa}\lambda_\eta(t,\bx)\dx},
\end{align*}
where $\lambda_\eta\geq 0$  for $(t,\bx)\in I\times \mathcal{A}_\kappa$ is an appropriately chosen weight function, cf. \cite[equ. (3.3)]{MuSc}, and $\mathcal{A}_\kappa:=S_{\kappa/2}\setminus S_\kappa$. It satisfies
\begin{align*}
&\Vert\mathscr K_\eta(\xi)\Vert_{L^q(I)}
\lesssim \Vert\xi\Vert_{L^q(I;L^1(\Gamma ))}
\\
&\Vert\partial_t\mathscr K_\eta(\xi)\Vert_{L^q(I)}
\lesssim \Vert\partial_t\xi\Vert_{L^q(I;L^1(\Gamma ))}
+
\Vert\xi\partial_t\eta\Vert_{L^q(I;L^1(\Gamma ))}
\end{align*}
for all $q\in[1,\infty]$. The corrector $\mathscr K_\eta$ above preconditions the boundary data to be compatible with the interior solenoidality. 
The following is proved in \cite[Prop. 3.3]{MuSc} and it provides a solenoidal extension. For that we introduce the solenoidal space $W^{1,1}_{\Div}(\mathcal O\cup S_{\alpha} ):=\{\mathbf w\in W^{1,1}(\mathcal O\cup S_{\alpha} )\, :\, \divx\mathbf w=0\}$. 
\begin{proposition}
\label{prop:musc}
For a given $\eta\in L^\infty(I;W^{1,2}(\Gamma))$ with $\|\eta\|_{L^\infty(I\times \Gamma)}<\alpha<L$, there is a linear operator
\begin{align*}
 \test :\{\xi\in L^1(I;W^{1,1}(\Gamma)):\,\mathscr K_\eta(\xi)=0\}\rightarrow L^1(I;W^{1,1}_{\Div}(\mathcal O\cup S_{\alpha} )),
\end{align*}
such that the tuple $( \test (\xi-\mathscr K_\eta(\xi)),\xi-\mathscr K_\eta(\xi))$ satisfies
\begin{align*}
 \test (\xi-\mathscr K_\eta(\xi))&\in L^\infty(I;L^2({\mathcal{O}}_\eta))\cap L^2(I;W^{1,2}_{\Div}({\mathcal{O}}_\eta)),\\
\xi-\mathscr K_\eta(\xi)&\in L^\infty(I;W^{2,2}( \Gamma ))\cap  W^{1,\infty}(I;L^{2}( \Gamma )),
\\
( \test &(\xi-\mathscr K_\eta(\xi))\circ\bm{\varphi}_\eta=\bn(\xi-\mathscr K_\eta(\xi)),
\\
\partial_t( \test &(\xi-\mathscr K_\eta(\xi))  \in L^2(I;L^2({\mathcal{O}}_\eta)),
\\
 \test (\xi-\mathscr K_\eta&(\xi))(t,x)=0 \text{ for } (t,x)\in I \times ({\mathcal{O}} \setminus S_{\alpha})
\end{align*}
provided we have $\xi\in L^\infty(I;W^{2,2}( \Gamma ))\cap  W^{1,\infty}(I;L^{2}(\Gamma))$.
In particular, we have the estimates
\begin{align}\label{musc1}
\| \test (\xi-\mathscr K_\eta(\xi))\|_{L^q(I;W^{1,p}({\mathcal{O}} \cup S_{\alpha}  ))}\lesssim \|\xi\|_{L^q(I;W^{1,p}( \Gamma ))}+\|\xi\naby \eta\|_{L^q(I;L^{p}( \Gamma ))},\\
\label{musc2}\|\partial_t \test (\xi-\mathscr K_\eta(\xi))\|_{L^q(I;L^{p}( {\mathcal{O}}\cup S_{\alpha}))}\lesssim \|\partial_t\xi\|_{L^q(I;L^{p}( \Gamma ))}+\|\xi\partial_t \eta\|_{L^q(I;L^{p}( \Gamma ))},
\end{align}
for any $p\in (1,\infty),q\in[1,\infty]$.
\end{proposition}
The following result is a consequence of Proposition \ref{prop:musc}. 
\begin{corollary}\label{cor:3.5}
Let the assumptions of Proposition \ref{prop:musc} be satisfied and in addition, let $a,r\in[2,\infty]$, $p,q\in(1,\infty)$ and $s\in[0,1]$, and assume that $\eta\in L^r(I;W^{2,a}(\Gamma))\cap W^{1,r}(I;L^a(\Gamma))$. Let $\xi \in W^{s,p}(\Gamma)$ and let $\xi_\delta$ be a smooth approximation of $\xi$ in $\Gamma$. Then $\mathcal E^\eta_{\delta}(\xi):=\test (\xi_\delta-\mathscr K_\eta(\xi_\delta))$ satisfies all the conclusions in Proposition \ref{prop:musc}. In particular,
\begin{align*}
\Vert \partial_t \mathcal E^{\eta}_\delta(\xi) \Vert_{L^r(I;L^a(\mathcal{O}\cup S_\alpha))}
\lesssim
\Vert (\xi_\delta)\partial_t\eta \Vert_{L^r(I;L^a(\Gamma))}
\end{align*}
and
\begin{align*}
\Vert \mathcal E^{\eta}_\delta(\xi) - \test(\xi-\mathscr K_\eta(\xi))\Vert_{L^p(\mathcal{O}\cup S_\alpha)}
\lesssim
\Vert \xi_\delta - \xi \Vert_{L^p(\Gamma)}
\end{align*}
holds uniformly in $t\in I$.
\end{corollary}
For the final statement of this subsection, borrowed from \cite[Lemma 3.5]{MuSc}, we first introduce the following fractional difference quotient in space in the direction $\be_i$ given by $\Delta_{h}^sf(\by)=h^{-s}(f(\by+\be_ih)-f(\by))$ for some $h>0$. 
Now, we define
\begin{align*}
D^{s,\mathscr{K}}_{-h,h}\eta:=\Delta_{-h}^s\Delta_{h}^s\eta
-\mathscr K_\eta(\Delta_{-h}^s\Delta_{h}^s\eta),
\end{align*}
where $s\in(0,\frac{1}{2})$ and the result is as follows:
\begin{lemma}
\label{lem:higherInt}
Let the assumptions of Proposition \ref{prop:musc} be satisfied and in addition,
let $p, \tilde{a}\in(1,\infty)$ be such that $p'<\tilde{a}\leq \frac{3p'}{3-p'}$ if $p'<3$, and $p'<\tilde{a} <\infty$ otherwise. Furthermore, assume that $\eta\in C^{0,\theta}(\Gamma) \cap W^{1, \frac{\tilde{a}p}{\tilde{a}p-\tilde{a}-p}}(\Gamma)$ and $\bu \in W^{1,p'}(\Oeta)$. Then
\begin{align*}
\bigg\vert
\int_{\Oeta} \bu\cdot  \test(D^{s,\mathscr{K}}_{-h,h}\xi)\dx
\bigg\vert
\leq
\Big( h^{\theta-s} +\Vert\Delta_h^s\eta \Vert_{W^{1, \frac{\tilde{a}p}{\tilde{a}p-\tilde{a}-p}}(\Gamma)}\Big)\Vert \bu \Vert_{W^{1,p'}(\Oeta)}\Vert \xi \Vert_{L^p(\Gamma)}
\end{align*}
and when $\partial_t\xi \in L^p(\Gamma)$,
\begin{align*}
\bigg\vert
\int_{\Oeta} \bu\cdot &\partial_t \test(D^{s,\mathscr{K}}_{-h,h}\xi)\dx
\bigg\vert
\lesssim
\Big( h^{\theta-s} +\Vert\Delta_h^s\eta \Vert_{W^{1, \frac{\tilde{a}p}{\tilde{a}p-\tilde{a}-p}}(\Gamma)}\Big)\Vert \bu \Vert_{W^{1,p'}(\Oeta)}\Vert \partial_t\xi \Vert_{L^p(\Gamma)}
\\&+
\Big( \big\Vert \,\vert\Delta_h^s\xi(t)\vert\, \vert\partial_t\eta \vert\,\big\Vert_{L^{\tilde{a}}(\Gamma)}
+
\big\Vert \Delta_h^s\xi(t) \big\Vert_{L^1(\Gamma)} 
\big\Vert\partial_t\eta \vert\big\Vert_{L^1(\Gamma)}
\Big)\Vert \bu \Vert_{W^{1,p'}(\Oeta)}.
\end{align*}
Here, the constants only depends on $\alpha,L$ and $\Vert\eta\Vert_{C^{0,\theta}(\Gamma)}$.
\end{lemma}

\subsection{Weak martingale solutions}
\label{sec:weak}
We are interested in a solution to \eqref{contEq}--\eqref{shellEq} that is weak in the probabilistic sense and also weak in the deterministic sense. From the probabilistic point of view, this means that the stochastic basis is also an unknown of the system and from the deterministic angle, we want a distributional solution of the system integrated against a deterministic test function pair $(\phi,\bm{\phi})\in W^{2,2}(\Gamma )\times W^{1,2}_{\Div}(\mathcal O)$ that satisfies $ \bm{\phi} \circ\bm{\varphi}= \phi\bn$ at the fluid-structure interface $\Gamma$.

We are now deriving the weak formulation of the coupled system assuming we have a sufficiently regular solution at hand.
Since the momentum equation \eqref{momEq} is merely a random PDE rather than a SPDE,  and advected by the large-scale incompressible vector field, we can directly apply Reynolds transport theorem \cite{harouna2017stochastic} to obtain
for $(\iota_\eta\phi,\mathcal J_{\eta(t)} \bm{\phi})$ (recalling the definitions from Section \ref{sec:geometry})
\begin{align*}
\dd  \int_{\mathcal{O}_{\eta}}\bu  \cdot \mathcal J_{\eta(t)} \bm{\phi}\dx
&
=
\int_{\mathcal{O}_{\eta}} \partial_t\bu  \cdot \mathcal J_{\eta(t)} \bm{\phi} \dx\dt
+
\int_{\mathcal{O}_{\eta}} \bu  \cdot  \partial_t(\mathcal J_{\eta(t)} \bm{\phi}) \dx\dt
\\&+
\int_{\mathcal{O}_{\eta}}(\bu\cdot\nabx)(\bu  \cdot \mathcal J_{\eta(t)} \bm{\phi})\dx\dt.
\end{align*}
%using that $\bfu\circ\bfvarphi_\eta=\partial_t\eta\bn$ on $\Gamma$.
We can now use the momentum equation \eqref{momEq} and the divergence-free condition on $\bm{\phi}$ (which transfers to $\mathcal J_{\eta(t)} \bm{\phi}$) to obtain
\begin{align*}
 \int_{\mathcal{O}_{\eta}} \partial_t\bu  \cdot \mathcal J_{\eta(t)} \bm{\phi} \dx\dt
&=
-\int_{\mathcal{O}_{\eta}} ((\bu\cdot\nabx)\bu) \cdot \mathcal J_{\eta(t)} \bm{\phi} \dx\dt
-
\int_{\mathcal{O}_{\eta}} \nabx\bu:\nabx \mathcal J_{\eta(t)} \bm{\phi} \dx\dt
\\&
%+
%\int_{\mathcal{O}_{\eta}}  \pi\divx( \mathcal J_{\eta(t)} \bm{\phi} )\dx\dt
+
\int_{\mathcal{O}_{\eta}}   \divx(\mathbb{T}(\bu,\pi) \mathcal J_{\eta(t)} \bm{\phi} )\dx\dt
\end{align*}
with the latter satisfying
\begin{align*}
\int_{\mathcal{O}_{\eta}}   \divx(\mathbb{T}(\bu,\pi) \mathcal J_{\eta(t)} \bm{\phi} )\dx\dt
&=
\int_{\partial\mathcal{O}_{\eta}}   \bn_\eta\cdot(\mathbb{T}(\bu,\pi) \mathcal J_{\eta(t)} \bm{\phi} )\dd\mathcal{H}^2\dt
%\\
%&=
%\int_{\partial\mathcal{O}_{\eta}}   \bn_\eta\cdot(\mathbb{T}(\bu,\pi) \circ\bm{\varphi}_\eta\circ\bm{\varphi}_\eta^{-1}(\mathcal J_{\eta(t)} \bm{\phi}) \circ\bm{\varphi}_\eta\circ\bm{\varphi}_\eta^{-1})\dd\mathcal{H}^2\dt
%\\
%&=
%\int_{\Gamma}   \bn_\eta\circ\bm{\varphi}_\eta \cdot(\mathbb{T}(\bu,\pi) \circ\bm{\varphi}_\eta (\mathcal J_{\eta(t)} \bm{\phi}) \circ\bm{\varphi}_\eta )\vert\mathrm{det}(\nabx\bm{\varphi}_\eta)\vert\dy\dt
%\\
%&=
%\int_{\Gamma}   \bn_\eta\circ\bm{\varphi}_\eta \cdot(\mathbb{T}(\bu,\pi) \circ\bm{\varphi}_\eta \phi \bn )\,\vert\mathrm{det}(\nabx\bm{\varphi}_\eta)\vert\dy\dt
\\
&=
\int_{\Gamma}   g_\eta\phi  \dy\dt.
\end{align*}
To obtain a distributional formulation for the shell equation \eqref{shellEq}, we first transform it into the It\^o equation
\begin{align*}
\dd \partial_t\eta +\big[\Dely^2\eta +g_\eta
-
\tfrac{1}{2} 
((\bsigma\cdot\naby)(\bsigma\cdot\naby)\partial_t \eta )
\big]\dt 
+
((\bsigma\cdot\naby)\partial_t \eta ) 
\dd B_t=0,
\label{shellEqIto}
\end{align*}
cf. the discussion in Section \ref{sec:strat}.
If we now use It\^o's formula, we obtain
\begin{align*}
\dd \int_{\Gamma}   \partial_t\eta \, \iota_\eta\phi\dy
&=
-
\int_{\Gamma}  \iota_\eta\phi \bigg[\Dely^2\eta +g_\eta
-
\tfrac{1}{2}
((\bsigma\cdot\naby)(\bsigma\cdot\naby)\partial_t \eta )
\bigg]
\dy\dt
\\&
\quad+\int_{\Gamma}   \partial_t\eta \, \partial_t(\iota_\eta\phi)\dy+\int_{\Gamma}  
( (\bsigma\cdot\naby)\partial_t \eta ) 
\iota_\eta\phi \dy\dd B_t,
\end{align*} 
where due to the periodicity of the boundary of $\Gamma$,
\begin{align*}
\int_{\Gamma}  \iota_\eta\phi \Dely^2\eta \dy\dt
=
\int_{\Gamma} \Dely \iota_\eta\phi \Dely\eta \dy\dt.
\end{align*}
If we now use the identity $(\bv_1\cdot\nabx)(\bv_2\cdot\bv_3)-((\bv_1\cdot\nabx)\bv_2)\cdot\bv_3
=((\bv_1\cdot\nabx)\bv_3)\cdot\bv_2$, it follows that
\begin{equation}
\begin{aligned}
\label{weaklimit'}
\dd  \bigg(&\int_{\mathcal{O}_{\eta}}\bu  \cdot \mathcal J_{\eta(t)} \bm{\phi}\dx
+
 \int_{\Gamma}   \partial_t\eta \, \iota_\eta\phi\dy
 \bigg)
=
\int_{\Gamma}   \big(\partial_t\eta \, \partial_t(\iota_\eta)\phi
-
 \Dely \iota_\eta\phi \Dely\eta
 \big)\dy\dt
\\&
+ 
\int_{\mathcal{O}_{\eta}} \bigg(\bu  \cdot  \partial_t(\mathcal J_{\eta(t)}\bm{\phi})
+
 ((\bu\cdot\nabx )\mathcal J_{\eta(t)} \bm{\phi}) \cdot \bu
-
 \nabx\bu:\nabx \mathcal J_{\eta(t)} \bm{\phi}
\bigg) \dx\dt 
 \\&
 +
 \frac{1}{2}
\int_{\Gamma} \partial_t\eta\,
((\bsigma\cdot\naby)(\bsigma\cdot\naby)(\iota_\eta\phi))
\dy\dt 
- 
\int_{\Gamma}  
\partial_t\eta\,( (\bsigma\cdot\naby)
\iota_\eta\phi) \dy\dd B_t.
\end{aligned}
\end{equation} 
Note that $\divx( \mathcal J_{\eta(t)} \bm{\phi} )=0$ and thus, no pressure term appears in the weak formulation. 
The term containing $\partial_t \mathcal J_{\eta}$ is still not well-defined and needs to be rewritten. First of all, we have
\begin{align*}
\mathcal{J}_{\eta}\bfphi
&=
\nabx  \bm{\Psi}_{\eta}\circ \bm{\Psi}_{\eta}^{-1}(\mathrm{det}\nabx  \bm{\Psi}_{\eta}\circ \bm{\Psi}_{\eta}^{-1})^{-1}
\bfphi\circ \bm{\Psi}_{\eta}^{-1}\\
&=\nabx  \bm{\Psi}_{\eta}^{-1}(\mathrm{det}\nabx  \bm{\Psi}_{\eta}^{-1})^{-1}
\bfphi\circ \bm{\Psi}_{\eta}^{-1}\\
&=\nabx  \bm{\Psi}_{-\eta}(\mathrm{det}\nabx  \bm{\Psi}_{-\eta})^{-1}
\bfphi\circ \bm{\Psi}_{-\eta}
\end{align*}
so that
\begin{align*}
\partial_t(\mathcal{J}_{\eta}\bfphi)
&=\partial_t\nabx  \bm{\Psi}_{-\eta}(\mathrm{det}\nabx  \bm{\Psi}_{-\eta})^{-1}
\bfphi\circ \bm{\Psi}_{-\eta}
\\
&-\nabx  \bm{\Psi}_{-\eta}(\mathrm{det}\nabx  \bm{\Psi}_{-\eta})^{-2}\mathrm{tr}((\mathrm{cof}\nabx  \bm{\Psi}_{-\eta})^\top\partial_t\nabx  \bm{\Psi}_{-\eta})
\bfphi\circ \bm{\Psi}_{-\eta}\\
&+\nabx  \bm{\Psi}_{-\eta}(\mathrm{det}\nabx  \bm{\Psi}_{-\eta})^{-1}
\nabx\bfphi\circ \bm{\Psi}_{-\eta}\partial_t \bm{\Psi}_{-\eta}.
\end{align*}
By using Gau\ss\, theorem, we obtain
\begin{align*}
\int_{\mathcal O_\eta}\bfu\cdot\partial_t\nabx  \bm{\Psi}_{-\eta}&(\mathrm{det}\nabx  \bm{\Psi}_{-\eta})^{-1}
\bfphi\circ \bm{\Psi}_{-\eta}\dx
\\&
=\int_{\partial\mathcal O_\eta}((\bfu\cdot\partial_t \bm{\Psi}_{-\eta})(\mathrm{det}\nabx  \bm{\Psi}_{-\eta})^{-1}
\bfphi\circ \bm{\Psi}_{-\eta}) \cdot\bn_\eta\,\dd\mathcal H^2
\\
&-\int_{\mathcal O_\eta}\partial_t \bm{\Psi}_{-\eta}\cdot\Div\big(
\bfu\otimes(\mathrm{det}\nabx  \bm{\Psi}_{-\eta})^{-1}\bfphi\circ \bm{\Psi}_{-\eta}\big)\dx
\end{align*}
and similarly
\begin{align*}
\int_{\mathcal O_\eta}&\bfu\cdot\nabx  \bm{\Psi}_{-\eta}(\mathrm{det}\nabx  \bm{\Psi}_{-\eta})^{-2}\mathrm{tr}((\mathrm{cof}\nabx  \bm{\Psi}_{-\eta})^\top\partial_t\nabx  \bm{\Psi}_{-\eta})
\bfphi\circ \bm{\Psi}_{-\eta}\dx
\\&=
\int_{\mathcal O_\eta}\sum_{j=1}^3\sum_{i=1}^3((\mathrm{cof}\partial_j \Psi_{-\eta}^i) \partial_t\partial_j  \Psi_{-\eta}^i)\bfu\cdot\nabx  \bm{\Psi}_{-\eta} (\mathrm{det}\nabx  \bm{\Psi}_{-\eta})^{-2}
\bfphi\circ \bm{\Psi}_{-\eta}\dx
\\&=
\int_{\mathcal O_\eta}
\sum_{j=1}^3\partial_j\Big(\sum_{i=1}^3((\mathrm{cof}\partial_j \Psi_{-\eta}^i)\partial_t  \Psi_{-\eta}^i)\bfu\cdot\nabx  \bm{\Psi}_{-\eta} (\mathrm{det}\nabx  \bm{\Psi}_{-\eta})^{-2}
\bfphi\circ \bm{\Psi}_{-\eta}\Big)\dx
\\
&-
\int_{\mathcal O_\eta} \sum_{j=1}^3\sum_{i=1}^3\partial_t  \Psi_{-\eta}^i \partial_j\Big( (\mathrm{cof}\partial_j  \Psi_{-\eta}^i) \bfu\cdot(\nabx  \bm{\Psi}_{-\eta} (\mathrm{det}\nabx  \bm{\Psi}_{-\eta})^{-2}
\bfphi\circ \bm{\Psi}_{-\eta})\Big)\dx\\
&=
\int_{\partial\mathcal O_\eta}
\sum_{j=1}^3\Big(\sum_{i=1}^3((\mathrm{cof}\partial_j \Psi_{-\eta}^i)\partial_t  \Psi_{-\eta}^i)\bfu\cdot\nabx  \bm{\Psi}_{-\eta} (\mathrm{det}\nabx  \bm{\Psi}_{-\eta})^{-2}
\bfphi\circ \bm{\Psi}_{-\eta}\Big)\,n^j_{\eta}\,\dd\mathcal H^2
\\
&-
\int_{\mathcal O_\eta} \sum_{j=1}^3\sum_{i=1}^3\partial_t  \Psi_{-\eta}^i \partial_j\Big( (\mathrm{cof}\partial_j  \Psi_{-\eta}^i) \bfu\cdot(\nabx  \bm{\Psi}_{-\eta} (\mathrm{det}\nabx  \bm{\Psi}_{-\eta})^{-2}
\bfphi\circ \bm{\Psi}_{-\eta})\Big)\dx
\end{align*}
where $ \Psi_{-\eta}^i$ is the $i$-th component of $\bm{\Psi}_{-\eta} $ and $n^j_\eta$ that of $\bn_\eta$. The last term of $\partial_t(\mathcal{J}_{\eta}\bfphi)$ does not require such an integration by parts.
Combining H\"older's inequality with Sobolev's embedding and using that $\bfPsi_\eta$ has the same regularity as $\eta$, one easily checks that for a weak solution with regularity as below, all terms are well-defined. 

 With this preparation, we now give the precise notion of a solution.
\begin{definition}[Weak martingale solution]
\label{def:strongSolutionAlone}
Let $(\eta_0, \eta_1, \bu_0,\bsigma)$ be a dataset such that
\begin{equation}
\begin{aligned}
\label{datasetAlone}
&
\eta_0 \in W^{2,2}(\Gamma ) \text{ with } \Vert \eta_0 \Vert_{L^\infty( \Gamma )} < L, \quad
\eta_1 \in L^{2}(\Gamma ), 
\\
&\bu_0\in L^{2}_{\mathrm{\divx}}(\mathcal{O}_{\eta_0} ) \text{ is such that }
 \bu_0 \circ\bm{\varphi}_{\eta_0}   =\eta_1 \bn \text{ on } \Omega\times \Gamma,
\\&
\Vert\bsigma\Vert_{W^{1,\infty}(\Gamma)}\lesssim 1.
\end{aligned}
\end{equation} 
We call 
$( (\Omega,\mathfrak{F},(\mathfrak{F})_{t\geq0},\mathbb{P}), \eta, \bu )$
a \textit{weak martingale solution}  of \eqref{contEq}--\eqref{shellEq} with data $(\eta_0, \eta_1, \bu_0,\bsigma)$ provided that the following holds:
\begin{itemize}
\item[(a)] $(\Omega,\mathfrak{F},(\mathfrak{F})_{t\geq0},\mathbb{P})$ is a stochastic basis with a complete right-continuous filtration;
\item[(b)] $B_t$ is an $(\mathfrak{F}_t)$-Brownian motion;
\item[(c)]  the shell function $\eta$ is $(\mathfrak{F}_t)$-adapted with  $
\Vert \eta \Vert_{L^\infty(I \times \Gamma )} <L$ a.s. and and for all $s\in(0,1/2)$
\begin{align*}
\eta \in L^{\infty}\big(I;W^{2,2}(\Gamma ) \cap L^2(I;W^{2+s,2}(\Gamma)) \big), \qquad \partial_t\eta \in  C_w\big(\overline{I};L^{2}(\Gamma )  \big)\qquad a.s.;
\end{align*}
\item[(d)] the velocity $\bu$ is $(\mathfrak{F}_t)$-adapted with $\bu\circ\bm{\varphi}_\eta =\bn\partial_t\eta$ on $I\times \Gamma $ a.s. 
\begin{align*}
\bu\in  C_w \big(\overline{I}; L^2_{\divx}(\mathcal{O}_{\eta} ) \big)\cap L^2\big(I;W^{1,2}(\Oeta)  \big) \qquad a.s.;
\end{align*}
\item[(e)] equation \eqref{weaklimit'} holds
a.s. for all $(\phi,\bm{\phi})\in W^{2,2}(\Gamma )\times W^{1,2}_{\Div}(\mathcal O)$ with $\bm{\phi} \circ\bm{\varphi}= \phi\bn$ on $\Gamma$.
\item[(e)] The energy inequality holds in the sense that
\begin{align}
\tfrac{1}{2}\int_{\mathcal O_{\eta(t)}}|\bu(t)|^2\dx&+\int_0^t\int_{\mathcal O_{\eta(s)}}|\nabx\bu |^2\dx\ds+\tfrac{1}{2}\int_\Gamma|\partial_t\eta(t)|^2\dy
+
\tfrac{1}{2}\int_\Gamma|\Dely\eta(t)|^2\dy
\nonumber
\\
&\leq \tfrac{1}{2}\int_{\mathcal O_{\eta_0}}|\bu_0|^2\dx+\tfrac{1}{2}\int_\Gamma|\eta_1|^2\dy
+
\tfrac{1}{2}\int_\Gamma|\Dely\eta_0|^2\dy
\label{eq:en}
\end{align}
a.s. for a.a. $t\in I$.
\end{itemize}
\end{definition}

The following is our main result.

\begin{theorem}\label{thm:main}
Let $(\eta_0, \eta_1, \bu_0,\bsigma)$ be a dataset such that
\eqref{datasetAlone} holds. Then there is a weak martingale solution of \eqref{contEq}--\eqref{shellEq} 
with data $(\eta_0, \eta_1, \bu_0,\bsigma)$ in the sense of Definition \ref{def:strongSolutionAlone}.
The interval of existence is of the form $\overline{I}=(0,t)$, where $t<T$
 only if $\lim_{s\rightarrow t}\|\eta(s)\|_{L^\infty(\Gamma)}=L$
a.s. in $\Omega_0$ for some $\Omega_0\subset\Omega$ with $\mathbb P(\Omega_0)>0$.
\end{theorem}

\section{The linearised problem}
\label{sec:linear}
\noindent In the first instant, we wish to construct a weak solution to a system with a regularized geometry and a regularized convection term. Here, by a  \textit{regularized geometry}, we mean a regularization  of a solution to a \textit{given} shell equation and not the solution to our anticipated shell equation \eqref{shellEq}. Thus, we aim at solving the system
%\todo{please add a term $\epsilon \mathcal '(\eta)$ in the shell equation giving $\eta\in L^\infty(W^{3,2})$}
\begin{align}
\divx \bu=0, 
\label{contEqeps}\\
 \partial_t \bu  + (\bu_\epsilon\cdot \nabx)\bu
= 
\Delx \bu -\nabx\pi,
\label{momEqeps}
\\
\dd \partial_t\eta +(\epsilon\mathcal{L}'(\eta) +\epsilon\partial_t\Dely^2\eta +\Dely^2\eta +g_{\eta_\epsilon})\dt 
+
((\bsigma\cdot\naby)\partial_t \eta )\circ
\dd B_t=0,
\label{shellEqeps}
\end{align}
in $I\times \mathcal{O}_{\eta_\epsilon}$ where
\begin{align*}
&g_{\eta_\epsilon}=\bn^\top(\mathbb{T}(\bu, \pi )\bn_{\eta_\epsilon}  )\circ\bm{\varphi}_{\eta_\epsilon} \vert
\det(\naby\bm{\varphi}_{\eta_\epsilon})\vert,\quad \mathbb{T}(\bu, \pi )= (\nabx\bu+\nabx\bu^\top)-\pi\mathbb{I}_{3\times3},
\end{align*}
$\mathcal{L}'$ is the operator given $\int_\Gamma\mathcal{L}'(\eta)\phi \dy = \int_\Gamma \naby^3\eta:\naby^3\phi \dy$ for all $\phi\in W^{3,2}(\Gamma)$, and $\epsilon>0$ is a fixed regularisation parameter. With some slight abuse of notation we denote by $f_\epsilon$ the regularisation of a function on the fluid domain (which is previously extended by zero to the whole space) as well as the regularization of a function defined on $I\times \Gamma$. 
The regularisation is taken with respect to space and time, where the temporal regularization is taken backwards in extending functions to $(-\infty,T)$ by their values at time 0.
A martingale solution to \eqref{contEqeps}--\eqref{shellEqeps} can be defined analogously to Definition \ref{def:strongSolutionAlone}. We aim to show the following result (the proof of Theorem \ref{thm:main'} can be found in the next section).
\begin{theorem}\label{thm:main'}
Let $(\eta_0, \eta_1, \bu_0,\bsigma)$ be a dataset such that
\eqref{datasetAlone} holds and we have additionally $\eta_0\in W^{3,2}(\Gamma)$. 
Then there is a weak martingale solution of \eqref{contEqeps}--\eqref{shellEqeps} 
with data $(\eta_0, \eta_1, \bu_0,\bsigma)$. The interval of existence is of the form $\overline{I}=(0,t)$, 
where $t<T$ only if $\lim_{s\rightarrow t}\|\eta(s)\|_{L^\infty(\Gamma)}=L$
a.s. in $\Omega_0$ for some $\Omega_0\subset\Omega$ with $\mathbb P(\Omega_0)>0$.
\end{theorem}

 In order to solve \eqref{contEqeps}--\eqref{shellEqeps}
we linearize the problem by replacing the regularized velocity in the convective term with a regularization  of a \textit{given} velocity field $\bv\in\mathbb{R}^3$. We also replace the regularized geometry with a regularized geometry with respect to a  given structure displacement $\zeta$  with an initial state $\zeta(0,\cdot)=\eta_0$. The corresponding regularization of the pair $(\zeta,\bv)$ is denoted by $(\zeta_\epsilon,\bv_\epsilon)$.
%\footnote{Here, $f_\epsilon:=\mathcal{R}_\epsilon f$, where $(\mathcal{R}_\epsilon)_{\epsilon>0}$ commutes with $\partial_t$. See \cite{lengeler2014weak} for more details.} 
The solution we seek will be constructed as the limit $N\rightarrow\infty$ of the solution $(\eta^N,\bu^N)$ to a finite dimensional Galerkin approximation system incorporating these regularizing terms. Since this is a linear system we aim to construct a probabilistically strong solution defined on a stochastic basis $(\Omega,\mathfrak F,(\mathfrak F_t)_{t\geq0},\mathbb P)$ and driven by a given Brownian motion  $(B_t)$ relative to $(\mathfrak F_t)$. Suppose that $(\zeta,\bv)$ (and thus its regularization $(\zeta_\epsilon,\bv_\epsilon)$) are a given pair of $(\mathfrak F_t)$-progressively measurable\footnote{To be understood in the sense of random distributions, cf. \cite[Chapter 2.8]{BFHbook}.} random variables with values in 
$C(\overline I\times \Gamma)\times L^2(I;L^2(\mathcal O\cup S_\alpha))$ belonging to $L^p(\Omega)$ for some sufficiently large $p$.
%\begin{align*}
%W^{1,2}(I;W^{2,2}(\Gamma))\times L^\infty(I;W^{3,2}(\Gamma))\cap W^{1,\infty}(I;L^2(\Gamma))\times L^\infty(I;L^2(\mathcal O\cup S_\alpha)),
%\end{align*}
where we suppose that $\epsilon$ is small enough such that  $\|\zeta_\epsilon\|_{L^\infty(I\times\Gamma)}< \alpha<L$ a.s. 
We now look for an $(\mathfrak F_t)$-progressively measurable process
$(\eta,\bu)$ with values in the space
\begin{equation}
\begin{aligned}
\label{regPairDomain}
W^{1,2}(I;W^{2,2}(\Gamma))&\times L^\infty(I;W^{3,2}(\Gamma))\cap W^{1,\infty}(I;L^2(\Gamma))
\\&\times L^\infty(I;L^2(\Omega_{\zeta_\epsilon}))\cap L^2(I;W^{1,2}_{\divx}(\Omega_{\zeta_\epsilon}))
\end{aligned}
\end{equation}
 such that
\begin{equation}
\begin{aligned}
\label{weakeps}
\dd  \bigg(&\int_{\mathcal{O}_{\zeta_\epsilon}}\bu  \cdot \mathcal J_{\zeta_{\epsilon}(t)} \bm{\phi}\dx
+
 \int_{\Gamma}   \partial_t\eta \, \iota_{\zeta_{\epsilon}}\phi\dy
 \bigg)
\\&=
\int_{\Gamma}   \big(\partial_t\eta \, \partial_t(\iota_{\zeta_{\epsilon}}\phi)
-
\Dely \iota_{\zeta_{\epsilon}}\phi \,\Dely\eta
-
\epsilon
\Dely \iota_{\zeta_{\epsilon}}\phi \,\partial_t\Dely\eta
\big)\dy\dt
\\&
-\epsilon
\int_{\Gamma}
\naby^3 \iota_{\zeta_{\epsilon}}\phi :\naby^3\eta
\dy\dt
+\int_{\Gamma} \bigg(\frac{1}{2}\bn_{\zeta_\epsilon } \cdot \bn \iota_{\zeta_\epsilon}\phi \,  \partial_t\zeta_{\epsilon} \,\partial_t\eta \,
 \vert\det(\naby\bm{\varphi}_{\zeta_\epsilon })\vert
 \bigg)\dy\dt
\\
&+
 \int_{\mathcal{O}_{\zeta_\epsilon}}\Big(  \bu\cdot \partial_t (\mathcal{J}_{\zeta_\epsilon (t)} \bm{\phi} )
-
\frac{1}{2}((\bv_\epsilon\cdot\nabx)\bu)\cdot (\mathcal{J}_{\zeta_\epsilon(t)}\bfphi )  \Big) \dx\dt 
\\
&+
 \int_{\mathcal{O}_{\zeta_\epsilon}}\Big(  
\frac{1}{2}((\bv_\epsilon\cdot\nabx)\mathcal{J}_{\zeta_\epsilon(t)}\bm{\phi}) \cdot \bu
-\nabx \bu:\nabx (\mathcal{J}_{\zeta_\epsilon (t)}\bm{\phi}) \Big) \dx\dt  
 \\
&
 +
 \frac{1}{2}
\int_{\Gamma} 
((\bsigma\cdot\naby)(\bsigma\cdot\naby)\partial_t \eta ) \iota_{\zeta_{\epsilon}}\phi
\dy\dt 
+
\int_{\Gamma}  
( (\bsigma\cdot\naby)\partial_t \eta ) 
\iota_{\zeta_{\epsilon}}\phi \dy\dd B_t
\end{aligned}
\end{equation}
for all $(\phi,\bm{\phi})\in W^{3,2}(\Gamma )\times W^{1,2}_{\Div}(\mathcal O)$ with $\bm{\phi} \circ\bm{\varphi}= \phi\bn$ on $\Gamma$. Moreover, we require $\bu\circ\bfvarphi_{\zeta_\epsilon}=\bn\partial_t\eta$ on $I \times \Gamma$.
% and assume that they satisfy the interface condition $ \bv_\epsilon\circ\bm{\varphi}_{\zeta_\epsilon}=\bn\partial_t \zeta_\epsilon$ on $I\times\Gamma$. We claim that $(\bu^N,\eta^N)$ exists that solves  
%\begin{align}
% \partial_t \bu^N  + (\bv_\epsilon\cdot \nabx)\bu^N  
%= 
%\Delx \bu^N -\nabx\pi^N,
%\label{momEqN}
%\\
%\dd \partial_t\eta^N +(\Dely^2\eta^N +g^N_{\zeta_\epsilon})\dt 
%+
%((\bsigma\cdot\naby)\partial_t \eta^N )\circ
%\dd B_t=0,
%\label{shellEqN}
%\end{align}

%Compared to \eqref{eq:en} this gives the additional term $\epsilon\int_\Gamma|\naby^3\eta(t)|^2\dy$
%on the left-hand side and thus an estimate for $\eta\in L^\infty(I;W^{3,2}(\omega))$. In particular, the boundary of
%$\mathcal O_{\eta(t)}$ is Lipschitz uniformly in time. The aim is then to prove the following
\begin{theorem}\label{thm:mainveps}
Let $(\eta_0, \eta_1, \bu_0,\bsigma)$ be a dataset such that
\eqref{datasetAlone} holds
and we have additionally $\eta_0\in W^{3,2}(\Gamma)$. Let $(\Omega,\mathfrak F,(\mathfrak F_t)_{t\geq0},\mathbb P)$ be a stochastic basis with a complete, right-continuous filtration and let $(B_t)$ be an $(\mathfrak{F}_t)$-Brownian motion.  
Then there is a unique probabilistically strong solution of \eqref{weakeps}
with data $(\eta_0, \eta_1, \bu_0,\bsigma)$. The interval of existence is of the form $\overline{I}=(0,t)$, 
where $t<T$ only if $\lim_{s\rightarrow t}\|\eta(s)\|_{L^\infty(\Gamma)}=L$
a.s. in $\Omega_0$ for some $\Omega_0\subset\Omega$ with $\mathbb P(\Omega_0)>0$.
\end{theorem}
It will turn out that the solution satisfies the energy equality
\begin{align}\label{eq:eneps}
\begin{aligned}
\tfrac{1}{2}&\int_{\mathcal{O}_{\zeta_\epsilon}}  |\bu(t)|^2\dx
+
\int_0^t\int_{\mathcal{O}_{\zeta_\epsilon}}|\nabx\bu|^2\dx\ds
+
\epsilon
\int_0^t\int_\Gamma
|\partial_s\Dely\eta|^2\dy\ds
\\&+
\int_\Gamma\Big(\tfrac{1}{2}|\partial_t\eta(t)|^2
+
\tfrac{1}{2}
|\Dely\eta(t)|^2
+\epsilon |\naby^3\eta(t)|^2\Big)\dy
\\
&
= \tfrac{1}{2}\int_{\mathcal O_{\zeta_\epsilon(0)}}|\bu_0|^2\dx
+
\int_\Gamma\Big(\tfrac{1}{2}|\eta_1|^2
+\tfrac{1}{2}|\Dely\eta_0|^2+ \epsilon|\naby^3\eta_0|\Big)\dy
\end{aligned}
\end{align}
a.s. for a.a. $t\in I$.

The aim of the following subsection is to construct a Galerkin approximation of \eqref{contEqeps}--\eqref{shellEqeps}
on a given stochastic basis $(\Omega,\mathfrak{F},(\mathfrak{F})_{t\geq0},\mathbb{P})$, while its limit passage 
(and thus the proof of Theorem \ref{thm:main'}) can be found in the next section.

\subsection{The linearised Galerkin problem}

Let us now explain in which function spaces we seek  the finite dimensional objects $(\eta^N,\bn^N)$. Let $(Y_i)_{i\in \mathbb{N}}$ be a basis of $W^{2,2}(\Gamma )$
 %(the space on mean-free elements in $W^{2,2}(\Gamma )$) 
 and let $( \mathbf{X}_i)_{i\in \mathbb{N}}$ be a basis of $W^{1,2}_{0,\divx}(\mathcal{O})$.
Clearly, there exists divergence-free vectors fields 
$\mathbf{Y}_i$ that are solving Stokes systems in the reference domain $\mathcal{O}$ 
with boundary data $(Y_i\bn)\circ\bfvarphi^{-1}$. We then set
\begin{align}\label{eq:bwi}
\bm{w}_i  = \left\{
  \begin{array}{lr}
    \mathbf{X}_i & : i \text{ even},\\
    \mathbf{Y}_i & : i \text{ odd}
  \end{array}
\right.
\end{align}
and set $w_i=\bm{w}_i \circ\bm{\varphi}|_{\Gamma}\cdot\bn $.

Now we take the pair $(\zeta,\bv)$ 
from 
$$L^\infty(\Omega;C(\overline I\times \Gamma))\times L^\infty(\Omega;L^2(I;L^2(\R^3)))$$
being $(\mathfrak F_t)$-progressively measurable.
%Now suppose that there is an $(\mathfrak F_t)$-progressively measurable  process 
%$\zeta$ such that
%$$\zeta \in  L^\infty(I;W^{3,2}(\Gamma))\cap W^{1,\infty}(I;L^2(\Gamma))\quad\text{a.s.}$$
%and $\sup_I\|\zeta\|_{L^\infty(\Gamma)}<L$ a.s. as well as $\bv\in L^2(I;W^{1,2}_{\divx}(\mathcal O_{\zeta}))$
%with $\bv \circ\bm{\varphi}_{\zeta}   =\partial_t\zeta \bn$ on $\mathcal O\times \Gamma$.
We search for $(\alpha^N_i)_{k,N\in \mathbb{N}}:\Omega\times\overline{I}\rightarrow\mathbb{R}^n$ 
such that $$\bu^N=\sum_{i=1}^N\alpha^N_i(\mathcal{J}_{\zeta_\epsilon (t)}\bm{w}_i)\quad\text{and}\quad 
\eta^N(t,\cdot)=\sum_{i=1}^N\int_0^t\alpha^N_i\iota_{\zeta_\epsilon}w_i\ds+\eta_0$$ solve
 the equations\footnote{We neglect the dependency of the unknown $(\eta^N,\bu^N)$ on $\epsilon$ at this point for simplicity.}   
\begin{align}
\dd&\bigg(\int_{\mathcal{O}_{\zeta_\epsilon}}
\bu^N  \cdot \mathcal{J}_{{\zeta_\epsilon}(t)}\bm{w}_j \dx
+
\int_{\Gamma}  \partial_t \eta^N \,\iota_{\zeta_\epsilon} w_j\dy
\bigg)
\nonumber 
\\
&=\int_{\Gamma} \big(\partial_t\eta^N \, \partial_t(\iota_{\zeta_{\epsilon}}w_j)
 -
 \Dely \iota_{\zeta_\epsilon} w_j \Dely\eta^N
 -
\epsilon
\Dely \iota_{\zeta_{\epsilon}}w_j \,\partial_t\Dely\eta^N
\big)\dy\dt
\nonumber 
\\&
-\epsilon
\int_{\Gamma}
\naby^3 \iota_{\zeta_{\epsilon}}w_j :\naby^3\eta^N
\dy\dt
+
\int_{\Gamma} \bigg(\frac{1}{2}\bn_{\zeta_\epsilon} \cdot \bn^\top \iota_{\zeta_\epsilon} w_j \,  \partial_t{\zeta_\epsilon} \,\partial_t\eta^N  \,
 \vert\det(\naby\bm{\varphi}_{\zeta_\epsilon})\vert
\bigg)\dy\dt
\nonumber 
 \\
&+
 \int_{\mathcal{O}_{\zeta_\epsilon}}\Big(  \bu^N\cdot \partial_t (\mathcal{J}_{{\zeta_\epsilon}(t)} \bm{w}_j  )
-
\frac{1}{2}((\bv_\epsilon\cdot\nabx)\bu^N)\cdot (\mathcal{J}_{{\zeta_\epsilon}(t)}\bm{w}_j )  \Big) \dx\dt 
\nonumber 
\\
&+
 \int_{\mathcal{O}_{\zeta_\epsilon}}\Big(  
\frac{1}{2}((\bv_\epsilon\cdot\nabx)\mathcal{J}_{{\zeta_\epsilon}(t)}\bm{w}_j) \cdot \bu^N
-\nabx \bu^N:\nabx (\mathcal{J}_{{\zeta_\epsilon}(t)}\bm{w}_j ) \Big) \dx\dt 
\nonumber 
\\&+
\frac{1}{2} 
 \int_{\Gamma}  ((\bsigma\cdot\naby)(\bsigma\cdot\naby)\partial_t \eta^N)  \iota_{\zeta_\epsilon} w_j \dy\dt
+
  \int_{\Gamma}  ((\bsigma\cdot\naby)\partial_t \eta^N) \iota_{\zeta_\epsilon} w_j\dy\dd B_t
\label{galerkinweak1}
\end{align}
for $1\leq j\leq N$ with an initial condition $\alpha^N_i(0)$ which is such that
\begin{align}
%&\eta^N(0,\cdot)\rightarrow \eta_0 \qquad \text{in}\qquad W^{3,2}(\Gamma ),\label{initialGalerkinConvEta0}
%\\
&\partial_t\eta^N(0,\cdot)\rightarrow \eta_1 \qquad \text{in}\qquad L^2(\Gamma ),\label{initialGalerkinConvEta}
\\
&\bu^N(0,\cdot)\rightarrow \bu_0 \qquad \text{in}\qquad L^2(\mathcal{O}_{\zeta_\epsilon(0)}).
\label{initialGalerkinConvVel}
\end{align}
Note that the derivation of the weak formulation \eqref{galerkinweak1}
is slightly different to the derivation of \eqref{weaklimit'} due to the differences in their respective advective terms. The treatment of the former advection term  goes as follows. By using the trivial identity
\begin{align*}
\int_{\mathcal{O}_{\zeta_\epsilon}}  
 ((\bv_\epsilon\cdot\nabx )\mathcal J_{\zeta_\epsilon (t)} \bm{w}_j) \cdot \bu^N
  \dx\dt 
  &=
  \frac{1}{2}
\int_{\mathcal{O}_{\zeta_\epsilon}}  
 ((\bv_\epsilon\cdot\nabx )\mathcal J_{\zeta_\epsilon (t)} \bm{w}_j) \cdot \bu^N
  \dx\dt 
  \\&+
  \frac{1}{2}
  \int_{\mathcal{O}_{\zeta_\epsilon}}  
 ((\bv_\epsilon\cdot\nabx )\mathcal J_{\zeta_\epsilon (t)} \bm{w}_j) \cdot \bu^N
  \dx\dt 
\end{align*}
we rewrite the last term as follows
\begin{align*}
  \int_{\mathcal{O}_{\zeta_\epsilon}}&  
 ((\bv_\epsilon\cdot\nabx )\mathcal J_{\zeta_\epsilon (t)} \bm{w}_j) \cdot \bu^N
  \dx\dt 
  =
\int_{\mathcal{O}_{\zeta_\epsilon}}  
 \divx( \bv_\epsilon\otimes\mathcal J_{\zeta_\epsilon (t)} \bm{w}_j) \cdot \bu^N
  \dx\dt  
\\&=
\int_{\partial \mathcal{O}_{\zeta_\epsilon}}  
 \bn_{\zeta_\epsilon }\cdot([ \bv_\epsilon\otimes \mathcal J_{\zeta_\epsilon (t)} \bm{w}_j ] \bu^N)
  \dd\mathcal{H}^2\dt  
 -
  \int_{\mathcal{O}_{\zeta_\epsilon}}  
  ( \bv_\epsilon\otimes\mathcal J_{\zeta_\epsilon (t)} \bm{w}_j) :\nabx \bu^N
  \dx\dt  
\\&=
\int_{\partial \mathcal{O}_{\zeta_\epsilon}}  
 \bn_{\zeta_\epsilon }  \cdot([ \bv_\epsilon\circ\bm{\varphi}_{\zeta_\epsilon }\circ\bm{\varphi}_{\zeta_\epsilon }^{-1}  ((\mathcal J_{\zeta_\epsilon (t)} \bm{w}_j)\circ\bm{\varphi}_{\zeta_\epsilon }\circ\bm{\varphi}_{\zeta_\epsilon }^{-1})^\top  ] \bu^N\circ\bm{\varphi}_{\zeta_\epsilon }\circ\bm{\varphi}_{\zeta_\epsilon }^{-1} )
  \dd\mathcal{H}^2\dt  
 \\&
 \qquad-
  \int_{\mathcal{O}_{\zeta_\epsilon}}  
  ( (\bv_\epsilon\cdot\nabx) \bu^N)\cdot\mathcal J_{\zeta_\epsilon (t)} \bm{w}_j
  \dx\dt  
  \\&=
\int_{\Gamma }  
 \bn_{\zeta_\epsilon } \cdot \bn^\top \iota_{\zeta_\epsilon}w_j \,  \partial_t \zeta_\epsilon  \,\partial_t\eta^N  \,
 \vert\det(\naby\bm{\varphi}_{\zeta_\epsilon })\vert
  \dy\dt  
  \\&\qquad+
  \int_{\mathcal{O}_{\zeta_\epsilon}}  
  ( (\bv_\epsilon\cdot\nabx) \bu^N)\cdot\mathcal J_{\zeta_\epsilon (t)} \bm{w}_j
  \dx\dt  
\end{align*}
where we have used $\bu^N\circ\bm{\varphi}_{\zeta_\epsilon } =\bn\partial_t\eta^N $ and $ \mathcal J_{\zeta_\epsilon (t)}\bm{w}_j \circ\bm{\varphi}_{\zeta_\epsilon }= \iota_{\zeta_\epsilon}w_j\bn$ in the last step. This explains the presence of the Jacobian determinant in \eqref{galerkinweak1}.
\\
Moving on, we note that equation \eqref{galerkinweak1}  is equivalent to
{\small
\begin{align*}
\dd &\bigg[\sum_{i=1}^N \alpha^N_i \bigg(\int_{\mathcal{O}_{\zeta_\epsilon}}\mathcal{J}_{\zeta_\epsilon (t)}\bm{w}_i  \cdot \mathcal{J}_{\zeta_\epsilon (t)}\bm{w}_j \dx
+
\int_{\Gamma} \iota_{\zeta_\epsilon}   w_i \,\iota_{\zeta_\epsilon}  w_j\dy
\bigg)\bigg]
\\
&=
\sum_{i=1}^N
 \int_{\mathcal{O}_{\zeta_\epsilon}}
 \alpha^N_i
 \Big(  \mathcal{J}_{\zeta_\epsilon (t)}\bm{w}_i\cdot \partial_t  (\mathcal{J}_{\zeta_\epsilon (t)}\bm{w}_j ) 
-
\frac{1}{2}((\bv_\epsilon\cdot\nabx)\mathcal{J}_{\zeta_\epsilon (t)}\bm{w}_i)\cdot(\mathcal{J}_{\zeta_\epsilon (t)}\bm{w}_j ) \Big) \dx\dt 
\\
&+
\sum_{i=1}^N
 \int_{\mathcal{O}_{\zeta_\epsilon}}
 \alpha^N_i
 \Big( 
\frac{1}{2}((\bv_\epsilon\cdot\nabx)\mathcal{J}_{\zeta_\epsilon (t)}\bm{w}_j) \cdot
(\mathcal{J}_{\zeta_\epsilon (t)} \bm{w}_i)
-
  \nabx (\mathcal{J}_{\zeta_\epsilon (t)}\bm{w}_i):\nabx (\mathcal{J}_{\zeta_\epsilon (t)}\bm{w}_j ) \Big) \dx\dt 
 \\&
 +
 \sum_{i=1}^N
 \frac{\alpha^N_i}{2}
 \int_{\Gamma}  \bn_{\zeta_\epsilon } \cdot \bn^\top \iota_{\zeta_\epsilon} w_j \,  \partial_t\zeta_\epsilon \,\iota_{\zeta_\epsilon} w_i  \,
 \vert\det(\naby\bm{\varphi}_{\zeta_\epsilon })\vert\dy\dt 
  -
   \int_{\Gamma}  
 \Dely \eta_0\, \Dely \iota_{\zeta_\epsilon} w_j \dy\dt 
  \\&
  -
  \sum_{i=1}^N
  \int_{\Gamma}  \int_0^t\alpha^N_i(s)
 \Dely\iota_{\zeta_\epsilon}  w_i(s)\, \Dely \iota_{\zeta_\epsilon} w_j(t) \ds\dy\dt
 -
 \epsilon
   \int_{\Gamma}  
 \naby^3 \eta_0: \naby^3 \iota_{\zeta_\epsilon} w_j \dy\dt 
 \\&
  -\epsilon
  \sum_{i=1}^N
  \int_{\Gamma}  \int_0^t\alpha^N_i(s)
 \naby^3\iota_{\zeta_\epsilon}  w_i(s): \naby^3 \iota_{\zeta_\epsilon} w_j(t) \ds\dy\dt
 -\epsilon
  \sum_{i=1}^N
  \int_{\Gamma}   \alpha^N_i 
 \Dely\iota_{\zeta_\epsilon}  w_i   \Dely \iota_{\zeta_\epsilon} w_j   \dy\dt 
\\&+
\sum_{i=1}^N
\frac{\alpha^N_i}{2} 
 \int_{\Gamma}  ((\bsigma\cdot\naby)(\bsigma\cdot\naby)\iota_{\zeta_\epsilon} w_i)  \iota_{\zeta_\epsilon} w_j \dy\dt
+
\sum_{i=1}^N
\alpha^N_i 
  \int_{\Gamma}  ((\bsigma\cdot\naby)\iota_{\zeta_\epsilon} w_i) \iota_{\zeta_\epsilon} w_j\dy\dd B_t.
\end{align*}
}
To simplify notations, we drop the summation signs by employing Einstein's summation convention. Then for
{\small
\begin{align*}
(a_{ij}(t))
&:=
\int_{\mathcal{O}_{\zeta_\epsilon}}(\mathcal{J}_{\zeta_\epsilon (t)}\bm{w}_i ) \cdot(\mathcal{J}_{\zeta_\epsilon (t)} \bm{w}_j) \dx
+
\int_{\Gamma}   \iota_{\zeta_\epsilon} w_i \,\iota_{\zeta_\epsilon}  w_j\dy,
\\
(b_{ij}(t))
&:=
 \int_{\mathcal{O}_{\zeta_\epsilon}}
 \Big( (\mathcal{J}_{\zeta_\epsilon (t)} \bm{w}_i)\cdot \partial_t(\mathcal{J}_{\zeta_\epsilon (t)}  \bm{w}_j)  
-
\frac{1}{2}((\bv_\epsilon\cdot\nabx(\mathcal{J}_{\zeta_\epsilon (t)}\bm{w}_i)\cdot(\mathcal{J}_{\zeta_\epsilon (t)}\bm{w}_j )\Big) \dx,
\\&
+\int_{\mathcal{O}_{\zeta_\epsilon}}
 \Big(
\frac{1}{2}((\bv_\epsilon\cdot\nabx(\mathcal{J}_{\zeta_\epsilon (t)}\bm{w}_j) \cdot (\mathcal{J}_{\zeta_\epsilon (t)}\bm{w}_i)
 -
  \nabx (\mathcal{J}_{\zeta_\epsilon (t)} \bm{w}_i):\nabx (\mathcal{J}_{\zeta_\epsilon (t)}\bm{w}_j ) \Big) \dx
 \\&
 + 
 \frac{1}{2}
 \int_{\Gamma}  \bn_{\zeta_\epsilon } \cdot \bn^\top \iota_{\zeta_\epsilon}w_j \,  \partial_t\zeta_\epsilon \,\iota_{\zeta_\epsilon}w_i  \,
 \vert\det(\naby\bm{\varphi}_{\zeta_\epsilon })\vert\dy
-\epsilon
\int_{\Gamma}  
 \Dely\iota_{\zeta_\epsilon}  w_i \, \Dely \iota_{\zeta_\epsilon} w_j \dy
 \\&
 +
\frac{1}{2}
 \int_{\Gamma}  ((\bsigma\cdot\naby)(\bsigma\cdot\naby)\iota_{\zeta_\epsilon}w_i) \iota_{\zeta_\epsilon} w_j \dy,
\\
(c_{ij}(t,s))
&:=
-
\int_{\Gamma}  
 \Dely\iota_{\zeta_\epsilon}  w_i(s)\, \Dely \iota_{\zeta_\epsilon} w_j(t) \dy
 -
 \epsilon
\int_{\Gamma}  
 \naby^3\iota_{\zeta_\epsilon}  w_i(s): \naby^3 \iota_{\zeta_\epsilon} w_j(t) \dy,
\\
(d_{j}(t))
&:= -
   \int_{\Gamma}  
 \Dely \eta_0\, \Dely \iota_{\zeta_\epsilon} w_j \dy
 -
 \epsilon
   \int_{\Gamma}  
 \naby^3 \eta_0:\naby^3 \iota_{\zeta_\epsilon} w_j \dy,
 \\
(e_{ij}(t))
&:=
\int_{\Gamma}  ((\bsigma\cdot\naby)\iota_{\zeta_\epsilon} w_i)\iota_{\zeta_\epsilon}  w_j\dy,
\end{align*}
}
we can rewrite the above as the following system of SDEs
\begin{equation}
\begin{aligned}
\label{galerkinweak3}
\int_I\dd \big[\alpha^N_i(t) (a_{ij}(t))\big]
&=
\int_I
 \alpha^N_i(t)(b_{ij}(t))\dt  
  - 
  \int_I
   \int_0^t\alpha^N_i(s)(c_{ij}(t,s)) \ds \dt
\\& -
 \int_I
   (d_j(t))\dt  
+
\int_I
\alpha^N_i(t) (e_{ij}(t))\dd B_t.
\end{aligned}
\end{equation}
Since the coefficient matrix $(a_{ij}(t))$ is symmetric and positive definite, it is
invertible. As the problem is linear, we can infer the existence of a unique global solution, cf. \cite[Theorem 3.1.1.]{PrRo}. 
Moreover, we have the following energy estimate which is obtained by applying It\^{o}'s formula \footnote{This can be rewritten in terms of the $\alpha_i^N$, cf. \eqref{galerkinweak3}, such that a finite dimensional version is sufficient. Note that the coefficients in \eqref{galerkinweak3} are random but differentiable in time.} to the process $t\mapsto \tfrac{1}{2}\int_{\mathcal O_{\zeta_\epsilon}}|\bu^N(t)|^2\dx+\tfrac{1}{2}\int_\Gamma|\partial_t\eta^N(t)|^2\dy$: it holds
\begin{align}\label{eq:energyN}
\begin{aligned}
\tfrac{1}{2}\int_{\mathcal O_{\zeta_\epsilon}} &|\bu^N(t)|^2\dx
+\int_0^t\int_{\mathcal O_{\zeta_\epsilon}}|\nabx\bu^N|^2\dx\ds
+
\epsilon
\int_0^t\int_\Gamma
|\partial_s\Dely\eta^N|^2\dy\ds
\\&+
\int_\Gamma\Big(\tfrac{1}{2}|\partial_t\eta^N(t)|^2
+
\tfrac{1}{2}
|\Dely\eta^N(t)|^2
+\epsilon |\naby^3\eta^N(t)|^2\Big)\dy
\\
&
 =
\tfrac{1}{2}\int_{\mathcal O_{\eta_0}}|\bu_0|^2\dx
+
\int_\Gamma\Big(\tfrac{1}{2}|\eta_1|^2
+\tfrac{1}{2}|\Dely\eta_0|^2+ \epsilon|\naby^3\eta_0|\Big)\dy
\end{aligned}
\end{align}
a.s. for a.a. $t\in I$. From the above, we then obtain
\begin{align}
&\sup_I\Big(\Vert \eta^N \Vert_{W^{2,2}(\Gamma )}^{2} 
+\epsilon
\Vert \eta^N \Vert_{W^{3,2}(\Gamma )}^{2}
\Big) \lesssim1,\label{GalerkinConvEta1Np}
\\
&\sup_I\Vert\partial_t\eta^N \Vert_{L^2(\Gamma )}^{2}\lesssim 1,\label{GalerkinConvEtaNp}
\\
&\epsilon\int_I\Vert\partial_t\eta^N \Vert_{W^{2,2}(\Gamma )}^{2}\dt\lesssim 1,
\\
&\sup_I\Vert \bu^N \Vert_{L^2(\mathcal O_{\zeta_\epsilon})}^{2}\lesssim 1,
\label{GalerkinVolNp1}
\\
&\int_I\Vert\nabx \bu^N \Vert_{L^{2}(\mathcal O_{\zeta_\epsilon})}^2\dt\lesssim1. \label{GalerkinVolNp2}
\end{align}
In addition, for any $s\in(0,\frac{1}{2})$, it follows from $\bu^N\circ\bm{\varphi}_{\zeta_\epsilon } =\bn\partial_t\eta^N $, \eqref{GalerkinVolNp2} and the trace theorem that
\begin{equation}
\begin{aligned}
\label{GalerkinConvEta2Np}
\int_I\Vert \partial_t\eta^N \Vert_{W^{s,2}(\Gamma)}^2\dt
\lesssim 1
\end{aligned}
\end{equation}
holds. 

\subsection{Tightness of $\partial_t\eta^N$}
\label{ssec:comp}
The effort of this subsection is to prove tightness of the law of
$\partial_t\eta^N$ on $L^2$ in order to pass to the limit in the stochastic integral. 
 We define the projection $\mathcal P^N$ and the extension $\mathscr F_N^{\zeta_\epsilon}$ (for a given $\zeta:\omega\rightarrow(-L,L)$) 
\begin{align*}
\mathcal P^N b=\sum_{k=1}^N\alpha_k(b)\iota_{\zeta_\epsilon}w_k,\quad
\mathscr F_N^{\zeta_\epsilon} b=\sum_{k=1}^N\alpha_k(b)\mathcal J_{\zeta_\epsilon}\bm w_k,
\end{align*}
where $\alpha_k(b)=\langle b,\iota_{\zeta_\epsilon} w_k\rangle_{W^{3,2}(\Gamma)}$
 if $w_k=Y_\ell$ for some $\ell\in\mathbb N$ and
$\alpha_k(b)=0$ otherwise.
Obviously, we have $ \mathscr F_N^{\zeta_\epsilon} b\circ\bm{\varphi}_{\zeta_\epsilon }=\bn\mathcal P^N b$ for any $b\in W^{3,2}(\Gamma)$. We have by definition,
\begin{align}\label{eq:PN}
\|\mathcal P^N b\|_{W^{3,2}(\Gamma)}^2\leq\,\|b\|_{W^{3,2}(\Gamma)}^2\quad \forall b\in W^{3,2}(\Gamma).
\end{align}
The eigenvalue equation for the basis vectors implies additionally that
\begin{align}\label{eq:PN2}
\|\mathcal P^N b\|_{L^{2}(\Gamma)}^2\lesssim \|b \|_{L^2(\Gamma)}^2\quad \forall b\in L^2(\Gamma).
\end{align}
By interpolation we obtain an estimate on $W^{s,2}(\Gamma)$ for any $s\in[0,3]$, that is
\begin{align}
\label{eq:projection}
\|\mathcal P^N b\|_{W^{s,2}(\Gamma)}\lesssim \|b\|_{W^{s,2}(\Gamma)}.
\end{align} 
Finally, for ${\zeta_\epsilon}\in W^{3,2}(\Gamma)$ with $\|\zeta_\epsilon\|_{L^\infty(\Gamma)}<\alpha$ for $\alpha\in(0,L)$ we have
\begin{align}\label{eq:FN}
\begin{aligned}
\|\mathscr F^{\zeta_\epsilon}_N b\|_{W^{s+1/2,2}(\mathcal O\cup S_{\alpha})}^2
&= 
\bigg\|\sum_{k=1}^N\alpha_k(b)\mathcal J_{\zeta_\epsilon}\bm w_k
\bigg\|_{W^{s+1/2,2}(\mathcal O\cup S_{\alpha})}^2 
\\
&\lesssim \bigg\|\sum_{k=1}^N\alpha_k(b)\iota_{\zeta_\epsilon} w_k
\bigg\|_{W^{s,2}(\Gamma)}^2
\\&
\leq\,\|b\|_{W^{s,2}(\Gamma)}^2  
\end{aligned}
\end{align}
for all $b\in W^{s,2}(\Gamma)$ and
for any $s\in[0,3]$. Here we used that $\sum_{k=1}^N\alpha_k(b)\bm w_k$ solves the
homogeneous Stokes problem with boundary datum $\sum_{k=1}^N\alpha_k(b)w_k\mathbf n$ in $\mathcal O$ and well-known elliptic estimates (see \cite[Chapter IV]{MR2808162}).
Similarly, we obtain for $\zeta_\epsilon^1,\zeta_\epsilon^2\in W^{1,2}(\Gamma)$
\begin{align}\label{eq:FN2}
\begin{aligned}
\|\mathscr F^{\zeta_\epsilon^1}_N b&-\mathscr F^{\zeta_\epsilon^2}_N b\|_{L^{2}(\mathcal O\cup S_{\alpha})}^2
\\&=\bigg\|(\mathcal J_{\zeta_\epsilon^1}-\mathcal J_{\zeta_\epsilon^2})\sum_{k=1}^N\alpha_k(b)\bm w_k
\bigg\|_{L^{2}(\mathcal O\cup S_{\alpha})}^2
\\&\lesssim \|(\zeta_\epsilon^1,\zeta_\epsilon^2)\|^4_{W^{1,\infty}(\Gamma)} \|\zeta_\epsilon^1-\zeta_\epsilon^2\|_{L^2(\Gamma)}^2\bigg\|\sum_{k=1}^N\alpha_k(b)\bm w_k
\bigg\|_{W^{1,\infty}(\mathcal O)}^2
\\&
+ \|(\zeta_\epsilon^1,\zeta_\epsilon^2)\|^2_{W^{1,\infty}(\Gamma)}\|\naby(\zeta_\epsilon^1- \zeta_\epsilon^2)\|_{L^2(\Gamma)}^2\bigg\|\sum_{k=1}^N\alpha_k(b)\bm w_k
\bigg\|_{L^{\infty}(\mathcal O)}^2
\\
&\lesssim  (1+\|(\zeta_\epsilon^1,\zeta_\epsilon^2)\|^4_{W^{1,\infty}(\Gamma)})\|\zeta_\epsilon^1-\zeta_\epsilon^2\|_{W^{1,2}(\Gamma)}^2\bigg\|\sum_{k=1}^N\alpha_k(b)\bm w_k
\bigg\|_{W^{3,2}(\mathcal O)}^2
\\
&\lesssim  (1+\|(\zeta_\epsilon^1,\zeta_\epsilon^2)\|^4_{W^{1,\infty}(\Gamma)})\|\zeta_\epsilon^1-\zeta_\epsilon^2\|_{W^{1,2}(\Gamma)}^2\bigg\|\sum_{k=1}^N\alpha_k(b)w_k
\bigg\|_{W^{3,2}(\Gamma)}^2
\\
&\leq\, (1+\|(\zeta_\epsilon^1,\zeta_\epsilon^2)\|^4_{W^{1,\infty}(\Gamma)})\|\zeta_\epsilon^1-\zeta_\epsilon^2\|_{W^{1,2}(\Gamma)}^2\|b\|_{W^{3,2}(\Gamma)}^2 
\end{aligned}
\end{align}
for all $ b\in W^{3,2}(\Gamma)$.
%, where the estimate depends on the $W^{1,\infty}(\Gamma)$-norm of $\zeta_\epsilon^1$ and $\zeta_\epsilon^2$.
Following \cite{lengeler2014weak}, we write
\begin{align}\label{eq:convrhouN}\begin{aligned}
\int_{I}\int_{\mathcal O_{\zeta_\epsilon}}|\bu^N|^2\dxt&
+\int_{I}\int_\Gamma|\partial_t\eta^N|^2\,\dy \dt\\&
=\int_{I}\int_{\mathcal O\cup S_{\alpha}}\mathbb I_{\mathcal O_{\zeta_\epsilon}}
\bu^N\cdot\mathscr F_N^{\zeta_\epsilon} \partial_t\mathcal \eta^N\dxt+\int_{I}\int_\Gamma|\partial_t\eta^N|^2\,\dy \dt\\
&+\int_{I}\int_{\mathcal O\cup S_{\alpha}}
\mathbb I_{\mathcal O_{\zeta_\epsilon}}
\bu^N\cdot(\bu^N-\mathscr F_N^{\zeta_\epsilon}\partial_t\mathcal \eta^N)\dxt.
\end{aligned}
\end{align}
%noting that $\mathcal P^N\eta^N=\eta^N$.
We consider the space $X:=L^2(\Gamma)\times W^{-s,2}(\mathcal{O}\cup S_{\alpha})$ with  $s\in(0,1/4)$.  
In order to apply Theorem~\ref{thm:auba} yielding tightness of the corresponding laws
we need to equip $L^2(I;X'\times X)$ with an unconventional topology which we denote by  $\tau_\sharp$
and define the convergence $\rightarrow^\sharp$  as follows. 
 We say that
$$((\zeta^N, \bv^N),(\xi^N,\mathbf{w}^N))
\rightarrow^\sharp((\zeta, \mathbf v),(\xi, {\mathbf w}))\,\,\text{in}\,\,L^2(I;X'\times X)$$ provided
that
$$((\zeta^N, \bv^N),(\xi^N,\mathbf{w}^N))
 \rightharpoonup ((\zeta, \mathbf v),(\xi, {\mathbf w}))\,\,\text{in}\,\,L^2(I;X'\times X)$$
and it holds
\begin{align}\label{eq:tau}
\begin{aligned} 
\int_I\langle \bv^N,\mathbf{w}^N\rangle_{W^{s,2},W^{-s,2}}\dt
&+\int_I\int_\Gamma\zeta^N\,\xi^N\dy\dt\\
&\longrightarrow \int_I\langle \mathbf v,{\mathbf w}\rangle_{W^{s,2},W^{-s,2}}\dt
+\int_I\int_\Gamma\zeta\,\xi\dy\dt.
\end{aligned}
\end{align}
Since this topology is finer than the weak topology on $L^2(I;X'\times X)$ it is clear
 that $(L^2(I;X'\times X),\tau_\sharp)$ is a quasi-Polish
space such that Jakubowski's version of the Skorokhod representation theorem
applies. We obtain the following result concerning tightness.
\begin{lemma}\label{lemma:fg}
The laws of
\begin{align*}
((\partial_t\eta^N,\mathbb I_{\mathcal O_{\zeta_\epsilon}}\mathbf u^N),
(\partial_t\eta^N,\mathscr F_N^{\zeta_\epsilon}\partial_t\eta^N))\quad\text{and}\quad
((\partial_t\eta_n,\mathbb I_{\mathcal O_{\zeta_\epsilon}}\mathbf u^N),
(0,\mathbf u^N-\mathscr F_N^{\zeta_\epsilon}\partial_t\eta^N))
\end{align*}
on $(L^2(I;X'\times X),\tau_\sharp)$ are tight.
\end{lemma}
\begin{proof}
According to Theorem \ref{thm:auba} we first need boundedness
 in $L^2(I;X'\times X)$.
By \eqref{GalerkinConvEta1Np}--\eqref{GalerkinConvEta2Np} and the properties of the projection and extension operators
above this follows immediately (even uniformly in probability). Note that the extension by zero is a bounded operator on $W^{s,2}$ for $s<\frac{1}{4}$.
For $(b)$  we observe that we may assume that
 a regularizer $b\mapsto (b)_\kappa$ exists such that for any $s,a\in \R$
\begin{align}
\label{eq:molly}\|b-(b)_\kappa\|_{W^{a,2}(\Gamma)}\lesssim \kappa^{s-a}\|b\|_{W^{s,2}(\Gamma)},\quad b\in W^{s,2}(\Gamma).
\end{align}
The estimate is well known for $a,s\in\mathbb N_0$, while the general case follows by interpolation and duality. 
Next we introduce the mollification operator on $\partial_t\eta^N$ by considering for $\kappa>0$
 and $N\in\mathbb N$,
$
f_{N}(t):=(\mathcal P^N(\partial_t\eta^N(t)),
\mathscr F_N^{\zeta_\epsilon(t)}(\mathcal P^N(\partial_t\eta^N(t))))
$
and set
\[
f_{N,\kappa}(t):=(\mathcal P^N((\partial_t\eta^N(t))_\kappa),
\mathscr F_N^{\zeta_\epsilon(t)}(\mathcal P^N((\partial_t\eta^N(t))_\kappa))).
\]
We find by the continuity of the mollification operator from \eqref{eq:molly}, the continuity of the projection
operator from \eqref{eq:projection} and the estimate for the extension operator
\eqref{eq:FN} that for a.e.\ $t\in I$ and $s<s_0<1/2$
 \begin{align}\label{eq:fNkappa}
 \|f_{N,\kappa}-f_N\|_{L^2(\Gamma)\times W^{s,2}(\mathcal O\cup S_{L/2})}
\lesssim \kappa^{s_0-s}\|\partial_t\eta^N\|_{W^{s_0,2}(\Gamma)},
 \end{align}
 which can be made arbitrarily small in $L^2$ by choosing $\kappa$ appropriately, 
cf. \eqref{GalerkinConvEta2Np}.
% This allows to deduce the $L^1(I;Z)$ bound (for $s=s_y$) as well as the convergence property (choosing $s=s_x$).
Similarly, we have
 \begin{align*}\label{eq:fNkappa}
 \|f_{N,\kappa}\|_{W^{1,2}(\Gamma)\times W^{1,2}(\mathcal O\cup S_{L/2})}
\lesssim \kappa^{-1}\|\partial_t\eta^N\|_{L^{2}(\Gamma)}.
 \end{align*}%\todo{what is $q$?}
Clearly, if $f_N\rightharpoonup f$ in $L^2(I;X)$ then
 we can deduce a converging subsequence such that $f_{N,\kappa}\weakto f_{\kappa}$ (for some $ f_{\kappa}$) in 
$L^2(I;X)$ for any $\kappa>0$,
 which implies (b).

 For (c) we have to control
$\langle g_N(t)-g_N(s),f_{N,\kappa}(t) \rangle$, 
where $g_N(t):=(\partial_t\eta^N(t),\mathbb I_{\mathcal O_{\zeta_\epsilon(t)}}\mathbf u^N(t))$
and hence we decompose
 \begin{align*}
&\langle g_N(t)-g_N(s),f_{N,\kappa}(t)\rangle
\\
&=\big(\langle g_N(t),(\mathcal P^N((\partial_t\eta^N(t))_\kappa),
\mathscr F_N^{\zeta_\epsilon(t)}(\mathcal P^N((\partial_t\eta^N(t))_\kappa)))\rangle
\\
&\quad-\langle g_N(s),(\mathcal P^N((\partial_t\eta^N(t))_\kappa),\mathscr F_N^{\zeta_\epsilon(s)}
(\mathcal P^N((\partial_t\eta^N(t))_\kappa)))\rangle \big)
\\
&\quad +\langle g_N(s),(0,\mathscr F_N^{\zeta_\epsilon(s)}(\mathcal P^N((\partial_t\eta^N(t))_\kappa))
-\mathscr F_N ^{\zeta_\epsilon(t)}(\mathcal P^N((\partial_t\eta^N(t))_\kappa))\rangle
\\&=:(I)+(II).
 \end{align*} 
We begin estimating $(II)$ to find that
 \begin{align*}
(II)& = \int_{\mathcal O_{\zeta_\epsilon(s)}} \bu^N(s)
\cdot \Big(\mathscr F_N^{\zeta_\epsilon(s)}(\mathcal P^N((\partial_t\eta^N(t))_\kappa))
-\mathscr F_N ^{\zeta_\epsilon(t)}(\mathcal P^N((\partial_t\eta^N(t))_\kappa)\big)\dx
\\&
\lesssim\|\bu^N(s)\|_{L^2(\mathcal O_{\zeta_\epsilon(s)})} (1+\|\zeta_\epsilon\|^2_{L^\infty(I;W^{1,\infty}(\Gamma))})\|\zeta_\epsilon(t)-\zeta_\epsilon(s)\|_{W^{1,2}(\Gamma)}
\\&
\qquad
\times\|\mathcal P^N((\partial_t\eta^N(t))_\kappa)\|_{W^{3,2}(\Gamma)}
%\\
%&
%\lesssim|t-s|\|\bu^N(s)\|_{L^2(\mathcal O_{\zeta_\epsilon(s)})} (1+\|\zeta_\epsilon\|^2_{L^\infty(I;W^{1,\infty}(\Gamma))})\|\zeta_\epsilon\|_{W^{1,\infty}(I;W^{1,2}(\Gamma))}
%\|\mathcal P^N((\partial_t\eta^N(t))_\kappa)\|_{W^{3,2}(\Gamma)}
 \end{align*}
using \eqref{eq:FN2}.
By \eqref{eq:molly} and \eqref{eq:projection} the last term can be estimated by
\begin{align*}
\|\mathcal P^N((\partial_t\eta^N(t))_\kappa)\|_{W^{3,2}(\Gamma)}
\leq c\|(\partial_t\eta^N(t))_\kappa\|_{W^{3,2}(\Gamma)}\leq c\kappa^{-3}\|\partial_t\eta^N(t)\|_{L^2(\Gamma)},
\end{align*}
which is bounded by \eqref{GalerkinConvEtaNp}.
 %(and boundedness of )
Due to this as well as
%he embedding 
%\begin{align*}
%W^{1,2}(I;L^2(\Gamma))\cap L^{2}(I;W^{3,2}(\Gamma))\hookrightarrow W^{2/3,2}(I;W^{1,2}(\Gamma))
%\hookrightarrow C^{1/6}(\overline I;W^{1,2}(\Gamma))
%\end{align*}
%as well as
 \eqref{GalerkinConvEtaNp} and \eqref{GalerkinVolNp1}
we further  have that
 \begin{align*}
|(II)|
&
\leq c(\kappa)|t-s| (1+\|\zeta_\epsilon\|^3_{W^{1,\infty}(I;W^{1,\infty}(\Gamma))})\\
&
\leq c(\kappa,\epsilon)|t-s| (1+\|\zeta\|^3_{C(\overline I\times\Gamma)}).
 \end{align*}
In teh last step we used standard properties of the mollification recalling that $\epsilon$ is a fixed parameter on this level. It follows by Markov's inequality for $K>0$
\begin{align*}
\mathbb P\big(|t-s|^{-1}|(II)|\geq K\big)
&\leq \,\frac{c(\kappa,\epsilon)}{K}\mathbb E\Big[1+\|\zeta\|^3_{C(\overline I\times\Gamma)}\Big]\leq \,\frac{c(\kappa,\epsilon)}{K},
\end{align*}
which can be amde arbitrarily small for large $K$ provided $\zeta\in L^3(\Omega;C(\overline I\times\Gamma))$.
%\begin{align*}
%|(II)|\lesssim |t-s|^{1/6}\|\mathcal P^N((\partial_t\eta^N(t))_\kappa)\|_{W^{3,2}(\Gamma)}.
%\end{align*}
% such that 
%we conclude
%$$(II)\leq\,c(\kappa)|t-s|^{1/6}.$$
  The term $(I)$ is estimated using the test function
 $f_{N,\kappa}$ obtaining
\begin{align}  
\nonumber
\int_t^s&
(I) \dd r 
\\=&
-\int_t^s\int_{\Gamma} 
%\big(\partial_r\eta^N \, \partial_r(\mathcal P^N((\partial_r\eta^N)_\kappa)
 \Dely (\mathcal P^N((\partial_t\eta^N(t))_\kappa) \Dely\eta^N
 \dy\dd r
 \nonumber
 \\\nonumber
 &-\epsilon
\int_t^s\int_{\Gamma} 
\Big(
 \Dely(\mathcal P^N((\partial_t\eta^N(t))_\kappa) \partial_t\Dely\eta^N 
 +
 \naby^3 (\mathcal P^N((\partial_t\eta^N(t))_\kappa) :\naby^3\eta^N \Big)\dy\dd r
\\\nonumber&+
\int_t^s\int_{\Gamma} \bigg(\frac{1}{2}\bn_{\zeta_\epsilon} \cdot \bn^\top (\mathcal P^N((\partial_t\eta^N(t))_\kappa)  \partial_t{\zeta_\epsilon} \,\partial_t\eta^N  \,
 \vert\det(\naby\bm{\varphi}_{\zeta_\epsilon})\vert
\bigg)\dy\dd r
 \\\nonumber
&+
\int_t^s
 \int_{\mathcal{O}_{\zeta_\epsilon}}\bu^N\cdot \partial_r \mathscr F_N ^{\zeta_\epsilon(r)}(\mathcal P^N((\partial_t\eta^N(t))_\kappa))
\dx\dd r  \\\nonumber&-\int_t^s \int_{\mathcal{O}_{\zeta_\epsilon}}
\frac{1}{2}((\bv_\epsilon\cdot\nabx)\bu^N)\cdot \mathscr F_N ^{\zeta_\epsilon}(\mathcal P^N((\partial_r\eta^N(t))_\kappa))  \dx\dd r 
\\\nonumber
&+
\int_t^s \int_{\mathcal{O}_{\zeta_\epsilon}}
\frac{1}{2}((\bv_\epsilon\cdot\nabx)\mathscr F_N ^{\zeta_\epsilon}(\mathcal P^N((\partial_t\eta^N(t))_\kappa))) \cdot \bu^N
\dx\dd r  \\\nonumber&-\int_t^s \int_{\mathcal{O}_{\zeta_\epsilon}}\nabx \bu^N:\nabx \mathscr F_N ^{\zeta_\epsilon}(\mathcal P^N((\partial_t\eta^N(t))_\kappa))  \dx\dd r  
\\\nonumber&+
\frac{1}{2}
\int_t^s 
 \int_{\Gamma}  ((\bsigma\cdot\naby)(\bsigma\cdot\naby)\partial_t \eta^N)  (\mathcal P^N((\partial_t\eta^N(t))_\kappa) \dy\dd r
\\%\label{galerkinweak1}
&+
\int_t^s
  \int_{\Gamma}  ((\bsigma\cdot\naby)\partial_t \eta^N) (\mathcal P^N((\partial_t\eta^N(t))_\kappa)\dy\dd B_r.
\end{align} 
Due to \eqref{eq:molly}, \eqref{eq:projection} and
\eqref{eq:FN} and the uniform estimates \eqref{GalerkinConvEta1Np}--\eqref{GalerkinConvEta2Np},
it is clear that all terms can be estimated by $c(\kappa, \epsilon)|t-s|^{1/2}(1+\|\zeta\|_{C(\overline I\times \Gamma)}+\|\nabx\mathbf u^N\|_{L^2(I\times \mathcal O_{\zeta_\epsilon})})$
except for the last one.
Here we have by Markov's inequality and It\^o-isometry (using also $\divy \bsigma=0$)
\begin{align*}
\mathbb P\bigg(|t-s|^{-1/2}\bigg|&\int_t^s\int_{\Gamma}  
( (\bsigma\cdot\naby)\partial_t \eta^N ) 
\mathcal P^N(\partial_t\eta^N(t))_\kappa \dy\dd B_r\bigg|\geq K\bigg)
\\
&\leq\,\frac{1}{K^2}\mathbb E\bigg||t-s|^{-1/2}\int_t^s\int_{\Gamma}  
 \partial_t \eta^N ((\bsigma\cdot\naby)\mathcal P^N(\partial_t\eta^N(t))_\kappa) \dy\dd B_r\bigg|^2
\\
&=\,\frac{1}{K^2}\mathbb E\bigg[|t-s|^{-1}\int_t^s\bigg|\int_{\Gamma}  
 \partial_t \eta^N ((\bsigma\cdot\naby)\mathcal P^N(\partial_t\eta^N(t))_\kappa)\dy\bigg|^2\,\dd r\bigg]
\\
&\lesssim \,\frac{1}{K^2}\mathbb E\bigg[|t-s|^{-1}\int_t^s\|\partial_t\eta^N\|_{L^2(\Gamma)}
\|\naby\mathcal P^N(\partial_t\eta^N(t))_\kappa\|_{L^2(\Gamma)}\,\dd r\bigg]
\\
&\lesssim \,\frac{\kappa^{-1}}{K^2}\mathbb E\bigg[|t-s|^{-1}
\int_t^s\|\partial_t\eta^N(r)\|_{L^2(\Gamma)}\|\partial_t\eta^N(t)\|_{L^2(\Gamma)}\,\dd r\bigg]\lesssim \,\frac{\kappa^{-1}}{K^2}
\end{align*}
using \eqref{GalerkinConvEtaNp}.
This can be made arbitrarily small for $K$ large.
Now we set $Z:=W^{s_0,2}(\Gamma)\times W^{s_0,2}(\mathcal O\cup S_\alpha)$, where $s_0\in(s,\frac{1}{4})$. Noticing that property (d) from Theorem \ref{thm:auba} 
follows by the usual compactness in (negative) Sobolev spaces, we conclude tightness of the law of 
$((\partial_t\eta^N,\mathbb I_{\mathcal O_{\zeta_\epsilon}}\mathbf u^N),
(\partial_t\eta^N,\mathscr F_N^{\zeta_\epsilon}\partial_t\eta^N))$
on $(L^2(I;X'\times X),\tau_\sharp)$.

The tightness for $((0,\mathbb I_{\mathcal O_{\zeta_\epsilon}}\mathbf u^N),
(0,\mathbf u^N-\mathscr F_N^{\zeta_\epsilon}\partial_t\eta^N))$ follows along the same line, the only difference
being the regularisation of
$$f_N:=\big(0,\mathbf u^N-\mathscr F_N^{\zeta_\epsilon}\partial_t\eta^N\big).$$
While $\mathscr F_N^{\zeta_\epsilon}\partial_t\eta^N$ can be replaced by
$\mathscr F_N ^{\zeta_\epsilon(s)}(\mathcal P^N((\partial_t\eta^N(t))_\kappa))$
as above, we need to regularise $\mathbf u^N$ accordingly to preserve the homogeneous boundary conditions
of $f_N$. Recalling the definition of $\bm w_i$ from \eqref{eq:bwi}
we define $\mathbf X_i^\kappa\in W^{1,2}_{0,\divx}(\mathcal O)$ as a spatial regularisation of $\mathbf X_i$
and $\mathbf Y_i^\kappa$ as the solution to the Stokes problem with boundary datum
$Y_i^\kappa\bn$ (which inherits the regularity of $Y_i^\kappa$). Then we set
\begin{align}\label{eq:bwikappa}
\bm{w}_i^\kappa  = \left\{
  \begin{array}{lr}
    \mathbf{X}_i^\kappa & : i \text{ even},\\
    \mathbf{Y}_i^\kappa & : i \text{ odd}.
  \end{array}
\right.
\end{align}
If $\mathbf u^N=\sum_{i=1}^N\alpha_i^N\mathcal J_{\zeta_\epsilon}\bm{w}_i$ we set
\begin{align*}
f_{N,\kappa}=\bigg(0,\sum_{i=1}^N\alpha_i^N\mathcal J_{\zeta_\epsilon}\bm{w}_i^\kappa-\mathscr F_N ^{\zeta_\epsilon}(\mathcal P^N((\partial_t\eta^N(t))_\kappa))\bigg).
\end{align*}
Hence $f_{N,\kappa}$ has zero boundary conditions and thus only the fluid part is seen in
the expression $\langle g_N(t)-g_N(s),f_{N,\kappa}(t)\rangle$ 
(where $g_N:=(\partial_t\eta^N,\mathbb I_{\mathcal O_{\zeta_\epsilon}}\mathbf u^N)$
 as above). In particular, the noise is not seen and we conclude 
\begin{align*}
\langle g_N(t)-g_N(s),f_{N,\kappa}(t)\rangle\leq\, c(\kappa,\epsilon)|t-s|^{1/6}(1+\|\zeta\|_{C(\overline I\times\Gamma)}+\|\nabx\mathbf {u}^N\|_{L^2(I\times \mathcal O_{\zeta_\epsilon})})
\end{align*}
obtaining again the claimed tightness.
\end{proof}

\subsection{Stochastic compactness}\label{sec:compN}
With the bounds from \eqref{GalerkinConvEta1Np}--\eqref{GalerkinConvEta2Np} in hand, we wish to obtain compactness. For this, we define the path space
\begin{align*}
\chi= \chi_{B} 
\times\chi_{{\bu}}^2\times\chi_{\nabla{\bu}}\times\chi_{\eta}\times\chi_\zeta\times \chi_{f,g}^2
\end{align*}
where
\begin{align*}
&
\chi_{B}=C(\overline{I}), 
\quad \chi_{\bu}= 
\big(L^2(I;L^2(\mathcal{O}\cup S_{\alpha})),w\big),\quad \chi_{\nabla{\bu}}= 
\big(L^2(I;L^2(\mathcal{O}\cup S_{\alpha})),w\big),
\\&\chi_{\eta}
=\big(W^{1,\infty}(I;L^{2}(\Gamma)),w^\ast\big)\cap
\big(L^\infty(I;W^{3,2}(\Gamma)),w^\ast\big)\cap \big(W^{1,2}(I;W^{2,2}(\Gamma )),w\big),
\\&\chi_{\zeta}=C(\overline I\times \Gamma),\quad 
\chi_{f,g}=\big(L^2(I;X'\times X),\tau_\sharp\big).
\end{align*}
%with $s\in(0,\frac{1}{2})$. 
From \eqref{GalerkinConvEta1Np}--\eqref{GalerkinConvEta2Np}
(together with Alaoglu-Bourbaki Theorem) we obtain the following.
\begin{lemma}\label{lem:tight} For fixed $\epsilon>0$, the joint law
{\small
\begin{align*}
\bigg\{ \mathcal{L}\bigg[B_t,\mathbb I_{\mathcal O_{\zeta_\epsilon}}\mathbf u^N,\bv,
\mathbb I_{\mathcal O_{\zeta_\epsilon}}\nabx \mathbf u^N, \eta^N,\zeta,\begin{pmatrix}((\partial_t\eta^N,\mathbb I_{\mathcal O_{\zeta_\epsilon}}\mathbf u^N),
(\partial_t\eta^N,\mathscr F^{\zeta_\epsilon}_N\partial_t\eta^N))
\\
((\partial_t\eta^N,\mathbb I_{\mathcal O_{\zeta_\epsilon}}\mathbf u^N),(0,\mathbf u^N-\mathscr F^{\zeta_\epsilon}_N\partial_t\eta^N))\end{pmatrix}\bigg] \bigg];\, N\in \mathbb{N} \bigg\}
\end{align*}
}
is tight on $\chi$.
\end{lemma}
Now we use Jakubowski's version of the Skorokhod
representation theorem, see \cite{jakubow},
to infer the following result (we refer to \cite[Theorem A.1]{on2} for a statement 
which combines Prokhorov's and Skorokhod's theorem for quasi-Polish spaces) 
which one obtains after taking a non-relabelled subsequence.
%\begin{remark}
%If we replace \eqref{energy} by \eqref{N3a} we only get 
%\begin{align}\label{eq:apriori1'}
%\E\bigg[\|\bfu_\ep(t)\|_{L^2_x}^{2n}+\ep\int_0^T\|\bfu_\ep\|_{L^2_x}^{2n-2}\int_{\mt}|\nabla\bfu^\ep|^2\dxt\bigg]\leq\,c(n,\Lambda,\Phi)
%\end{align}
%for a.a. $t$. This is however enough (with $n=1$) to conclude \eqref{eq:p0} (with $p=1$) and \eqref{eq:vt2}. On the other hand we can use Proposition \ref{cor:B3FR} to estimate for any $R>0$
%\begin{align*}
%R\,\p\bigg(\sup_{0<t<T}&\|\bfu_\ep(t)\|_{L^2_x}^{2n}+\ep\int_0^T\|\bfu_\ep\|_{L^2_x}^{2n-2}\int_{\mt}|\nabla\bfu^\ep|^2\dxt\geq R\bigg)\\&\leq \,2\E\bigg[\|\bfu_\ep(0)\|_{L^2_x}^{2n}+\|\bfu_\ep(T)\|_{L^2_x}^{2n}+\ep\int_0^T\|\bfu_\ep\|_{L^2_x}^{2n-2}\int_{\mt}|\nabla\bfu^\ep|^2\dxt\bigg]
%\end{align*}
%which is bounded
%by \eqref{eq:apriori1'}. This gives again the required compactness.
%\end{remark}

\begin{proposition}\label{prop:skorokhodN}
There exists a complete probability space $(\tilde\Omega,\tilde{\mathfrak F},\tilde{\mathbb P})$ 
with $\chi$-valued random variables
{\small
\begin{align*}
\tilde{\boldsymbol{\Theta}}^N:=\bigg[\tilde B_t^N,\mathbb I_{\mathcal O_{\tilde\zeta^N_\epsilon}}\tilde{\mathbf u}^N,\tilde\bv^N,
\mathbb I_{\mathcal O_{\tilde\zeta^N_\epsilon}}\nabx\tilde{\mathbf u}^N, \tilde\eta^N ,\tilde\zeta^N,
\begin{pmatrix}
((\partial_t\tilde\eta^N,\mathbb I_{\mathcal O_{\tilde\zeta^N_\epsilon}}\tilde{\mathbf u}^N),
(\partial_t\tilde\eta^N,\mathscr F^{\tilde\zeta^N_\epsilon}_N\partial_t\tilde\eta^N))
\\
((\partial_t\tilde\eta^N,\mathbb I_{\mathcal O_{\tilde{\zeta}^N_\epsilon}}\tilde{\mathbf u}^N),(0,\tilde{\mathbf u}
^N-\mathscr F^{\tilde\zeta^N_\epsilon}_N\partial_t\tilde\eta^N))
\end{pmatrix}
\bigg]
\end{align*}
}
for $N\in\mathbb N$ and
$$\tilde{\boldsymbol{\Theta}}:=\bigg[\tilde B_t,\mathbb I_{\mathcal O_{\tilde\zeta_\epsilon}}\tilde{\mathbf u},\tilde\bv,
\mathbb I_{\mathcal O_{\tilde\zeta_\epsilon}}\nabx\tilde{\mathbf u}, \tilde\eta,\tilde\zeta ,
\begin{pmatrix}
((\partial_t\tilde\eta,\mathbb I_{\mathcal O_{\tilde\zeta_\epsilon}}\tilde{\mathbf u}),
(\partial_t\tilde\eta,\mathscr F^{\tilde\zeta_\epsilon}\partial_t\tilde\eta))
\\
((\partial_t\tilde\eta,\mathbb I_{\mathcal O_{\tilde\zeta_\epsilon}}\tilde{\mathbf u}),(0,\tilde{\mathbf u}
-\mathscr F^{\tilde\zeta_\epsilon}\partial_t\tilde\eta))
\end{pmatrix}
\bigg]$$
such that
\begin{enumerate}
 \item[(a)] For all $n\in \mathbb N$ the law of 
$\tilde{\boldsymbol{\Theta}}^N$
on $\chi$ is given by
\begin{align*}
 \mathcal{L}\bigg[B_t,\mathbb I_{\mathcal O_{\zeta_\epsilon}}\mathbf u^N,\bv,
\mathbb I_{\mathcal O_{\zeta_\epsilon}}\nabx \mathbf u^N, \eta^N,\zeta,\begin{pmatrix}((\partial_t\eta^N,\mathbb I_{\mathcal O_{\zeta_\epsilon}}\mathbf u^N),
(\partial_t\eta^N,\mathscr F^{\zeta_\epsilon}_N\partial_t\eta^N))\\
((\partial_t\eta^N,\mathbb I_{\mathcal O_{\zeta_\epsilon}} \mathbf u^N),(0,{\mathbf u}^N
-\mathscr F^{\zeta_\epsilon}_N\partial_t\eta^N))\end{pmatrix}\bigg]
\end{align*}
 \item[(b)] $\tilde{\boldsymbol{\Theta}}^N$ converges $\,\tilde{\mathbb P}$-almost surely to 
$\tilde{\boldsymbol{\Theta}}$
 in the topology of $\chi$, i.e.
\begin{align} \label{wWS116x} 
\begin{aligned}
\tilde B_t^N&\rightarrow \tilde B_t \quad \mbox{in}\quad C(\overline I) \ \tilde{\mathbb P}\mbox{-a.s.}, 
\\
\mathbb I_{\mathcal O_{\tilde\zeta_\epsilon^N}}\tilde{\mathbf u}^N,\,\tilde\bv^N&\weakto \mathbb I_{\mathcal O_{\tilde\zeta_\epsilon}}\tilde{\mathbf u} ,\,\tilde\bv\quad \mbox{in}\quad L^2(I;L^{2}(\mathcal O\cup S_{\alpha}))\ \tilde{\mathbb P}\mbox{-a.s.}, 
\\
\mathbb I_{\mathcal O_{\zeta^N_\epsilon}}\nabx\tilde{\mathbf u}^N&\weakto  \mathbb I_{\mathcal O_{\tilde\zeta_\epsilon}}\nabx\tilde{\mathbf u} \quad \mbox{in}\quad L^2(I;L^{2}(\mathcal O\cup S_{\alpha}))\ \tilde{\mathbb P}\mbox{-a.s.}, 
\\
\tilde \eta^N&\weakto^\ast \tilde \eta\quad \mbox{in}\quad L^\infty(I;W^{3,2}(\Gamma)) \ \tilde{\mathbb P}\mbox{-a.s.}, 
\\
\tilde \eta^N&\weakto^\ast \tilde \eta\quad \mbox{in}\quad W^{1,\infty}(I;L^{2}(\Gamma)) \ \tilde{\mathbb P}\mbox{-a.s.}, 
\\
\tilde \eta^N&\weakto \tilde \eta\quad \mbox{in}\quad W^{1,2}(I;W^{2,2}(\Gamma)) \ \tilde{\mathbb P}\mbox{-a.s.},\\
\tilde\zeta^N&\rightarrow \tilde\zeta \quad \mbox{in}\quad C(\overline I\times \Gamma) \ \tilde{\mathbb P}\mbox{-a.s.},
\end{aligned}
\end{align}
as well as (recalling the definition of $\tau_\sharp$ from \eqref{eq:tau})
\begin{align}\label{wWS116B}\begin{aligned}
\int_{I}\int_{\mathcal O\cup S_{\alpha}}&\mathbb I_{\mathcal O_{\tilde\zeta_\epsilon^N}}
\bu^N\cdot\mathscr F^{\tilde\zeta^N_\epsilon}_N\partial_t\tilde\eta^N\dxt+\int_{I}\int_\Gamma|\partial_t\tilde\eta^N|^2\,\dy\dt
\\
&\longrightarrow \int_{I}\int_{\mathcal O\cup S_{\alpha}}\mathbb I_{\mathcal O_{\tilde\zeta_\epsilon}}
\tilde\bfu\cdot\mathscr F^{\tilde\zeta_\epsilon}\partial_t\tilde\eta\dxt+\int_{I}\int_\Gamma|\partial_t\tilde\eta|^2\,\dy\dt
\end{aligned}
\end{align}
and
\begin{align}\label{wWS116C}\begin{aligned}
\int_{I}\int_{\mathcal O\cup S_{\alpha}}&
\mathbb I_{\mathcal O_{\tilde\zeta_\epsilon^N}}
\tilde\bu^N\cdot(\tilde\bfu^N-\mathscr F^{\tilde\zeta^N_\epsilon}_N\partial_t\tilde\eta^N)\dxt\\
&\longrightarrow \int_{I}\int_{\mathcal O\cup S_{\alpha}}
\mathbb I_{\mathcal O_{\tilde\zeta_\epsilon}}
\tilde\bfu\cdot(\tilde\bfu-\mathscr F^{\tilde\zeta_\epsilon}\partial_t\tilde\eta)\dxt
\end{aligned}
\end{align}
$\tilde{\mathbb P}$-a.s.
\end{enumerate}
\end{proposition}
Now we introduce the filtration on the new probability space, 
which ensures the correct measurabilities of the new random variables.
%We denote by $\bfr_t$ the operator of restriction to the interval $[0,t]$ acting on various path spaces. In particular, if $X$ stands for one of the path spaces $L^r_{\mathrm loc}(0,\infty;L^r(\mathcal{O}))$ or $C_{\mathrm loc}([0,\infty),\mathfrak U_0)$ and $t\in[0,T]$, we define
%\begin{align}\label{restr}
%\bfr_t:X\rightarrow X|_{[0,t]},\quad f\mapsto f|_{[0,t]}.
%\end{align}
%Clearly, $ \bfr_t$ is a continuous mapping.
Let $(\tilde{\mathfrak F}_t)_{t\geq0}$ and $(\tilde{\mathfrak F}_t^{N})_{t\geq0}$ be the $\tilde{\mathbb P}$-augmented
canonical filtration on the variables $\tilde{\boldsymbol{\Theta}}$ and
$\tilde{\boldsymbol{\Theta}}^N $, respectively, that is\footnote{Some of the variables are not continuous in time, for those one can deifne $\sigma_t$ as the history of a random distribution, cf. \cite[Chapter 2.8]{BFHbook}}
%\todo{add filtrations}
\begin{align*}
\tilde{\mathfrak F}_t=\sigma\big[\sigma_t(\tilde B_t)\cup \sigma_t(\mathbb I_{\mathcal O_{\tilde\zeta_\epsilon}}\tilde{\mathbf u})\cup \sigma_t(\tilde\bv) \cup \sigma_t(
\mathbb I_{\mathcal O_{\tilde\zeta_\epsilon}}\nabx\tilde{\mathbf u}) \cup \sigma_t( \tilde\eta) \cup \sigma_t(\tilde\zeta )% \cup \sigma_t
%\begin{pmatrix}
%((\partial_t\tilde\eta,\mathbb I_{\mathcal O_{\tilde\zeta_\epsilon}}\mathbf u),
%(\partial_t\tilde\eta,\mathscr F^{\tilde\zeta_\epsilon}\partial_t\tilde\eta))
%\\
%((\partial_t\tilde\eta,\mathbb I_{\mathcal O_{\tilde\zeta}}\tilde{\mathbf u}),(\partial_t\tilde\eta,\tilde{\mathbf u}
%-\mathscr F^{\tilde\zeta_\epsilon}\partial_t\tilde\eta))\end{pmatrix}
\big]
\end{align*}
for $t\in I$ and similarly for $\tilde{\mathfrak F}_t^N$.
By \cite[Theorem 2.9.1]{BFHbook} the weak equation continuous to hold on the new probability space.
Combining \eqref{wWS116B} and \eqref{wWS116C} we have
\begin{align*}
\int_{I}\int_{\mathcal O_{\tilde\zeta^N_\epsilon}}&|\tilde\bfu^N|^2\dxt
+\int_{I}\int_\Gamma|\partial_t\tilde\eta^N|^2\,\dy\dt\\&\longrightarrow \int_{I}\int_{\mathcal O_{\tilde\zeta}}|\tilde\bfu|^2\dxt
+\int_{I}\int_\Gamma|\tilde\partial_t\eta|^2\,\dy\dt
\end{align*}
 $\tilde{\mathbb P}$-a.s. By uniform convexity of the $L^2$-norm this implies
\begin{align*}
\partial_t\tilde\eta^N&\rightarrow \partial_t\tilde\eta
 \quad \mbox{in}\quad L^2(I;L^{2}(\Gamma))\ \tilde{\mathbb P}\mbox{-a.s.}
\end{align*}
This is sufficient to pass to the limit in the weak formulation of the equations (note that all terms except for the stochastic integral can be treated 
by using \eqref{wWS116x}).

Since we have a linear problem at hand, we may now take appropriate subsequence and pass to the limit in 
\eqref{galerkinweak1}. Thus we obtain a martingale solution to \eqref{weakeps}. In order to obtain a probabilistically strong solution we prove in the following subsections pathwise uniqueness as well as convergence in probability of the original sequence.

\subsection{Pathwise uniqueness}
We are now going to prove that any martingale solution  satisfies the energy equality \eqref{eq:eneps}.
Pathwise uniqueness is then a direct consequence of the linearity of the problem.
\begin{proposition}\label{prop:unique}
Suppose that $(\eta,\bu)$ is a weak martingale solution to  \eqref{weakeps}. Then it holds $\mathbb P$-a.s.
\begin{align*}
\tfrac{1}{2}\int_{\mathcal{O}_{\zeta_\epsilon}} & |\bu(t)|^2\dx
+
\int_0^t\int_{\mathcal{O}_{\zeta_\epsilon}}|\nabx\bu|^2\dx\ds
+
\epsilon
\int_0^t\int_\Gamma
|\partial_s\Dely\eta|^2\dy\ds
\\&+
\int_\Gamma\Big(\tfrac{1}{2}|\partial_t\eta(t)|^2
+
\tfrac{1}{2}
|\Dely\eta(t)|^2
+\epsilon |\naby^3\eta(t)|^2\Big)\dy
\\
&
= \tfrac{1}{2}\int_{\mathcal O_{\zeta_\epsilon(0)} }|{\bu}_0|^2\dx
+
\int_\Gamma\Big(\tfrac{1}{2}|\eta_1|^2
+
\tfrac{1}{2}
|\Delx\eta_0|^2+ \epsilon|\nabx^3\eta_0|\Big)\dy
\end{align*} 
for a.a. $t\in I$.
\end{proposition}
%D: The following requires $\partial_t\eta\in W^{2,2}(\Gamma)$. Hence I suggest adding the term $\epsilon\partial_t\Dely^2\eta$ in the regularised shell equation.  
\begin{proof}
We rewrite \eqref{weakeps} as an equation
for $(\partial_t\bar\eta,\bar\bv)$, where $\bar\eta:=\int_0^t\iota_{\zeta_\epsilon}\partial_t\eta\,\dd s$ and $$\bar\bv:=\nabx\bfPsi_{\zeta_\epsilon}^\top(|\mathrm{det}(\nabx \bfPsi_{\zeta_\epsilon})|^{-1}\nabx\bfPsi_{\zeta_\epsilon})\bar\bu$$ with $\bar\bu:=\mathcal J_{\zeta_\epsilon}^{-1}\bu$ which reads as 
\begin{align}
\dd  \bigg\langle&\begin{pmatrix}\partial_t\bar\eta\\\bar\bv
\end{pmatrix}  
\cdot \begin{pmatrix} \phi
\nonumber
\\
\bm{\phi}
\end{pmatrix}
\bigg\rangle_{\mathbb H}
\nonumber
\\&=
\mathscr L  
\begin{pmatrix} 
\phi \\
\bm{\phi}
\end{pmatrix} +
\int_{\Gamma}  
( (\bsigma\cdot\naby)\partial_t \eta ) 
\iota_{\zeta_{\varepsilon}}\phi \dy\dd B_t
\nonumber
\\
&:=\int_{\Gamma} \big(\partial_t\eta \, \partial_t(\iota_{\zeta_{\epsilon}})\phi
 -
 \Dely \iota_{\zeta_\epsilon} \phi \Dely\eta
 -
\epsilon
\Dely \iota_{\zeta_{\epsilon}}\phi \,\partial_t\Dely\eta
\big)\dy\dt
\nonumber
\\&
-\epsilon
\int_{\Gamma}
\naby^3 \iota_{\zeta_{\epsilon}}\phi :\naby^3\eta
\dy\dt
+
\int_{\Gamma} \bigg(\frac{1}{2}\bn_{\zeta_\epsilon} \cdot \bn^\top \iota_{\zeta_\epsilon} \phi \,  \partial_t{\zeta_\epsilon} \,\partial_t\eta  \,
 \vert\det(\naby\bm{\varphi}_{\zeta_\epsilon})\vert
\bigg)\dy\dt
\nonumber
\\
&+
 \int_{\mathcal{O}_{\zeta_\epsilon (t)}}\Big(  \bu\cdot \partial_t (\mathcal{J}_{\zeta_\epsilon (t)}) \bm{\phi} 
-
\frac{1}{2}((\bv_\epsilon\cdot\nabx)\bu)\cdot (\mathcal{J}_{\zeta_\epsilon(t)}\bfphi )  \Big) \dx\dt 
\nonumber
\\
&+
 \int_{\mathcal{O}_{\zeta_\epsilon (t)}}\Big(  
\frac{1}{2}((\bv_\epsilon\cdot\nabx)\mathcal{J}_{\zeta_\epsilon(t)}\bm{\phi}) \cdot \bu
-\nabx \bu:\nabx (\mathcal{J}_{\zeta_\varepsilon (t)}\bm{\phi}) \Big) \dx\dt  
\nonumber
 \\
&
 +
 \frac{1}{2}
\int_{\Gamma} 
((\bsigma\cdot\naby)(\bsigma\cdot\naby)\partial_t \eta ) \iota_{\zeta_{\varepsilon}}\phi
\dy\dt 
- 
\int_{\Gamma}  
( (\bsigma\cdot\naby)\partial_t \eta ) 
\iota_{\zeta_{\varepsilon}}\phi \dy\dd B_t
\label{weakeps'}
\end{align}
where 
\begin{align*}
\mathbb H:=\big\{(\xi,\mathbf{w})\in W^{2,2}(\Gamma)\times W^{1,2}_{\divx}(\mathcal O):\,\mathbf{w} \circ\bm{\varphi}=\xi\bn\,\,\text{on}\,\Gamma \big\}.
\end{align*}
Note that $\mathbb H$ is a Hilbert space (as closed subset of the Hilbert space $W^{2,2}(\Gamma)\times W^{1,2}_{\divx}(\mathcal O)$) and $\mathscr L\in\mathbb H'$. Hence we get from It\^{o}'s formula in Hilbert spaces (see  \cite[Theorem 4.17]{DaPr})
applied to the mapping
\begin{align*}
t\mapsto \tfrac{1}{2}\int_{\mathcal O}\bar\bv\cdot
|\mathrm{det}(\nabx \bfPsi_{\zeta_\epsilon})|(\nabx\bfPsi_{\zeta_\epsilon}^\top
\nabx\bfPsi_{\zeta_\epsilon}^\top)^{-1}
\bar\bv\dx+\tfrac{1}{2}\int_\Gamma\Big|\tfrac{\partial_t\bar\eta}{\iota_{\zeta_\epsilon}}\Big|^2\,\dy,
\end{align*}
noticing that $|\mathrm{det}(\nabx \bfPsi_{\zeta_\epsilon})|(\nabx\bfPsi_{\zeta_\epsilon}^\top
\nabx\bfPsi_{\zeta_\epsilon}^\top)^{-1}
\bar\bv=\bar\bu$,
\begin{align*}
\tfrac{1}{2}\int_{\mathcal O_{\zeta_\epsilon(t)}}&|\bu|^2\dx
+\tfrac{1}{2}\int_\Gamma|\partial_t\eta|^2\dy
\\&=
\tfrac{1}{2}\int_{\mathcal O}\bar\bv\cdot
|\mathrm{det}(\nabx \bfPsi_{\zeta_\epsilon})|(\nabx\bfPsi_{\zeta_\epsilon}^\top
\nabx\bfPsi_{\zeta_\epsilon}^\top)^{-1}
\bar\bv\dx
+
\tfrac{1}{2}\int_\Gamma\Big|\tfrac{\partial_t\bar\eta}{\iota_{\zeta_\epsilon}}\Big|^2\dy
\\
&=
\tfrac{1}{2}\int_{\mathcal O}\bar\bv(0)\cdot\bar\bu(0)\dx+\tfrac{1}{2}\int_\Gamma|\eta_1|^2\,\dy+\int_0^t\mathscr L\begin{pmatrix}\iota_{\zeta_\epsilon}^{-2}\partial_t\bar\eta\\\bar\bu
\end{pmatrix} \,\dd s
\\
&+\int_0^t\int_{\Gamma}  
( (\bsigma\cdot\naby)\partial_t \eta ) 
\partial_t\eta\dy\dd B_s+\tfrac{1}{2}\int_0^t\int_{\Gamma}  
|(\bsigma\cdot\naby)\partial_t \eta |^2\dy\,\dd s\\
&+\tfrac{1}{2}\int_{\mathcal O}\bar\bv\cdot\partial_t\big(|\mathrm{det}(\nabx \bfPsi_{\zeta_\epsilon})|(\nabx\bfPsi_{\zeta_\epsilon}^\top
\nabx\bfPsi_{\zeta_\epsilon}^\top)^{-1}
\big)\bar\bv\dx-\int_\Gamma\tfrac{|\partial_t\bar\eta|^2}{\iota_{\zeta_\epsilon}^3}\partial_t\iota_{\zeta_\epsilon}\,\dy.
\end{align*}
Noticing various cancellations such as
\begin{align*}
\int_{\mathcal O}\bar\bv\cdot\partial_t\big(|\mathrm{det}(\nabx \bfPsi_{\zeta_\epsilon})|(\nabx\bfPsi_{\zeta_\epsilon}^\top
\nabx\bfPsi_{\zeta_\epsilon}^\top)^{-1}
\big)\bar\bv\dx&= \int_{\mathcal{O}_{\zeta_\epsilon (t)}} \bu\cdot (\partial_t \mathcal{J}_{\zeta_\epsilon (t)} \bar\bu)\dx 
\\
\int_\Gamma\tfrac{|\partial_t\bar\eta|^2}{\iota_{\zeta_\epsilon}^3}\partial_t\iota_\zeta\,\dy&=\int_{\Gamma}   \partial_t\eta \, (\partial_t\iota_{\zeta_{\varepsilon}})\iota_{\zeta_\epsilon}^{-2}\partial_t\bar\eta\dy
\end{align*}
we obtain
\begin{align*}
\tfrac{1}{2}\int_{\mathcal O_{\zeta_\epsilon(t)}}|\bu|^2\dx+\tfrac{1}{2}\int_\Gamma|\partial_t\eta|^2\dy=&\tfrac{1}{2}\int_{\mathcal O_{\zeta_\epsilon(0)}}|\bu_0|^2\dx+\tfrac{1}{2}\int_\Gamma|\eta_1|^2\,\dy\\
&-
\int_{\Gamma} \Dely \iota_{\zeta_\varepsilon}\partial_t\bar\eta \,\Dely\eta\dy
\\&-\int_0^t\int_{\mathcal O_\zeta}\nabx \bu:\nabx (\mathcal{J}_{\zeta_\varepsilon (t)}\bar\bu)  \dx\ds .
\end{align*}
The claim follows on noticing that $\bu=\mathcal{J}_{\zeta_\varepsilon (t)}\bar\bu$ and $\Dely (\iota_{\zeta_\varepsilon}\partial_t\eta)=\partial_t\Dely\bar\eta$.
\end{proof}

\subsection{Convergence in probability}\label{sec:cip}
In order to complete the proof of Theorem \ref{thm:main'}, we make use of  \cite[Chapter 2, Theorem 2.10.3]{BFHbook} which is a generalization of  the Gy\"{o}ngy--Krylov characterization of convergence in probability introduced in \cite{krylov} to quasi-Polish spaces. It applies to situations where pathwise uniqueness and existence of a martingale solution are valid and allows to establish existence of a probabilistically strong solution. We consider two sequences $(N_n),(N_m)\subset \mathbb N$ diverging to infinity. Let
$${\boldsymbol{\vartheta}}^N:=\bigg[\mathbb I_{\mathcal O_{\zeta_\epsilon}}\mathbf u^N,
\mathbb I_{\mathcal O_{\zeta_\epsilon}} \nabx{\mathbf u}^N, \eta^N , 
\begin{pmatrix}
((\partial_t\eta^N,\mathbb I_{\mathcal O_{\zeta_\epsilon}}\mathbf u^N),
(\partial_t\eta^N,\mathscr F^{\zeta_\epsilon}_N\partial_t\eta^N))
\\
((\partial_t\eta^N,\mathbb I_{\mathcal O_{\zeta_\epsilon}} \mathbf u^N),(\partial_t \eta^N, \mathbf u
^N-\mathscr F^{\zeta_\epsilon}_N\partial_t\eta^N))\end{pmatrix}\bigg],$$
for $N\in\mathbb N$ and set
$${\boldsymbol{\vartheta}}^n:={\boldsymbol{\vartheta}}^{N_n},\quad{\boldsymbol{\vartheta}}^m:={\boldsymbol{\vartheta}}^{N_m}.$$
We also set
$\bu^n:=\bu^{N_n}$, $\eta^n:=\eta^{N_n}$ and $\bu^m:=\bu^{N_m},\eta^m:=\eta^{N_m}$.
We consider the collection of joint laws of $({\boldsymbol{\vartheta}}^n,{\boldsymbol{\vartheta}}^m,\bv,\zeta,B_t)$
%, where
%\begin{align*}
%\bfX^n&=(\bfv^{N_n},\bfF_p(\cdot,\ep(\bfv^{N_n}))),\quad
%\bfX^m=(\bfv^{N_n},\bfF_p(\cdot,\ep(\bfv^{N_n})),
%\end{align*}
on the extended path space
$$\chi_{\tt ext}:=\big(\chi_\bu\times\chi_{\nabla\bu}\times\chi_\eta\times\chi_{f,g}^2\big)^2\times\chi_{\bu}\times\chi_\eta\times\chi_B.$$ 
%denoted by $\mu^{n,m}$. For this purpose we define the extended path space
%$$\mathcal{X}^J=\mathcal{X}_\varrho\times\mathcal{X}_{\bfu}\times\mathcal{X}_{\varrho\bfu}\times\mathcal{X}_\varrho\times\mathcal{X}_{\bfu}\times\mathcal{X}_{\varrho\bfu}\times\mathcal{X}_W$$
%As above, denote by $\mu_W$ the law of $W$ and set $\nu^{n,m}$ to be the joint law of
%$$(\varrho_n,\bfu_n,\mathcal{P}_H(\varrho_n\bfu_n),\varrho_m,\bfu_m,\mathcal{P}_H(\varrho_m\bfu_m),W)\quad\text{on}\quad\mathcal{X}^J.$$
Similarly to Lemma \ref{lem:tight} we obtain tightness of 
$$\{\mathcal L[ {\boldsymbol{\vartheta}}^n,{\boldsymbol{\vartheta}}^m,\bv
,\zeta,B_t);\,n,m\in\mathbb N\}$$
on $\chi_{\tt ext}$.

Let us take any subsequence $({\boldsymbol{\vartheta}}^{n_k},{\boldsymbol{\vartheta}}^{m_k},\bv,\zeta,B_t)$. By the Jakubowski-Skorokhod representation theorem we infer for a non-relabelled subsequence the existence of a probability space $(\bar{\Omega},\bar{\mathfrak F},\bar{\mathbb P})$ with a sequence of random variables  $(\hat{\boldsymbol{\vartheta}}^{n_k},\check{\boldsymbol{\vartheta}}^{m_k},\overline\bv^{k},\overline\zeta^{k},\overline B_t^{k})$
% with
%\begin{align*}
%\hat\bfX^{n_k}&=(\hat\bfv^{{n_k,}},\hat\bfF^{{n_k}}),\quad k\in\mn,\\
%\hat\bfX^{m_k}&=(\check\bfv^{{m_k,}},\check\bfF^{{m_k}}),\quad k\in\mn,
%\end{align*}
converging almost surely in $\chi_{\tt ext}$ to a random variable $(\hat{\boldsymbol{\vartheta}},\check{\boldsymbol{\vartheta}},\overline\bv,\overline\zeta,\overline B_t)$.
%with
%\begin{align*}
%\hat\bfX&=(\hat\bfv,\hat\bfF),\quad
%\check\bfX=(\hat\bfv,\hat\bfF).
%\end{align*}
%As before it follows that
%\begin{align}\label{eq:911}
%\hat\bfF=\bfF_{\bar p}(\cdot,\ep(\hat\bfv)),\quad \check\bfF=\bfF_{\bar p}(\cdot,\ep(\check\bfv)).
%\end{align}
Moreover,
$$\mathcal L[\hat{\boldsymbol{\vartheta}}^{n_k},\check{\boldsymbol{\vartheta}}^{m_k},\overline\bv^{k},\overline\zeta^{k},\overline B_t^{k}]=\mathcal L[{\boldsymbol{\vartheta}}^{n_k},{\boldsymbol{\vartheta}}^{m_k},\bv^k,\zeta^k,B_t^k]$$
on $\chi_{\tt ext}$ for all $k\in\mathbb N$.
%
%$$\bar{\prst}\big((\hat\varrho_{n_k},\hat\bfu_{n_k},\hat\bfq_{n_k},\check\varrho_{m_k},\check\bfu_{m_k},\check\bfq_{m_k},\bar W_k)\in\,\,\cdotp\big)=\nu^{n_k,m_k}(\cdot).$$
%$$\bar{\prst}\big((\hat\varrho,\hat\bfu,\hat\bfq,(\hat\psi,\hat\bfm),\check\varrho,\check\bfu,\check\bfq,(\check\psi,\check\bfm),\bar W)\in\,\,\cdotp\big)=\nu(\cdot).$$
Observe that, in particular, $\mathcal L[\hat{\boldsymbol{\vartheta}}^{n_k},\check{\boldsymbol{\vartheta}}^{m_k},\overline\bv^{k},\overline\zeta^{k},\overline B_t^{k}]$ converges weakly to the measure $\mathcal L[\hat{\boldsymbol{\vartheta}},\check{\boldsymbol{\vartheta}},\overline\bv,\overline\zeta,\overline B_t]$.
%$\mu$ defined by
%$$\mu(\cdot)=\bar{\prst}\big((\hat\varrho,\hat\bfu,\hat\bfq,\check\varrho,\check\bfu,\check\bfq)\in\,\,\cdotp\big).$$
As in Section \ref{sec:compN} we can
show that the limit objects are martingale solutions to \eqref{weakeps} defined on the same stochastic basis $(\bar{\Omega},\bar{\mathfrak F},(\bar{\mathfrak F}_t),\bar{\mathbb P})$, where $(\bar{\mathfrak F}_t)_{t\geq0}$ is the $\bar{\mathbb P}$-augmented canonical filtration of $(\hat{\boldsymbol{\vartheta}},\check{\boldsymbol{\vartheta}},\overline\bv,\overline\zeta,\overline B_t)$.
%In order to verify the assumption of Theorem \ref{diagonal} 
We employ the pathwise uniqueness result from Proposition \ref{prop:unique}. Therefore, the solutions $(\hat\eta,\hat\bu)$ and $(\check\eta,\check\bu)$ coincide $\bar{\mathbb P}$-a.s. and we have
{\small
\begin{align*}
\mathcal L&[\hat{\boldsymbol{\vartheta}},\check{\boldsymbol{\vartheta}},\overline\bv,\overline\zeta,\overline B_t]\Big(({\boldsymbol{\vartheta}}_1,{\boldsymbol{\vartheta}}_2,\bv,\zeta,B_t)\in\chi_{\tt ext}:\;(\eta_1,\bu_1)=(\eta_2,\bu_2)\Big)
%&\quad=\bar{\prst}\Big((\hat\bfv,\hat\bfF)=(\check\bfv,\check\bfF)\Big)
=\bar{\mathbb P}((\hat\eta,\hat\bu)=(\check\eta,\check\bu))=1.
\end{align*}
}
Now, we have all in hand to apply the Gy\"ongy--Krylov theorem from \cite[Chapter 2, Theorem 2.10.3]{BFHbook}. 
It implies that the original sequence ${\boldsymbol{\vartheta}}^N$ defined on the initial probability space $(\Omega,\mathfrak F,\mathbb P)$ converges in probability in the topology of $\chi_\bu\times\chi_{\nabla\bu}\times\chi_\eta\times\chi_{f,g}^2$ to the random variable
$${\boldsymbol{\vartheta}}:=\bigg[\mathbb I_{\mathcal O_{\zeta_\epsilon}} \mathbf u,
\mathbb I_{\mathcal O_{\zeta_\epsilon}}\nabx{\mathbf u}, \eta,
\begin{pmatrix}
((\partial_t\eta,\mathbb I_{\mathcal O_{\zeta_\epsilon}}\mathbf u),
(\partial_t\eta,\mathscr F^{\zeta_\epsilon}\partial_t\eta))
\\
((\partial_t\eta,\mathbb I_{\mathcal O_{\zeta_\epsilon}} \mathbf u),(0, \mathbf u -\mathscr F^{\zeta_\epsilon}\partial_t\eta))\end{pmatrix}\bigg].$$
Therefore, we finally deduce that $(\eta,\bu)$ is a probabilistically strong solution to \eqref{weakeps}. This completes the proof of Theorem \ref{thm:mainveps}.

\section{The nonlinear regularised problem}
\label{sec:regul}
The aim of this section is to obtain a solution of the regularised problem thus completing the proof of Theorem \ref{thm:main'}. This is done via a fixed point argument for which the main point is proving compactness of the mapping $(\zeta,\bv)\mapsto (\eta,\bu)$. Given a bounded sequence $(\zeta_n,\bv_n)$
 this is achieved in three steps:
\begin{itemize}
\item We first need to establish tightness of the probability laws, where the main difficulty arises for the velocity field.
\item We apply the stochastic compactness method based on Jakubowski's extension of the Skorokhod representation theorem. This yields a.s. convergence on a new probability space.
\item We apply a Gy\"ongy-Krylov type argument to obtain convergence in probability of the original sequence on the original probability space. This requires the pathwise uniqueness from Proposition \ref{prop:unique}. 
\end{itemize}
Suppose there is a sequence of processes $(\zeta^n,\bv^n)$ which are $(\mathfrak F_t)$-progressively measurable and bounded in
$$L^p(\Omega;C(\overline I\times\Gamma))\times L^p(\Omega;L^2(I;L^2(\mathcal O\cup S_\alpha)))$$
for some sufficiently large $p$. Now apply Theorem \ref{thm:mainveps} yielding a sequence $(\eta^n,\bu^n)$ of solutions
to \eqref{weakeps}. By the energy equality from Proposition \ref{prop:unique} 
 we obtain
\begin{align}
&\sup_I\Vert \eta^n \Vert_{W^{2,2}(\Gamma )}^{2} 
+
\epsilon\sup_I\Vert \eta^n \Vert_{W^{3,2}(\Gamma )}^{2} \lesssim1,\label{GalerkinConvEta1np}
\\
&\sup_I\Vert\partial_t\eta^n \Vert_{L^2(\Gamma )}^{2}\lesssim 1,\label{GalerkinConvEtanp}
\\
&\epsilon\int_I\Vert\partial_t\eta^n \Vert_{W^{2,2}(\Gamma )}^{2}\dt\lesssim 1,
\\
&\sup_I\Vert\bu^n \Vert_{L^2(\mathcal{O}_{\zeta^n_\epsilon})}^{2}\lesssim 1,
\label{GalerkinVolnp1}
\\
&\int_I\Vert\nabx \bu^n \Vert_{L^{2}(\mathcal{O}_{\zeta^n_\epsilon})}^2\dt\lesssim1. \label{GalerkinVolnp2}
\end{align}
In addition, for any $s\in(0,\frac{1}{2})$, it follows from $\bu^n\circ\bm{\varphi}_{\zeta_\epsilon^n } =\bn\partial_t\eta^n $, \eqref{GalerkinVolnp2} and the trace theorem that
\begin{equation}
\begin{aligned}
\label{GalerkinConvEta2np}
\int_I\Vert \partial_t\eta^n \Vert_{W^{s,2}(\Gamma)}^2\dt
\lesssim 1
\end{aligned}
\end{equation}
holds. 

\subsection{Tightness of the velocity sequence}
\label{ssec:compv}
The effort of this subsection is to prove tightness of the law of
$\bu^n$. Similarly to Lemma \ref{lemma:fg} we obtain the following result.
\begin{lemma}\label{lemma:fg'}
The laws of
\begin{align*}
((\partial_t\eta^n,\mathbb I_{\mathcal O_{\zeta^n_\epsilon}}\mathbf u^n),
(\partial_t\eta^n,\mathscr F^{\zeta^n_\epsilon}\partial_t\eta^n))\quad\text{and}\quad
((\partial_t\eta^n,\mathbb I_{\mathcal O_{\zeta_\epsilon^n}}\mathbf u^n),
(0,\mathbf u^n-\mathscr F^{\zeta_\epsilon^n}\partial_t\eta^n))
\end{align*}
on $(L^2(I;X'\times X),\tau_\sharp)$ are tight.
\end{lemma}
\begin{proof}
As in Lemma \ref{lemma:fg} we must verify the assumptions from Theorem \ref{thm:auba}. First of all, boundedness in $L^2(I;X'\times X)$ follows from
By \eqref{GalerkinConvEta1np}--\eqref{GalerkinConvEta2np} and the properties of the extension $\mathscr F^{\zeta_n}$ from Proposition \ref{prop:musc}.
For $(b)$  
%we observe that we may assume that
% a regularizer $b\mapsto (b)_\kappa$ exists such that for any $s,a\in \R$
%\begin{align}
%\label{eq:molly}\|b-(b)_\kappa\|_{W^{a,2}(\Gamma)}\leq 
%\,c\kappa^{s-a}\|b\|_{W^{s,2}(\Gamma)},\quad b\in W^{s,2}(\Gamma).
%\end{align}
%The estimate is well-known for $a,s\in\mathbb N_0$, while the general case follows by interpolation and duality. 
%Next 
we consider again the regularisation of $\partial_t\eta^n$ with parameter $\kappa>0$
 and set for $n\in\mathbb N$
\[
f_{n,\kappa}(t):=((\partial_t\eta^n(t))_\kappa,
\mathcal E_\kappa^{\zeta_\epsilon^n}(\partial_t\eta^n(t))),
\]
where $\mathcal E_\kappa^{\zeta_\epsilon^n}$ is given in Corollary \ref{cor:3.5}.
We find by the continuity of the mollification operator from \eqref{eq:molly} and the continuity of $\mathcal E_\kappa^{\zeta_\epsilon^n}$ from 
Corollary \ref{cor:3.5} that for a.e.\ $t\in (0,T)$ and $s<1/2$
 \begin{align}\label{eq:fNkappa}
 \|f_{n,\kappa}-f_n\|_{L^2(\Gamma)\times L^{2}(\mathcal O\cup S_{\alpha})}
\leq \,c\kappa^{s}\|\partial_t\eta^n\|_{W^{s,2}(\Gamma)},
 \end{align}
 which can be made arbitrarily small in $L^2$ by choosing $\kappa$ appropriately, 
cf. \eqref{GalerkinConvEta2np}.
% This allows to deduce the $L^1(I;Z)$ bound (for $s=s_y$) as well as the convergence property (choosing $s=s_x$).
Similarly, we have
 \begin{align*}\label{eq:fNkappa}
 \|f_{n,\kappa}\|_{W^{1,2}(\Gamma)\times W^{1,2}(\Omega\cup S_{\alpha})}
\leq \,c\kappa^{-1}\|\partial_t\eta^n\|_{L^{2}(\Gamma)}.
 \end{align*}
Clearly, if $f_N\rightharpoonup f$ in $L^2(I;X)$ then
 we can deduce a converging subsequence such that $f_{N,\kappa}\weakto f_{\kappa}$ (for some $ f_{\kappa}$) in 
$L^2(I;X)$ for any $\kappa>0$,
 which implies (b).

 For (c) we have to control
$\langle g_n(t)-g_n(s),f_{n,\kappa}(t) \rangle$, 
where $g_n:=(\partial_t\eta^n,\mathbb I_{\mathcal O_{\zeta^n_\epsilon}}\mathbf u^n)$
and hence decompose
 \begin{align*}
&\langle g_n(t)-g_n(s),f_{n,\kappa}(t)\rangle
\\
&=\big(\langle g_n(t),(\partial_t\eta^n(t))_\kappa,
\mathcal E_\kappa^{\zeta^n_\epsilon(t)}(\partial_t\eta^n(t)))\rangle
\\
&\quad-\langle g_n(s),(\partial_t\eta_n(t))_\kappa,\mathcal E_\kappa^{\zeta_\epsilon^n(s)}
(\partial_t\eta^n(t))\rangle \big)
\\
&\quad +\langle g_N(s),(0,\mathcal E_\kappa^{\zeta^n_\epsilon(t)}(\partial_t\eta^n(t)))
-\mathcal E_\kappa ^{\zeta^n_\epsilon(s)}(\partial_t\eta^n(t))\rangle=:(I)+(II).
 \end{align*}
We begin estimating $(II)$ to find that
 \begin{align*}
(II)& = \int_{\mathcal O_{\zeta_\epsilon^n(s)}} \bfu^n(s)
\cdot \int_t^s\partial_r \mathcal E_\kappa^{\zeta_\epsilon^n(r)}(\partial_t\eta^n(t))\,\dd r\dx
\\&
\leq\,c\|\bfu^n(s)\|_{L^2(\mathcal O_{\zeta_\epsilon^n(s)})}\|\partial_t(\zeta_\epsilon^n)_\kappa\|_{L^\infty(I\times \Gamma)}\|\partial_t\eta^n(t)\|_{L^\infty(I;L^{2}(\Gamma))}|t-s|\\
&
\leq\,c(\kappa)\|\bfu^n(s)\|_{L^2(\mathcal O_{\zeta_\epsilon^n(s)})}\|\partial_t\zeta_\epsilon^n\|_{L^\infty(I;L^2(\Gamma)}\|\partial_t\eta^n(t)\|_{L^\infty(I;L^{2}(\Gamma))}|t-s|.
 \end{align*}
using Corollary \ref{cor:3.5}. 
By \eqref{GalerkinConvEtanp} and \eqref{GalerkinVolnp1} we thus get
 \begin{align*}
|(II)|
&
\leq\,c(\kappa)\|\partial_t\zeta_\epsilon^n\|_{L^\infty(I;L^2(\Gamma)}|t-s|\\
&
\leq\,c(\kappa,\epsilon)\|\zeta^n\|_{L^\infty(I;L^2(\Gamma)}|t-s|.
 \end{align*}
As in the proof of Lemma \ref{lemma:fg} we conclude for $K>0$ that
\begin{align*}
\mathbb P\big(|t-s|^{-1}|(II)|\geq K\big)
&\leq \,\frac{c(\kappa,\epsilon)}{K},
\end{align*}
 provided $\zeta\in L^1(\Omega;L^\infty(I;L^2(\Gamma))$.
%we further  have 
%\begin{align*}
%|(II)|\lesssim |t-s|^{1/2}.
%\end{align*}
%By \eqref{eq:molly} and \eqref{eq:projection} the last term can be estimated by
%\begin{align*}
%\|\mathcal P^N((\partial_t\eta^N)_\kappa(t)\|_{W^{3,2}(\omega)}
%\leq c\|(\partial_t\eta^N)_\kappa(t)\|_{W^{3,2}(\Gamma)}\\&\leq c\kappa^{-3}\|\partial_t\eta^N(t)\|_{L^2(\omega)},
%\end{align*}
%which is bounded to to \eqref{GalerkinConvEtaNp} such that 
%we conclude
%$$(II)\leq\,c(\kappa)\|s-t|^\frac{1}{6}.$$
  The term $(I)$ is estimated using \eqref{weakeps} obtaining (this can be justified by It\^{o}'s formula similarly to the proof of Proposition \ref{prop:unique})
  {\small
\begin{equation}
\begin{aligned}
\label{weakeps'}
(I)&=- \int_t^s
\int_{\Gamma} \big(\epsilon   \Dely (\partial_t\eta^n(t))_\kappa \,\partial_t\Dely\eta^n
  +  \Dely (\partial_t\eta^n(t))_\kappa \,\Dely\eta^n
\big)\dy\,\dd r
\\&- \epsilon\int_t^s
\int_{\Gamma} \naby^3 (\partial_t\eta^n(t))_\kappa :\naby^3\eta^n\dy\,\dd r
\\&+\int_t^s\int_{\Gamma} \frac{1}{2}\bn_{\zeta^n_\epsilon } \cdot \bn (\partial_t\eta^n(t))_\kappa \,  \partial_t\zeta_\epsilon \,\partial_t\eta^n \,
 \vert\det(\naby\bm{\varphi}_{\zeta_\epsilon })\vert\dy\,\dd r
\\
&+
\int_t^s \int_{\mathcal{O}_{\zeta_\epsilon}}\Big(  \bu\cdot \partial_r\mathcal E_\kappa^{\zeta_\epsilon^n(r)}
(\partial_t\eta^n(t)) )
-
\frac{1}{2}((\bv_\epsilon\cdot\nabx)\bu)\cdot \mathcal E_\kappa^{\zeta_\epsilon^n(r)}
(\partial_t\eta^n(t))   \Big) \dx\,\dd r
\\
&+
 \int_t^s\int_{\mathcal{O}_{\zeta}}\Big(  
\frac{1}{2}((\bv_\epsilon\cdot\nabx)\mathcal E_\kappa^{\zeta_\epsilon^n(r)}
(\partial_t\eta^n(t))) \cdot \bu
-\nabx \bu:\nabx \mathcal E_\kappa^{\zeta_\epsilon^n(r)}
(\partial_t\eta^n(t))) \Big) \dx\,\dd r
 \\
&
 +
 \frac{1}{2}\int_t^s
\int_{\Gamma} 
((\bsigma\cdot\naby)(\bsigma\cdot\naby)\partial_t \eta ) (\partial_t\eta^n(t))_\kappa
\dy\,\dd r
\\&- 
\int_t^s\int_{\Gamma}  
( (\bsigma\cdot\naby)\partial_t \eta^n ) (\partial_t\eta^n(t))_\kappa\dy\dd B_r.
\end{aligned}
\end{equation}
}
Due to \eqref{eq:molly},  Corollary \ref{cor:3.5} and the uniform estimates \eqref{GalerkinConvEta1np}--\eqref{GalerkinConvEta2np},
it is clear that all terms can be estimated by $c(\kappa)|t-s|^{1/2}(1+\|\zeta^n\|_{C(I\times\Gamma)}+\|\nabx\mathbf u^n\|_{L^2(\mathcal O_{\zeta^n_\epsilon})})$
except for the last one. As in the proof of Lemma \ref{lemma:fg} it can be controlled via
%Here we have by Markov's inequality and It\^o-isometry (using also $\divy \bsigma=0$)
\begin{align*}
\mathbb P\bigg(|t-s|^{-1/2}\bigg|&\int_t^s\int_{\Gamma}  
( (\bsigma\cdot\naby)\partial_t \eta^n) 
(\partial_t\eta^n)_\kappa \dy\dd B_r\bigg|>K\bigg)\\
%&\leq\,\frac{1}{K^2}\mathbb E\bigg||t-s|^{-1/2}\int_t^s\int_{\Gamma}  
%( \bsigma\partial_r \eta^n ) 
%\cdot\naby(\partial_r\eta^n)_\kappa \dy\dd B_r\bigg|^2\\
%&=\,\frac{1}{K^2}\mathbb E\bigg[|t-s|^{-1/2}\int_t^s\bigg|\int_{\Gamma}  
%( (\bsigma\partial_r \eta^n ) \cdot\naby
%(\partial_r\eta^n)_\kappa \dy\bigg|^2\,\dd r\bigg]\\
%&\lesssim \,\frac{1}{K^2}\mathbb E\bigg[|t-s|^{-1/2}\int_t^s\|\partial_r\eta^n\|_{L^2(\Gamma)}
%\|\naby(\partial_r\eta^n)_\kappa)\|_{L^2(\Gamma)}\,\dd r\bigg]\\
&\lesssim \,\frac{\kappa^{-1}}{K^2}\mathbb E\bigg[|t-s|^{-1/2}
\int_t^s\|\partial_t\eta^n(r)\|_{L^2(\Gamma)}\|\naby(\partial_t\eta^n(t))_\kappa\|_{L^2(\Gamma)}\,\dd r\bigg]\\
&\lesssim \,\frac{1}{K^2}\mathbb E\bigg[|t-s|^{-1/2}
\int_t^s\|\partial_t\eta^n(r)\|_{L^2(\Gamma)}\|\partial_t\eta^n(t)\|_{L^2(\Gamma)}\,\dd r\bigg]
\lesssim \,\frac{\kappa^{-1}}{K^2}
\end{align*}
using \eqref{GalerkinConvEtanp}.
This can be made arbitrarily small for $K$ large.
 Noticing that property $(4)$ from Theorem \ref{thm:auba} 
follows by the usual compactness in (negative) Sobolev spaces, we conclude tightness of the law of 
$((\partial_t\eta^n,\mathbb I_{\mathcal O_{\zeta^n_\epsilon}}\mathbf u^n),
(\partial_t\eta^n,\mathscr F^{\zeta^n_\epsilon}\eta^N))$
on $(L^2(I;X'\times X),\tau_\sharp)$.

The tightness for $((\partial_t\eta^n,\mathbb I_{\mathcal O_{\zeta^n_\epsilon}}\mathbf u^n),
(0,\mathbf u^n-\mathscr F^{\zeta_\epsilon^n}\partial_t\eta^n))$ follows along the same lines, the only difference
being the regularisation of
$$f_n:=(0,\mathbf u^n-\mathscr F^{\zeta_\epsilon^n}\partial_t\eta^n).$$
While $\mathscr F^{\zeta_\epsilon^n}\partial_t\eta^n$ can be replaced by
$\mathcal E_\kappa ^{\zeta_\epsilon^n(s)}(\partial_t\eta^n(t))$
as above, we need to regularise $\mathbf u^n$ accordingly to preserve the homogeneous boundary conditions
on $f_n$, which can by done by using \cite[Proposition 2.28 and Lemma A.13]{lengeler2014weak}. 
This leads to a function $f_{n,\kappa}$ which has zero boundary conditions and thus only the fluid part is seen in
the expression $\langle g_n(t)-g_n(s),f_{n,\kappa}(t)\rangle$ 
(where $g_n:=(\partial_t\eta^n,\mathbb I_{\mathcal O_{\zeta^n_\epsilon}}\mathbf u^n)$
 as above). In particular, the noise is not seen and we conclude  that
\begin{align*}
\langle g_n(t)-g_n(s),f_{n,\kappa}(t)\rangle\leq\, c(\kappa)|t-s|^{1/2}(1+\|\zeta^n\|_{C(\overline I\times\Gamma)}+\|\nabx\mathbf {u}^n\|_{L^2(\mathcal O_{\zeta_\epsilon^n})})
\end{align*}
and we obtain again the claimed tightness.
\end{proof}

\subsection{Passage to the limit}
With the bounds from \eqref{GalerkinConvEta1np}--\eqref{GalerkinConvEta2np} at hand, we wish to obtain compactness.
% For this, we define the path space
%\begin{align*}
%\chi= \chi_{B} 
%\times\chi_{{\bu}}^2\times\chi_{\nabla{\bu}}\times\chi_{\eta}^2\times \chi_{f,g}^2
%\end{align*}
%where
%\begin{align*}
%&
%\chi_{B}=C(\overline{I}), 
%\quad \chi_{\bu}= 
%\big(L^2(I;L^2_{\divx}(\mathcal{O}\cup S_{L/2})),w\big),\quad \chi_{\nabla{\bu}}= 
%\big(L^2(I;L^2(\mathcal{O}\cup S_{L/2})),w\big),
%\\&\chi_{\eta}
%=\big(W^{1,\infty}(I;L^{2}(\Gamma)),w^\ast\big)\cap
%\big(L^\infty(I;W^{3,2}(\Gamma)),w^\ast\big)\cap \big(W^{1,2}(I;W^{s,2}(\Gamma )),w\big),
%\\&
%\chi_{f,g}=\big(L^2(I;X'\times X),\tau_\sharp\big),
%\end{align*}
%with $s\in(0,\frac{1}{2})$.
We consider the same path space $\chi$ as in Section \ref{sec:compN}.
 From \eqref{GalerkinConvEta1Np}--\eqref{GalerkinConvEta2Np}
(together the Alaoglu-Burbaki Theorem) and Lemma \ref{lemma:fg} we obtain similarly to Proposition \ref{prop:skorokhodN}:

\begin{proposition}\label{prop:skorokhod'}
There exists a complete probability space $(\tilde\Omega,\tilde{\mathfrak F},\tilde{\mathbb P})$ 
with $\chi$-valued random variables \footnote{The fact that the variable $\mathscr F^{\tilde\zeta^n_\epsilon}\partial_t\tilde\eta^n$ can be represented in that form follows from the measureability of $\mathscr F$ on the pathspace, cf. Proposition \ref{prop:musc}.}
$$\tilde{\boldsymbol{\Theta}}^n
:=
\bigg[\tilde B_t^n,\mathbb I_{\mathcal O_{\tilde\zeta^n_\epsilon}}\tilde{\mathbf u}^n,\tilde\bv^n,
\mathbb I_{\mathcal O_{\tilde{\zeta}^n_\epsilon}}\nabx\tilde{\mathbf u}^n, \tilde{\eta}^n ,\tilde\zeta^n,
\begin{pmatrix}
((\partial_t\tilde{\eta}^n,\mathbb I_{\mathcal O_{\tilde{\zeta}^n_\epsilon}}\tilde{\mathbf u}^n),
(\partial_t\tilde\eta^n,\mathscr F^{\tilde\zeta^n_\epsilon}\partial_t\tilde\eta^n))
\\
((0,\mathbb I_{\mathcal O_{\tilde\zeta^n_\epsilon}}\tilde{\mathbf u}^n),(\partial_t\tilde\eta^n,\tilde{\mathbf u}^n
-\mathscr F^{\tilde\zeta^n_\epsilon}\partial_t\tilde\eta^n))\end{pmatrix}\bigg],$$
for $n\in\mathbb N$ and
$$\tilde{\boldsymbol{\Theta}}:=\bigg[\tilde B_t,\mathbb I_{\mathcal O_{\tilde\zeta_\epsilon}}\tilde{\mathbf u}^n,\tilde\bv,
\mathbb I_{\mathcal O_{\tilde\zeta_\epsilon}}\nabx\tilde{\mathbf u}, \tilde\eta,\tilde\zeta ,\begin{pmatrix}((\partial_t\eta,\mathbb I_{\mathcal O_{\zeta_\epsilon}}\mathbf u^n),
(\partial_t\tilde\eta,\mathscr F^{\tilde\zeta_\epsilon}\partial_t\tilde\eta))\\
((0,\mathbb I_{\mathcal O_{\tilde\eta}}\tilde{\mathbf u}),(\partial_t\tilde\eta,\tilde{\mathbf u}
-\mathscr F^{\tilde\zeta_\epsilon}\partial_t\tilde\eta))\end{pmatrix}\bigg]$$
such that
\begin{enumerate}
 \item[(a)] for all $n\in \mathbb N$ the law of 
$\tilde{\boldsymbol{\Theta}}^n$
on $\chi$ is given by
\begin{align*}
 \mathcal{L}\bigg[B_t,\mathbb I_{\mathcal O_{\zeta^n_\epsilon}}\mathbf u^n,\bfv^n,
\mathbb I_{\mathcal O_{\zeta_\epsilon^n}}\nabx \mathbf u^n, \eta^n,\zeta^n ,
\begin{pmatrix}
((\partial_t\eta^n,\mathbb I_{\mathcal O_{\zeta^n_\epsilon}}\mathbf u^n),
(\partial_t\eta^n,\mathscr F^{\zeta^n}\partial_t\eta^n))
\\
((0,\mathbb I_{\mathcal O_{\zeta^n_\epsilon}}\mathbf u^n),(\partial_t\eta^n,\mathbf u^n-\mathscr F^{\zeta^n_\epsilon}\partial_t\eta^n))\end{pmatrix}\bigg]
\end{align*}
 \item[(b)] $\tilde{\boldsymbol{\Theta}}^n$ converges $\,\tilde{\mathbb P}$-almost surely to 
$\tilde{\boldsymbol{\Theta}}$
 in the topology of $\chi$, i.e.
\begin{align} \label{wWS116}
\begin{aligned}
\tilde B_t^n&\rightarrow \tilde B_t \quad \mbox{in}\quad C(\overline I) \ \tilde{\mathbb P}\mbox{-a.s.}, 
\\
\mathbb I_{\mathcal O_{\tilde\zeta_\epsilon^n}}\tilde{\mathbf u}^n,\,\tilde\bv^n&\weakto \mathbb I_{\mathcal O_{\tilde\zeta_\epsilon}}\tilde{\mathbf u} ,\,\tilde\bv\quad \mbox{in}\quad L^2(I;L^{2}(\mathcal O\cup S_{\alpha}))\ \tilde{\mathbb P}\mbox{-a.s.}, 
\\
\mathbb I_{\mathcal O_{\tilde\zeta^n_\epsilon}}\nabx\tilde{\mathbf u}^n&\weakto  \mathbb I_{\mathcal O_{\tilde\zeta_\epsilon}}\nabx \tilde{\mathbf u} \quad \mbox{in}\quad L^2(I;L^{2}(\mathcal O\cup S_{\alpha}))\ \tilde{\mathbb P}\mbox{-a.s.}, 
\\
\tilde \eta^n&\weakto^\ast \tilde \eta \quad \mbox{in}\quad L^\infty(I;W^{3,2}(\Gamma)) \ \tilde{\mathbb P}\mbox{-a.s.}, \\
\tilde \eta^n&\weakto^\ast \tilde \eta\quad \mbox{in}\quad W^{1,\infty}(I;L^{2}(\Gamma)) \ \tilde{\mathbb P}\mbox{-a.s.}, 
\\
\tilde \eta^n&\weakto \tilde \eta \quad \mbox{in}\quad W^{1,2}(I;W^{2,2}(\Gamma)) \ \tilde{\mathbb P}\mbox{-a.s.}, \\
\tilde\zeta^n&\rightarrow\tilde\zeta \quad \mbox{in}\quad W^{1,2}(I\times \Gamma) \ \tilde{\mathbb P}\mbox{-a.s.}, 
\end{aligned}
\end{align}
as well as (recalling the definition of $\tau_\sharp$ from \eqref{eq:tau})
\begin{align}\label{wWS116BA}\begin{aligned}
\int_{I}\int_{\mathcal O\cup S_{\alpha}}&\mathbb I_{\mathcal O_{\tilde\zeta_\epsilon^n}}
\bfu^n\cdot\mathscr F^{\tilde\zeta^n_\epsilon}\partial_t\tilde\eta^n\dxt+\int_{I}\int_\Gamma|\partial_t\tilde\eta^n|^2\,\dy \dt
\\
&\longrightarrow \int_{I}\int_{\mathcal O\cup S_{\alpha}}\mathbb I_{\mathcal O_{\tilde\zeta_\epsilon}}
\tilde\bfu\cdot\mathscr F^{\tilde\zeta_\epsilon}\partial_t\eta\dxt+\int_{I}\int_\Gamma|\partial_t\tilde\eta|^2\,\dy \dt
\end{aligned}
\end{align}
and
\begin{align}\label{wWS116CA}\begin{aligned}
\int_{I}\int_{\mathcal O\cup S_{\alpha}}&
\mathbb I_{\mathcal O_{\tilde\zeta_\epsilon^n}}
\tilde\bfu^n\cdot(\tilde\bfu^n-\mathscr F^{\tilde\zeta^n_\epsilon}\partial_t\tilde\eta^n)\dxt
\\
&\longrightarrow \int_{I}\int_{\mathcal O\cup S_{\alpha}}
\mathbb I_{\mathcal O_{\tilde\zeta_\epsilon}}
\tilde\bfu\cdot(\tilde\bfu-\mathscr F^{\tilde\zeta_\epsilon}\partial_t\tilde\eta)\dxt
\end{aligned}
\end{align}
$\tilde{\mathbb P}$-a.s.
\end{enumerate}
\end{proposition}
Now we introduce the filtration on the new probability space, 
which ensures the correct measurabilities of the new random variables.
%We denote by $\bfr_t$ the operator of restriction to the interval $[0,t]$ acting on various path spaces. In particular, if $X$ stands for one of the path spaces $L^r_{\mathrm loc}(0,\infty;L^r(\mathcal{O}))$ or $C_{\mathrm loc}([0,\infty),\mathfrak U_0)$ and $t\in[0,T]$, we define
%\begin{align}\label{restr}
%\bfr_t:X\rightarrow X|_{[0,t]},\quad f\mapsto f|_{[0,t]}.
%\end{align}
%Clearly, $ \bfr_t$ is a continuous mapping.
Let $(\tilde{\mathfrak F}_t)_{t\geq0}$ and $(\tilde{\mathfrak F}_t^{n})_{t\geq0}$ be the $\tilde{\mathbb P}$-augmented
canonical filtration on the variables $\tilde{\boldsymbol{\Theta}}$ and
$\tilde{\boldsymbol{\Theta}}^n $, respectively, that is\footnote{Some of the variables are not continuous in time. For those, one can define $\sigma_t$ as the history of a random distribution, cf. \cite[Chapter 2.8]{BFHbook}}
%\todo{add filtrations}
\begin{align*}
\tilde{\mathfrak F}_t=\sigma\big[\sigma_t(\tilde B_t)\cup \sigma_t(\mathbb I_{\mathcal O_{\tilde\zeta_\epsilon}}\tilde{\mathbf u})\cup \sigma_t(\tilde\bv) \cup \sigma_t(
\mathbb I_{\mathcal O_{\tilde\zeta_\epsilon}}\nabx\tilde{\mathbf u}) \cup \sigma_t( \tilde\eta) \cup \sigma_t(\tilde\zeta ) \big]
\end{align*}
for $t\in I$ and similarly for $\tilde{\mathfrak F}_t^n$.

By \cite[Theorem 2.9.1]{BFHbook} the weak equation continuous to hold on the new probability space.
Combining \eqref{wWS116BA} and \eqref{wWS116CA} we have
\begin{align*}
\int_{I}\int_{\mathcal O_{\tilde\zeta^n_\epsilon}}&|\tilde\bfu^n|^2\dxt
+\int_{I}\int_\Gamma|\partial_t\tilde\eta^n|^2\,\dy \dt
\\&
\longrightarrow \int_{I}\int_{\mathcal O_{\tilde\zeta_\epsilon}}|\tilde\bfu|^2\dxt
+\int_{I}\int_\Gamma|\partial_t\tilde\eta|^2\,\dy \dt
\end{align*}
 $\tilde{\mathbb P}$-a.s. By uniform convexity of the $L^2$-norm this implies
\begin{align*}
\tilde\eta^n&\rightarrow\tilde\eta\quad \mbox{in}\quad W^{1,2}(I;L^2(\Gamma)) ,\\
\mathbb I_{\mathcal O_{\tilde\zeta^n_\epsilon}}\tilde{\mathbf u}^n&\rightarrow \mathbb I_{\mathcal O_{\tilde\zeta_\epsilon}}\tilde{\mathbf u}
 \quad \mbox{in}\quad L^2(I;L^{2}(\mathcal O\cup S_{\alpha}))\ \tilde{\mathbb P}\mbox{-a.s.}.
\end{align*}
%This is sufficient to pass to the limit and the weak formulation of the equations (now that all terms except for the convective term can be treated 
%by \eqref{wWS116}).
%As far as the energy balance is concerned, this is even easier since there is no noise. By \eqref{wWS116} and the lower-semi continuity of the involved quantities
%we can pass to the limit (obtaining only an inequality).
We can now apply Proposition \ref{prop:unique}
and argue as in Section \ref{sec:cip} to obtain convergence of the original sequence, in particular,
\begin{align*}
\eta^n&\rightarrow\eta\quad \mbox{in}\quad C(\overline I\times\Gamma) \quad \mbox{in probability},\\
\mathbb I_{\mathcal O_{\zeta^n_\epsilon}}{\mathbf u}^n&\rightarrow \mathbb I_{\mathcal O_{\zeta_\epsilon}}{\mathbf u}
 \quad \mbox{in}\quad L^2(I;L^{2}(\mathcal O\cup S_{\alpha}))\quad \mbox{in probability}.
\end{align*}
This yields due the uniform-in-probability estimates
(for arbitrary $p<\infty$)
\begin{align*}
\eta^n&\rightarrow\eta\quad \mbox{in}\quad L^p(\Omega;C(\overline I\times\Gamma)),\\
\mathbb I_{\mathcal O_{\zeta^n_\epsilon}}{\mathbf u}^n&\rightarrow \mathbb I_{\mathcal O_{\zeta_\epsilon}}{\mathbf u}
 \quad \mbox{in}\quad L^p(\Omega;L^2(I;L^{2}(\mathcal O\cup S_{\alpha}))),
\end{align*}
which gives the desired compactness of the fixed-point map.

\section{The limit $\epsilon\rightarrow 0$}
\label{sec:limit}
For fixed $\epsilon>0$, Theorem \ref{thm:main'} yields the existence of a probabilistically strong solution $(\eta^\epsilon,\bu^\epsilon)$ to the regularized fluid-structure system defined on a given stochastic basis. Using the energy balance we obtain
\begin{align}
&\Big(\sup_I\Vert \eta^\epsilon \Vert_{W^{2,2}(\Gamma )}^{2}
+
\epsilon\sup_I\Vert \eta^\epsilon \Vert_{W^{3,2}(\Gamma )}^{2}\Big) \lesssim1,\label{GalerkinConvEta1epsp}
\\
&\sup_I\Vert\partial_t\eta^\epsilon \Vert_{L^2(\Gamma )}^{2}\lesssim 1,\label{GalerkinConvEtaepsp}
\\
&\epsilon\int_I\Vert\partial_t\eta^\epsilon \Vert_{W^{2,2}(\Gamma )}^{2}\dt\lesssim 1,
\\
&\sup_I\Vert\bu^\epsilon \Vert_{L^2(\mathcal{O}_{(\eta^\epsilon)_\epsilon})}^{2}\lesssim 1,
\label{GalerkinVolneps1}
\\
&\int_I\Vert\nabx \bu^\epsilon \Vert_{L^{2}(\mathcal{O}_{(\eta^\epsilon)_\epsilon})}^2\dt\lesssim1. \label{GalerkinVolneps2}
\end{align}
In addition, for any $s\in(0,\frac{1}{2})$, it follows from $\bu^\epsilon\circ\bm{\varphi}_{(\eta^\epsilon)_\epsilon} =\bn\partial_t\eta^\epsilon $, \eqref{GalerkinVolnp2} and the trace theorem that
\begin{equation}
\begin{aligned}
\label{GalerkinConvEta2epsp}
\int_I\Vert \partial_t\eta^\epsilon \Vert_{W^{s,2}(\Gamma)}^2\dt
\lesssim 1.
\end{aligned}
\end{equation}
Moreover, we can argue as in Proposition \ref{lemma:fg'} to obtain:
\begin{lemma}\label{lemma:fg''}
The laws of
\begin{align*}
((\partial_t\eta^\epsilon,\mathbb I_{\mathcal O_{(\eta^\epsilon)_\epsilon}}\mathbf u^\epsilon),
(\partial_t\eta^\epsilon,\mathscr F^{\eta^\epsilon} \partial_t\eta^\epsilon))\quad\text{and}\quad
((\partial_t\eta^\epsilon,\mathbb I_{\mathcal O_{(\eta^\epsilon)_\epsilon}}\mathbf u^\epsilon),
(\partial_t\eta^\epsilon,\mathbf u^\epsilon-\mathscr F^{\eta^\epsilon} \partial_t\eta^\epsilon))
\end{align*}
on $(L^2(I;X'\times X),\tau_\sharp)$ are tight.
\end{lemma}
Unfortunately, \eqref{GalerkinConvEta1epsp}--\eqref{GalerkinConvEta2epsp} is no longer sufficient  which is why we first improve the regularity of the shell by adapting a method from \cite{MuSc} (which we also applied in \cite{breit2021incompressible}).

\subsection{Higher regularity}
\label{sec:highreg}
As already explained in the introduction, the regularity arising from \eqref{GalerkinConvEta1epsp}--\eqref{GalerkinConvEta2epsp} is not sufficient to control the terms involving the Piola transform in the weak equation. Thus we aim at improving the spatial regularity of $\eta$ implementing ideas from \cite{MuSc}.

Finally, for some $h>0$, we let $\Delta_{h}^sf(\by)=h^{-s}(f(\by+\be_ih)-f(\by))$ for $i=1,2$ represent the fractional difference quotient in space in the direction $\be_i$. 
Now, for
\begin{align*}
D^{s,\mathscr{K}}_{-h,h}f:=\Delta_{-h}^s\Delta_{h}^s f
-\mathscr K_\eta(\Delta_{-h}^s\Delta_{h}^sf),
\end{align*}
where $s\in(0,\frac{1}{2})$, we consider the following as test function 
\begin{align*}
(\bm{\phi},\phi)=\big( \mathcal J_{\eta^\epsilon}^{-1}\big(\testE(D^{s,\mathscr{K}}_{-h,h}\eta^\epsilon)\big)\,,\, \iota_{\eta^\epsilon}^{-1}D^{s,\mathscr{K}}_{-h,h}\eta^\epsilon
\big).
\end{align*}
in the weak formulation of the regularized fluid-structure system (this can be justified by It\^{o}'s formula similarly to the proof of Proposition \ref{prop:unique}). By making the fourth order term the subject, we obtain 
{\small
\begin{align}
\int_{\Gamma}&\nonumber
 \Dely \mathscr{K}_\eta(\Delta_{-h}^s\Delta_{h}^s\eta^\epsilon) \Dely\eta^\epsilon \dy\dt
-
\int_{\Gamma}
 \Dely (\Delta_{-h}^s\Delta_{h}^s\eta^\epsilon) \Dely\eta^\epsilon \dy\dt
 \\&\nonumber+
\epsilon
\int_{\Gamma}
 \naby^3 \mathscr{K}_\eta(\Delta_{-h}^s\Delta_{h}^s\eta^\epsilon) :\naby^3\eta^\epsilon \dy\dt
-
 \epsilon
 \int_{\Gamma}
 \naby^3 (\Delta_{-h}^s\Delta_{h}^s\eta^\epsilon) :\naby^3\eta^\epsilon \dy\dt
\\&\nonumber+
\epsilon
\int_{\Gamma}
 \Dely \mathscr{K}_\eta(\Delta_{-h}^s\Delta_{h}^s\eta^\epsilon) \partial_t\Dely\eta^\epsilon \dy\dt
-
 \epsilon
 \int_{\Gamma}
 \Dely (\Delta_{-h}^s\Delta_{h}^s\eta^\epsilon)\partial_t\Dely\eta^\epsilon \dy\dt
\\&\nonumber=
\dd  \int_{\mathcal{O}_{(\eta^\epsilon)_\epsilon}}\bu^\epsilon  \cdot\testE(D^{s,\mathscr{K}}_{-h,h}\eta^\epsilon(t))\dx
+
\dd 
 \int_{\Gamma}   \partial_t\eta^\epsilon (D^{s,\mathscr{K}}_{-h,h}\eta^\epsilon(t))\dy
 \\&-
  \int_{\Gamma}   \partial_t\eta^\epsilon \partial_t (D^{s,\mathscr{K}}_{-h,h}\eta^\epsilon)\dy\dt
  \nonumber
\\&\nonumber-
\frac{1}{2}
\int_{\Gamma} 
\bn_{\eta^\epsilon} \cdot \bn^\top(D^{s,\mathscr{K}}_{-h,h}\eta^\epsilon) \,  \partial_t(\eta^\epsilon)_\epsilon\,\partial_t\eta^\epsilon  \,
 \vert\det(\naby\bm{\varphi}_{(\eta^\epsilon)_\epsilon})\vert
 \dy\dt
\\&\nonumber
-
\int_{\mathcal{O}_{(\eta^\epsilon)_\epsilon}}  \bu^\epsilon  \cdot  \partial_t(\testE(D^{s,\mathscr{K}}_{-h,h}\eta^\epsilon)) 
\dx\dt
\\&+\frac{1}{2}
\int_{\mathcal{O}_{(\eta^\epsilon)_\epsilon}}  ( ((\bu^\epsilon)_\epsilon\cdot\nabx) \bu^\epsilon)\cdot\testE(D^{s,\mathscr{K}}_{-h,h}\eta^\epsilon)
 \dx\dt 
 \nonumber
\\&\nonumber
-
\frac{1}{2}
\int_{\mathcal{O}_{(\eta^\epsilon)_\epsilon}}  
 (((\bu^\epsilon)_\epsilon\cdot\nabx )\testE(D^{s,\mathscr{K}}_{-h,h}\eta^\epsilon)) \cdot \bu^\epsilon 
 \dx\dt
 \\&+
 \int_{\mathcal{O}_{(\eta^\epsilon)_\epsilon}} 
 \nabx\bu^\epsilon:\nabx \testE(D^{s,\mathscr{K}}_{-h,h}\eta^\epsilon)
 \dx\dt 
 \nonumber
 \\&\nonumber
 -
 \frac{1}{2} 
\int_{\Gamma} 
((\bsigma\cdot\nabx)(\bsigma\cdot\nabx)\partial_t \eta^\epsilon )D^{s,\mathscr{K}}_{-h,h}\eta^\epsilon
\dy\dt 
+ 
\int_{\Gamma}  
( (\bsigma\cdot\nabx)\partial_t \eta^\epsilon ) 
D^{s,\mathscr{K}}_{-h,h}\eta^\epsilon \dy\dd B_t
  \\&=:I_1\dt+\ldots+I_{10}\dt.\label{galerkinweak1A}
\end{align}
}
a.s. First of all, note that
\begin{align*}
\int_{\Gamma}&
\big(
 \Dely (\Delta_{-h}^s\Delta_{h}^s\eta^\epsilon ) \Dely\eta^\epsilon   
 +
 \epsilon 
 \naby^3 (\Delta_{-h}^s\Delta_{h}^s\eta^\epsilon) :\naby^3\eta^\epsilon  
 +
 \epsilon 
 \Dely (\Delta_{-h}^s\Delta_{h}^s\eta^\epsilon)\partial_t\Dely\eta^\epsilon \big)\dy\dt
\\
&=-
 \int_{\Gamma}
 \vert\Dely  \Delta_{h}^s\eta^\epsilon \vert^2 \dy\dt
 -
 \epsilon
 \int_{\Gamma}
 \vert\naby^3  \Delta_{h}^s\eta^\epsilon \vert^2 \dy\dt
 -\frac{\epsilon}{2}
 \dd
 \int_{\Gamma}
 \vert\Dely  \Delta_{h}^s\eta^\epsilon \vert^2 \dy
%\\&
%\sim
% \Vert \eta^\epsilon \Vert_{W^{s+2,2}(\Gamma)}^2\dt
% +
% \epsilon
% \Vert \eta^\epsilon \Vert_{W^{s+3,2}(\Gamma)}^2\dt
% +
% \frac{\epsilon}{2}
% \dd
% \Vert \eta^\epsilon \Vert_{W^{s+2,2}(\Gamma)}^2
\end{align*}
and since the corrector $\mathscr K_{\eta^\epsilon} $ is spatially independent,
\begin{align*}
\int_{\Gamma}
 \Dely \mathscr K_{\eta^\epsilon} (\Delta_{-h}^s\Delta_{h}^s\eta^\epsilon ) \Dely\eta^\epsilon  \dy\dt
 =
0.
\end{align*}
Similarly, the two remaining $\epsilon$-terms are zero.
We now wish to take the $p$-th moment of the time integral of \eqref{galerkinweak1A} where $p\geq1$. 
To begin with, we have
\begin{align*}
\mathbb{E}\Big(\int_I I_1\dt \Big)^p 
&\lesssim \mathbb{E}\Big(\sup_I\Vert  \bu^\epsilon      \Vert_{L^{2}(\mathcal{O}_{(\eta^\epsilon)_\epsilon})}
\sup_I \Vert
\Delta_{-h}^s\Delta_{h}^s\eta^\epsilon 
\Vert_{L^{2}(\Gamma)}\Big)^p 
\\
&\lesssim \mathbb{E}\sup_I\Vert \bu^\epsilon \Vert_{L^{2}(\mathcal{O}_{(\eta^\epsilon)_\epsilon})}^{2p}
+
\mathbb{E}\sup_I \Vert
\eta^\epsilon 
\Vert_{W^{2,2}(\Gamma)}^{2p}
\end{align*}
by Proposition \ref{prop:musc}, where the right-hand side is uniformly bounded by \eqref{GalerkinConvEta1epsp} and \eqref{GalerkinVolneps1}.
Moreover,
\begin{align*}
\mathbb{E}\Big(\int_I I_2\dt \Big)^p 
&\lesssim 
\mathbb{E}\Big(
\sup_I\Vert\partial_t \eta^\epsilon  \Vert_{L^{2}(\Gamma)}
\Big[
\sup_I \Vert
\Delta_{-h}^s\Delta_{h}^s\eta^\epsilon 
\Vert_{L^{2}(\Gamma)}
+
\sup_I\vert \mathscr K_{\eta^\epsilon} (\Delta_{-h}^s\Delta_{h}^s\eta^\epsilon 
)\vert\Big]\Big)^p
\\
&\lesssim 
\mathbb{E}\Big(\sup_I\Vert\partial_t \eta^\epsilon  \Vert_{L^{2}(\Gamma)}
\Big[
\sup_I \Vert
\eta^\epsilon 
\Vert_{W^{2,2}(\Gamma)}
+
\sup_I\Vert
\Delta_{-h}^s\Delta_{h}^s\eta^\epsilon 
\Vert_{L^{1}(\Gamma)}\Big]
\Big)^p
\\
&\lesssim 
\mathbb{E}\sup_I\Vert \partial_t\eta^\epsilon  \Vert_{L^{2}(\Gamma)}^{2p}
+
\mathbb{E}
\sup_I \Vert
\eta^\epsilon 
\Vert_{W^{2,2}(\Gamma)}^{2p}
\end{align*}
which is again uniformly bounded.
Since we assume that $\eta$ (and consequently also $\partial_t\eta$) has mean value zero and $ \mathscr K_{\eta^\epsilon}$ maps to spatially independent functions we have $\int_\Gamma\partial_t\eta^\epsilon\,\partial_t\mathscr K_{\eta^\epsilon}(\Delta_{-h}^s\Delta_{h}^s\eta^\epsilon 
)\dy=0$. Thus it holds
\begin{align*}
 -I_3=\int_\Gamma| \partial_t\Delta_h^s\eta^\epsilon|^2\dy\lesssim \| \partial_t\eta^\epsilon\|_{W^{s,2}(\Gamma)}^2,
\end{align*}
which is uniformly bounded in $L^p(\Omega;L^1(I))$ by \eqref{GalerkinConvEta2epsp}.
Recalling the definition of $\bfvarphi_{\eta}$ from
\eqref{eq:bfvarphi} we further have by \eqref{GalerkinConvEta1epsp}, \eqref{GalerkinConvEtaepsp} and 2D Sobolev embeddings 
\begin{align*}
\mathbb E\Big(\int_II_4\dt\Big)&\lesssim \mathbb E\Big(\int_I \|\Delta_{-h}^s\Delta_h^s\eta^\epsilon\|_{L^\infty(\Gamma)} \big(1+\|\naby\eta^\epsilon\|_{L^3(\Gamma)}\big) \|\partial_t\eta^\epsilon\|^2_{L^3(\Gamma)}\dt\Big)^p\\
&\lesssim \mathbb E\Big(\int_I \|\eta^\epsilon\|_{W^{2,2}(\Gamma)}\|\partial_t\eta^\epsilon\|^2_{W^{s,2}(\Gamma)}\dt\Big)^p\\
&\lesssim \mathbb E\Big(\int_I \|\partial_t\eta^\epsilon\|^2_{W^{s,2}(\Gamma)}\dt\Big)^p
\end{align*}
which is again bounded by \eqref{GalerkinConvEta2epsp}.
By using Lemma \ref{lem:higherInt} (with $p=p'=2$, $\theta=s$ and $\tilde{a}=6$), and \eqref{GalerkinConvEta1epsp}--\eqref{GalerkinVolneps2} (together with the embedding $W^{2,2}(\Gamma)\hookrightarrow C^{0,\theta}(\Gamma)$ for all $\theta<1)$
%then provided that $ \bu^\epsilon\in W^{1,2}(\mathcal{O}_{(\eta^\epsilon)_\epsilon})$, $\eta^\epsilon \in C^{0,s}(\Gamma) \cap W^{1,3}(\Gamma)$ and $\partial_t\eta^\epsilon \in L^2(\Gamma)$ where $s\in(0,1)$ , 
we have that (for $\delta>0$ arbitrary)
{\small
\begin{align*}
\mathbb{E}\Big(\int_I &I_5\dt \Big)^p  
\\&\lesssim
\mathbb{E}\bigg(
\int_I\Big(1+\Vert \Delta_{-h}^s\Delta_{h}^s\eta^\epsilon  \Vert_{W^{1,3}(\Gamma)}\Big)\Vert \partial_t\eta^\epsilon \Vert_{L^2(\Gamma)}
\Vert \bu^\epsilon \Vert_{W^{1,2}(\mathcal{O}_{(\eta^\epsilon)_\epsilon})}
\dt
\\&
+
\int_I
\Vert (\Delta_{-h}^s\Delta_{h}^s\eta^\epsilon )\partial_t\eta^\epsilon \Vert_{L^{6/5}(\Gamma)}
\Vert \bu^\epsilon \Vert_{W^{1,2}(\mathcal{O}_{(\eta^\epsilon)_\epsilon})}
\dt
\bigg)^p
\\&\lesssim
\mathbb{E}
\bigg(
\sup_I
\Vert \partial_t\eta^\epsilon \Vert_{L^2(\Gamma)}
 \int_I\big(1
+
\Vert  \Delta_h^s\eta^\epsilon  \Vert_{W^{s+1,3}(\Gamma)} \big)\Vert  \bu^\epsilon \Vert_{ W^{1,2}(\mathcal{O}_{(\eta^\epsilon)_\epsilon})}\dt 
\\&
+
\sup_I
\Vert  \partial_t\eta^\epsilon \Vert_{L^{2}(\Gamma)}
 \int_I\Vert \Delta_{-h}^s\Delta_{h}^s\eta^\epsilon  \Vert_{ L^3(\Gamma)}  
\Vert \bu^\epsilon \Vert_{ W^{1,2}(\mathcal{O}_{(\eta^\epsilon)_\epsilon})}\dt
\bigg)^p
\\
&\lesssim
\mathbb{E}
\bigg(
 \int_I\big(1
+
\Vert \Delta_h^s \eta^\epsilon  \Vert_{W^{s+1,3}(\Gamma)} +\Vert \eta^\epsilon  \Vert_{ W^{2s,3}(\Gamma)} \big)\Vert  \bu^\epsilon \Vert_{ W^{1,2}(\mathcal{O}_{(\eta^\epsilon)_\epsilon})}\dt 
\bigg)^p
\\
&\leq \delta
\mathbb{E}
\bigg(
 \int_I
\Vert  \Delta_h^s\eta^\epsilon  \Vert_{W^{2,2}(\Gamma)}^2\dt\bigg)^p+c(\delta) \bigg(\int_I\big(\|\eta^\epsilon\|_{W^{2,2}(\Gamma)}^2+\Vert  \bu^\epsilon \Vert_{ W^{1,2}(\mathcal{O}_{(\eta^\epsilon)_\epsilon})}^2\big)\dt 
\bigg)^p
%&\lesssim  \bigg( 
%\mathbb{E}
%\sup_I
%\Vert  \partial_t\eta^\epsilon \Vert_{L^{2}(\Gamma)}^{2p}
%+
%\sup_I\Vert \eta^\epsilon  \Vert_{ W^{2,2}(\Gamma)}^{4p}
%+
%\mathbb{E}\Big( \int_I
%\Vert \bu^\epsilon \Vert_{ W^{1,2}(\mathcal{O}_{(\eta^\epsilon)_\epsilon})}^2\dt
%\Big)^{2p}\bigg)
\end{align*} 
}
where we have used the continuous embedding $W^{2,2}(\Gamma)\hookrightarrow W^{s+1,3}(\Gamma)$, $s\in(0,\frac{1}{2})$. The first term can be absorbed for $\delta$ small enough, while the second done is uniformly bounded, cf. \eqref{GalerkinConvEta1epsp} and \eqref{GalerkinVolneps2}.
Next, it follows from Proposition \ref{prop:musc}  that
{\small
\begin{align*}
\sup_I
\Vert\testE(D^{s,\mathscr{K}}_{-h,h}\eta^\epsilon ) \Vert_{W^{1,2}(\mathcal{O}_{(\eta^\epsilon)_\epsilon})}
&\lesssim
\sup_I\Big(
\Vert \Delta_{-h}^s\Delta_{h}^s\eta^\epsilon \Vert_{W^{1,2}(\Gamma)}
+
\Vert(\Delta_{-h}^s\Delta_{h}^s\eta^\epsilon )\naby\eta^\epsilon  \Vert_{L^{2}(\Gamma)}
\Big)
\\&\lesssim
\sup_I\Big(
\Vert\eta^\epsilon \Vert_{W^{2,2}(\Gamma)}
+
\Vert \Delta_{-h}^s\Delta_{h}^s\eta^\epsilon  \Vert_{L^{\infty}(\Gamma)}
\Vert\eta^\epsilon  \Vert_{W^{1,2}(\Gamma)}
\Big)
\\&\lesssim
\sup_I\Big(
\Vert \eta^\epsilon  \Vert_{W^{2,2}(\Gamma)}
+ 
\Vert \eta^\epsilon  \Vert_{W^{2,2}(\Gamma)}
^2 
\Big)
\end{align*}
}
therefore,
{\small
\begin{align*}
\mathbb{E}\Big(&\int_I(I_6+I_7+I_8) \dt\Big)^p\\
&\lesssim
\mathbb{E}\bigg(
\sup_I
\Vert\testE(D^{s,\mathscr{K}}_{-h,h}\eta^\epsilon ) \Vert_{W^{1,2}(\mathcal{O}_{\eta^\epsilon})}
\int_I\big(\Vert \bu^\epsilon \Vert_{W^{1,2}(\mathcal{O}_{(\eta^\epsilon)_\epsilon})}
^2+\Vert \bu^\epsilon \Vert_{W^{1,2}(\mathcal{O}_{(\eta^\epsilon)_\epsilon})}
\big)\dt
\bigg)^p
\\
&\lesssim
\mathbb{E}
\sup_I 
 \Vert\eta^\epsilon  \Vert_{W^{2,2}(\Gamma)}
^{4p}
+
\mathbb{E}\bigg(
\int_I\Vert \bu^\epsilon\Vert_{W^{1,2}(\mathcal{O}_{(\eta^\epsilon)_\epsilon})}^2 \dt\bigg)^p. 
\end{align*} }
This is uniformly bounded by \eqref{GalerkinConvEta1epsp} and \eqref{GalerkinVolneps2}.
Furthermore, since $\mathscr K_{\eta^\epsilon} $ is spatially independent, and $\Gamma $ is endowed with periodic boundary condition,
\begin{equation}
\begin{aligned}
 \mathbb{E}\Big(\int_I I_{9}\dt \Big)^p   
 &= 
  \mathbb{E}\bigg(\int_I\int_{\Gamma}  \partial_t \eta^\epsilon  \,(\bsigma\cdot\naby)(\bsigma\cdot\naby) \Delta_{-h}^s\Delta_{h}^s\eta^\epsilon 
  \dy\dt
  \bigg)^p
\\&
\lesssim
 \mathbb{E}\bigg(
 \int_I\Vert\partial_t \eta^\epsilon \Vert_{W^{s,2}(\Gamma)}
 \Vert(\bsigma\cdot\naby)(\bsigma\cdot\naby) \Delta_{-h}^s\Delta_{h}^s\eta^\epsilon 
  \Vert_{W^{-s,2}(\Gamma)}\dt
  \bigg)^p
\\&
\leq\delta
 \mathbb{E}\bigg(
 \int_I 
 \Vert\Delta_h^s \eta^\epsilon 
  \Vert_{W^{2,2}(\Gamma)}^2\dt
  \bigg)^p
  +
  c(\delta)\,
   \mathbb{E}\bigg(
 \int_I 
 \Vert \partial_t\eta^\epsilon 
  \Vert_{W^{s,2}(\Gamma)}^2\dt
  \bigg)^p,
 \end{aligned}
\end{equation} 
where $\delta>0$ is arbitrary.
Lastly, since $\mathscr K_{\eta^\epsilon} $ is spatially independent,
\begin{align*}
\int_II_{10} 
&=-
\int_I\int_{\Gamma}   \partial_t \eta^\epsilon    ((\bsigma\cdot\naby) \Delta_{-h}^s\Delta_{h}^s\eta^\epsilon )
  \dy\dd B_t.
\end{align*}
Therefore, by the Burkholder--Davis--Gundy inequality
\begin{equation}
\begin{aligned}
\mathbb{E}\sup_I\Big\vert\int_I I_{10} \dt\Big\vert^p  &\lesssim
\mathbb{E}\bigg(\int_I\Big(\int_{\Gamma}   \partial_t \eta^\epsilon    ((\bsigma\cdot\naby) \Delta_{-h}^s\Delta_{h}^s\eta^\epsilon )
  \dy\Big)^2\dt\bigg)^\frac{p}{2}
  \\&\lesssim
  \mathbb{E} \bigg(\int_I 
  \Vert \partial_t\eta^\epsilon \Vert_{L^2(\Gamma)}^2
   \Vert \naby \Delta_{-h}^s\Delta_{h}^s\eta^\epsilon \Vert_{L^2(\Gamma)}^2\dt 
   \bigg)^\frac{p}{2}
  \\&\lesssim
  \mathbb{E}\bigg( \Big(\sup_I 
  \Vert \partial_t\eta^\epsilon \Vert_{L^2(\Gamma)}^2
   \Big)^\frac{p}{2} 
   \Big(
   \sup_I\Vert \naby \Delta_{-h}^s\Delta_{h}^s\eta^\epsilon \Vert_{L^2(\Gamma)}^2
   \Big)^\frac{p}{2} 
   \bigg)
 \\&\lesssim
  \mathbb{E} \sup_I 
  \Vert \partial_t\eta^\epsilon \Vert_{L^2(\Gamma)}^{2p}
  +
    \mathbb{E} 
   \sup_I\Vert \eta^\epsilon \Vert_{W^{2,2}(\Gamma)}^{2p},
\end{aligned}
\end{equation}
bounded by
\eqref{GalerkinConvEta1epsp} and \eqref{GalerkinConvEtaepsp}.
If we combine everything, we obtain for $s\in(0,\frac{1}{2})$,
\begin{align}
\mathbb{E}\Big(\int_I\Vert\Delta_h^s\eta^\epsilon \Vert_{W^{s+2,2}(\Gamma )}^2\dt\Big)^p
+&
\epsilon
\mathbb{E}\Big(\int_I\Vert\Delta_h^s\eta^\epsilon \Vert_{W^{3,2}(\Gamma )}^2\dt\Big)^p
\nonumber
\\&+
\epsilon
\mathbb{E}\Big(\sup_I\Vert\Delta_h^s\eta^\epsilon \Vert_{W^{2,2}(\Gamma )}^2\Big)^p
\lesssim 1\label{GalerkinConvEta1NpS}
\end{align}
uniformly in $h$
and thus 
\begin{align}
\mathbb{E}\Big(\int_I\Vert\eta^\epsilon \Vert_{W^{s+2,2}(\Gamma )}^2\dt\Big)^p
+&
\epsilon
\mathbb{E}\Big(\int_I\Vert\eta^\epsilon \Vert_{W^{s+3,2}(\Gamma )}^2\dt\Big)^p
\nonumber
\\&+
\epsilon
\mathbb{E}\Big(\sup_I\Vert\eta^\epsilon \Vert_{W^{s+2,2}(\Gamma )}^2\Big)^p
\lesssim 1\label{GalerkinConvEta1NpS}
\end{align}
for all $p\geq1$.

\subsection{Passage to the limit}
With the bounds from \eqref{GalerkinConvEta1epsp}--\eqref{GalerkinConvEta2epsp} at hand, we wish to obtain compactness. For this, we define the path space
\begin{align*}
\chi_{\tt high}= \chi_{B} 
\times\chi_{{\bu}}\times\chi_{\nabla{\bu}}\times\chi^{\tt high}_{\eta}\times \chi_{f,g}^2
\end{align*}
where
\begin{align*}
\chi_{\eta}^{\tt high}
=&\big(W^{1,\infty}(I;L^{2}(\Gamma)),w^\ast\big)\cap
\big(L^\infty(I;W^{2,2}(\Gamma)),w^\ast\big)\cap \big(L^{2}(I;W^{s+2,2}(\Gamma )),w\big)
\\&\cap \big(W^{1,2}(I;W^{s,2}(\Gamma )),w\big)
\end{align*}
with $s\in(0,\frac{1}{2})$. From \eqref{GalerkinConvEta1epsp}--\eqref{GalerkinConvEta2epsp}, Lemma \ref{lemma:fg''} and \eqref{GalerkinConvEta1NpS} we obtain similarly to Proposition \ref{prop:skorokhod'}:
%\todo{better use $\epsilon_n$ to get a sequence?}

\begin{proposition}\label{prop:skorokhod''}
There exists a complete probability space $(\tilde\Omega,\tilde{\mathfrak F},\tilde{\mathbb P})$ 
with $\chi$-valued random variables
\begin{align*}
\tilde{\boldsymbol{\Theta}}^{\epsilon_n}:=\bigg[\tilde B_t^n,&\mathbb I_{\mathcal O_{(\tilde\eta^{\epsilon_n})_\epsilon}}\tilde{\mathbf u}^{\epsilon_n},
\mathbb I_{\mathcal O_{(\tilde\eta^{\epsilon_n})_\epsilon}}\nabx\tilde{\mathbf u}^{\epsilon_n}, \tilde\eta^{\epsilon_n} ,
\\&\begin{pmatrix}
((\partial_t\tilde{\eta}^{\epsilon_n},\mathbb I_{\mathcal O_{(\tilde{\eta}^{\epsilon_n})_\epsilon}}\tilde{\mathbf u}^{\epsilon_n}),
(\partial_t\tilde\eta^{\epsilon_n},\mathscr F^{(\tilde\eta^{\epsilon_n})_\epsilon} \partial_t\tilde\eta^{\epsilon_n}))
\\
((\partial_t\tilde\eta^{\epsilon_n},\mathbb I_{\mathcal O_{(\tilde\eta^{\epsilon_n})_\epsilon}}\tilde{\mathbf u}^{\epsilon_n}),(0,\tilde{\mathbf u}^{\epsilon_n}
-\mathscr F^{(\tilde\eta^{\epsilon_n})_\epsilon} \partial_t\tilde\eta^{\epsilon_n}))\end{pmatrix}\bigg],
\end{align*}
for $n\in\mathbb N$ and
$$\tilde{\boldsymbol{\Theta}}:=\bigg[\tilde B_t,\mathbb I_{\mathcal O_{\tilde\eta}}\tilde{\mathbf u},
\mathbb I_{\mathcal O_{\tilde\eta}}\nabx \tilde{\mathbf u}, \tilde\eta,
\begin{pmatrix}
((\partial_t\tilde\eta,\mathbb I_{\mathcal O_{\tilde\eta}}\tilde{\mathbf u}),
(\partial_t\tilde\eta,\mathscr F^{\tilde\eta}\partial_t\tilde\eta))
\\
((0,\mathbb I_{\mathcal O_{\tilde\eta}}\tilde{\mathbf u}),(\partial_t\tilde\eta,\tilde{\mathbf u}
-\mathscr F^{\tilde\eta}\partial_t\tilde\eta))\end{pmatrix}\bigg]$$
such that
\begin{enumerate}
 \item[(a)] For all $n\in \mathbb N$ the law of 
$\tilde{\boldsymbol{\Theta}}^{\epsilon_n}$
on $\chi$ is given by
{\small
\begin{align*}
 \mathcal{L}\bigg[B_t^{n}&,\mathbb I_{\mathcal O_{(\eta^{\epsilon_n})_{\epsilon_n}}}\mathbf u^{\epsilon_n},
\mathbb I_{\mathcal O_{(\eta^{\epsilon_n})_{\epsilon_n}}}\nabx\mathbf u^{\epsilon_n}, \eta^{\epsilon_n} ,
\\&\begin{pmatrix}
((\partial_t\eta^{\epsilon_n},\mathbb I_{\mathcal O_{(\eta^{\epsilon_n})_{\epsilon_n}}}\mathbf u^{\epsilon_n}),
(\partial_t\eta^{\epsilon_n},\mathscr F^{(\eta^{\epsilon_n})_{\epsilon_n}} \partial_t\eta^{\epsilon_n}))
\\
((0,\mathbb I_{\mathcal O_{(\eta^{\epsilon_n})_{\epsilon_n}}}\mathbf u^{\epsilon_n}),(\partial_t\eta^{\epsilon_n},\mathbf u^{\epsilon_n}-\mathscr F^{(\eta^{\epsilon_n})_{\epsilon_n}} \partial_t\eta^{\epsilon_n}))\end{pmatrix}\bigg]
\end{align*}
}
 \item[(b)] $\tilde{\boldsymbol{\Theta}}^{\epsilon_n}$ converges $\,\tilde{\mathbb P}$-almost surely to 
$\tilde{\boldsymbol{\Theta}}$
 in the topology of $\chi_{\tt high}$, i.e.
\begin{align} \label{wWS116}
\begin{aligned}
\tilde B_t^{n}&\rightarrow \tilde B_t \quad \mbox{in}\quad C(\overline I) \ \tilde{\mathbb P}\mbox{-a.s.}, 
\\
\mathbb I_{\mathcal O_{(\tilde\eta^{\epsilon_n} )_\epsilon}}\tilde{\mathbf u}^{\epsilon_n}&\weakto \mathbb I_{\mathcal O_{\tilde\eta}}\tilde{\mathbf u} \quad \mbox{in}\quad L^2(I;L^{2}(\mathcal O\cup S_{\alpha}))\ \tilde{\mathbb P}\mbox{-a.s.}, 
\\
\mathbb I_{\mathcal O_{(\eta^{\epsilon_n})_\epsilon}}\nabx\tilde{\mathbf u}^{\epsilon_n}&\weakto \mathbb I_{\mathcal O_{\tilde\eta}}\nabx\tilde{\mathbf u} \quad \mbox{in}\quad L^2(I;L^{2}(\mathcal O\cup S_{\alpha}))\ \tilde{\mathbb P}\mbox{-a.s.}, 
\\
\tilde \eta^{\epsilon_n}&\weakto^\ast \tilde \eta \quad \mbox{in}\quad L^\infty(I;W^{2,2}(\Gamma)) \ \tilde{\mathbb P}\mbox{-a.s.}, 
\\
\tilde \eta^{\epsilon_n}&\weakto^\ast \tilde \eta\quad \mbox{in}\quad W^{1,\infty}(I;L^{2}(\Gamma)) \ \tilde{\mathbb P}\mbox{-a.s.}, 
\\
\tilde \eta^{\epsilon_n}&\weakto \tilde \eta \quad \mbox{in}\quad W^{1,2}(I;W^{s,2}(\Gamma)) \ \tilde{\mathbb P}\mbox{-a.s.}, 
\\
\tilde \eta^{\epsilon_n}&\weakto \tilde \eta \quad \mbox{in}\quad L^{2}(I;W^{s+2,2}(\Gamma)) \ \tilde{\mathbb P}\mbox{-a.s.}, 
\end{aligned}
\end{align}
as well as (recalling the definition of $\tau_\sharp$ from \eqref{eq:tau})
\begin{align}\label{wWS116B}\begin{aligned}
\int_{I}\int_{\mathcal O\cup S_{\alpha}}&\mathbb I_{\mathcal O_{(\tilde\eta^{\epsilon_n})_\epsilon}}
\bfu^{\epsilon_n}\cdot\mathscr F^{(\tilde\eta^{\epsilon_n})_\epsilon}\partial_t\tilde\eta^{\epsilon_n}\dxt+\int_{I}\int_\Gamma|\partial_t\tilde\eta^{\epsilon_n}|^2\,\dy \dt
\\
&\longrightarrow \int_{I}\int_{\mathcal O\cup S_{\alpha}}\mathbb I_{\mathcal O_{\tilde\eta}}
\tilde\bfu\cdot\mathscr F^{\tilde\eta}\partial_t\eta\dxt+\int_{I}\int_\Gamma|\partial_t\tilde\eta|^2\,\dy \dt
\end{aligned}
\end{align}
and
\begin{align}\label{wWS116C}\begin{aligned}
\int_{I}\int_{\mathcal O\cup S_{\alpha}}&
\mathbb I_{\mathcal O_{(\tilde\eta^{\epsilon_n})_\epsilon}}
\tilde\bfu^{\epsilon_n}\cdot(\tilde\bfu^{\epsilon_n}-\mathscr F^{(\tilde\eta^{\epsilon_n})_\epsilon}\partial_t\tilde\eta^{\epsilon_n})\dxt
\\
&\longrightarrow \int_{I}\int_{\mathcal O\cup S_{\alpha}}
\mathbb I_{\mathcal O_{\tilde\eta}}
\tilde\bfu\cdot(\tilde\bfu-\mathscr F^{\tilde\eta}\partial_t\tilde\eta)\dxt
\end{aligned}
\end{align}
$\tilde{\mathbb P}$-a.s.
\end{enumerate}
\end{proposition}
%Now we introduce the filtration on the new probability space, 
%which ensures the correct measurabilities of the new random variables.
%We denote by $\bfr_t$ the operator of restriction to the interval $[0,t]$ acting on various path spaces. In particular, if $X$ stands for one of the path spaces $L^r_{\mathrm loc}(0,\infty;L^r(\mathcal{O}))$ or $C_{\mathrm loc}([0,\infty),\mathfrak U_0)$ and $t\in[0,T]$, we define
%\begin{align}\label{restr}
%\bfr_t:X\rightarrow X|_{[0,t]},\quad f\mapsto f|_{[0,t]}.
%\end{align}
%Clearly, $ \bfr_t$ is a continuous mapping.
%Let $(\tilde{\mathfrak F}_t)_{t\geq0}$ and $(\tilde{\mathfrak F}_t^\epsilon)_{t\geq0}$ be the $\tilde{\mathbb P}$-augmented
%canonical filtration of the variables $\tilde{\boldsymbol{\Theta}}$ and
%$\tilde{\boldsymbol{\Theta}}^\epsilon $, respectively, that is
%\todo{add filtrations}

Similar to the analysis after Propositions \ref{prop:skorokhodN} and \ref{prop:skorokhod'}, we can define the $\tilde{\mathbb P}$-augmented
canonical filtrations $(\tilde{\mathfrak F}_t)_{t\geq0}$ and $(\tilde{\mathfrak F}_t^\epsilon)_{t\geq0}$ on the variables $\tilde{\boldsymbol{\Theta}}$ and
$\tilde{\boldsymbol{\Theta}}^\epsilon $, respectively, which ensures the correct measurabilities of the new random variables.

By \cite[Theorem 2.9.1]{BFHbook} the weak equation continuous to hold on the new probability space.
Combining \eqref{wWS116B} and \eqref{wWS116C} we have
\begin{align*}
\int_{I}\int_{\mathcal O_{(\tilde\eta^{\epsilon_n})_\epsilon}}&|\tilde\bfu^{\epsilon_n}|^2\dxt
+\int_{I}\int_\Gamma|\partial_t\tilde\eta^{\epsilon_n}|^2\,\dy \dt\\&\longrightarrow \int_{I}\int_{\mathcal O_{\tilde\eta}}|\tilde\bfu|^2\dxt
+\int_{I}\int_\Gamma|\tilde\partial_t\eta|^2\,\dy \dt
\end{align*}
 $\tilde{\mathbb P}$-a.s. By uniform convexity of the $L^2$-norm this implies
\begin{align*}
\tilde\eta^{\epsilon_n}&\rightarrow\tilde\eta\quad \mbox{in}\quad W^{1,2}(I;L^2(\Gamma)) ,\\
\mathbb I_{\mathcal O_{(\tilde\eta^{\epsilon_n})_\epsilon}}\tilde{\mathbf u}^{\epsilon_n}&\rightarrow \mathbb I_{\mathcal O_{\tilde\eta}}\tilde{\mathbf u}
 \quad \mbox{in}\quad L^2(I;L^{2}(\mathcal O\cup S_{\alpha}))\ \tilde{\mathbb P}\mbox{-a.s.}.
\end{align*}
This is sufficient to pass to the limit and the weak formulation of the equations (note that all terms except for the convective term can be treated 
by \eqref{wWS116}).
As far as the energy balance is concerned, this is even easier since there is no noise. By \eqref{wWS116} and the lower semi-continuity of the involved quantities
we can pass to the limit (obtaining only an inequality).

\section*{Declarations}
%
%Some journals require declarations to be submitted in a standardised format. Please check the Instructions for Authors of the journal to which you are submitting to see if you need to complete this section. If yes, your manuscript must contain the following sections under the heading `Declarations':
%
%\begin{itemize}
%\item Funding
%\item Conflict of interest/Competing interests (check journal-specific guidelines for which heading to use)
%\item Ethics approval 
%\item Consent to participate
%\item Consent for publication
%\item Availability of data and materials
%\item Code availability 
%\item Authors' contributions
%\end{itemize}
%
%\noindent
%If any of the sections are not relevant to your manuscript, please include the heading and write `Not applicable' for that section. 
%%%%%%%%%%%%%%%%%%%%

\smallskip
\par\noindent
{\bf Funding}. Not applicable.

\smallskip
\par\noindent
{\bf Conflict of Interest}. The authors declares that they have no conflict of interest.

\smallskip
\par\noindent
{\bf Ethics approval}. Not applicable.

\smallskip
\par\noindent
{\bf Consent to participate}. Not applicable.

\smallskip
\par\noindent
{\bf Consent for publication}. Not applicable.

\smallskip
\par\noindent
{\bf Availability of data and materials}. Data sharing is not applicable to this article as no datasets were generated or analysed during the current study.

\smallskip
\par\noindent
{\bf Code availability}. Not applicable.

\smallskip
\par\noindent
{\bf Authors' contributions}. All authors contributed to the writing and review of this manuscript.

\end{document}